\documentclass{amsart}
\usepackage{graphicx,amsfonts,amssymb,amsmath,amsthm,url,amscd}

\theoremstyle{plain} %--default
\newtheorem{theorem}             {Theorem}  [section]
\newtheorem{lemma}      [theorem]{Lemma}
\newtheorem{corollary}  [theorem]{Corollary}
\newtheorem{proposition}[theorem]{Proposition}

\theoremstyle{definition}
\newtheorem{definition} [theorem]{Definition}

\theoremstyle{remark}
\newtheorem{remark} [theorem]    {Remark}

\def\sinc{\operatorname{sinc}}
\def\Cl{\operatorname{Cl}}
\def\sgn{\operatorname{sgn}}

\def\res{\operatorname{res}}
\def\SO{\operatorname{SO}}
\def\vol{\operatorname{vol}}
\def\fin{\operatorname{fin}}
\def\div{\operatorname{div}}
\def\Tr{\operatorname{Tr}}
\def\Hom{\operatorname{Hom}}
\def\ad{\operatorname{ad}}

\def\PGL{\operatorname{PGL}}
\def\SL{\operatorname{SL}}
\def\GL{\operatorname{GL}}
\def\norm{\operatorname{N}}
\def\onehalf{\tfrac{1}{2}}
\def\vecgamma{{\mathbf{\Gamma}}}
\renewcommand{\Re}{\mathrm{Re}}

\def\eps{\varepsilon}

\newcommand{\efrac}[2]{\genfrac{}{}{0pt}{}{#1}{#2}}

\begin{document}

\title{Mass equidistribution of Hilbert modular eigenforms}
\author{Paul D. Nelson}
\begin{abstract}
  Let $\mathbb{F}$ be a totally real number field, and let $f$
  traverse a sequence of nondihedral holomorphic eigencuspforms
  on $\GL_2/\mathbb{F}$ of weight
  $(k_1,\dotsc,k_{[\mathbb{F}:\mathbb{Q}]})$, trivial central
  character and full level.  We show that the mass of $f$
  equidistributes on the Hilbert modular variety as
  $\max(k_1,\dotsc,k_{[\mathbb{F}:\mathbb{Q}]}) \rightarrow
  \infty$.

  Our result answers affirmatively a natural analogue of a
  conjecture of Rudnick and Sarnak (1994).  Our proof
  generalizes the argument of Holowinsky-Soundararajan (2008)
  who established the case $\mathbb{F} = \mathbb{Q}$.  The
  essential difficulty in doing so is to adapt Holowinsky's
  bounds for the Weyl periods of the equidistribution problem in
  terms of manageable shifted convolution sums of Fourier
  coefficients to the case of a number field with nontrivial
  unit group.
\end{abstract}
\maketitle

\setcounter{tocdepth}{1}
\tableofcontents
\section{Introduction}

\subsection{Statement of main result}\label{sec:stat-main-result}
Let $\mathbb{F}$ be a totally real number field and $f$ a
holomorphic Hilbert modular eigencuspform on $\PGL_2/\mathbb{F}$
of weight $k = (k_1,\dotsc,k_{[\mathbb{F}:\mathbb{Q}]})$ and
full level.  The mass $|f|^2$ descends to a finite measure on
the Hilbert modular variety; our aim in this paper is to prove
that the measures so obtained equidistribute with respect to the
uniform measure as the weight $k$ of $f$ tends to $\infty$.
Motivation for this problem, as discussed in \S\ref{sec:motivation}, comes from its connection to quantum chaos
by analogy with the \emph{quantum unique ergodicity} conjecture of
Rudnick and Sarnak \cite{MR1266075} as well as from its
connection to central problems in the analytic theory of
$L$-functions, specifically those such as the subconvexity
problem that concern the rate of growth of central $L$-values.
Our result and its method of proof directly generalize recent
work of Holowinsky and Soundararajan
\cite{holowinsky-soundararajan-2008} in the case $\mathbb{F} =
\mathbb{Q}$, but the generalization is not immediate.

To state our principal result, let $\mathbb{A}$ be the adele ring of
$\mathbb{F}$ and $K$ a maximal compact subgroup of the group
$\PGL_2(\mathbb{A})$.
The space $Y =
\PGL_2(\mathbb{F}) \backslash \PGL_2(\mathbb{A}) / K$
is a disjoint union (indexed by a quotient of the narrow
class group of $\mathbb{F}$)
of finite-volume non-compact complex manifolds
of dimension $[\mathbb{F}:\mathbb{Q}]$.
Let $\mu$ be the
quotient
measure on $Y$ induced by a fixed Haar measure on
$\PGL_2(\mathbb{A})/K$.
\begin{theorem}\label{thm:1}
  Let $f : \PGL_2(\mathbb{A}) \rightarrow \mathbb{C}$ traverse a
  sequence of nondihedral holomorphic eigencuspforms of
  weight $(k_1,\dotsc,k_{[\mathbb{F}:\mathbb{Q}]})$ as above, so
  that $|f|^2 \, d \mu$ traverses a sequence of measures on $Y$.
  Fix a compactly supported function $\phi \in C_c(Y)$.  Then
  \begin{equation}\label{eq:1}
    \frac{\int \phi |f|^2 \, d \mu}{\int |f|^2 \, d \mu}
    \rightarrow \frac{\int \phi \, d \mu}{\int \, d \mu }
    \quad \text{  as } \max(k_1,\dotsc,k_{[\mathbb{F}:\mathbb{Q}]}) \rightarrow \infty.
  \end{equation}
\end{theorem}
In words, the measures $|f|^2 \, d \mu$ \emph{equidistribute} as
any one of the weight components $k_i$ tend to $\infty$.  We
could normalize $d \mu$ and $|f|^2 \, d \mu$ to be probability
measures, in which case Theorem \ref{thm:1} asserts that $|f|^2
\, d \mu$ \emph{converges weakly} to $d \mu$.  Theorem
\ref{thm:1} is false for certain\footnote{those induced from
  idele class characters on unramified totally imaginary
  quadratic extensions of $\mathbb{F}$; see \S\ref{sec:holom-eigenc}}
dihedral forms $f$ that
vanish identically on half of the connected components of $Y$;
in that case, the analogous assertion that $|f|^2$
equidistributes as $\max(k_1,\dotsc,k_{[\mathbb{F}:\mathbb{Q}]}) \rightarrow \infty$ on
the union of the remaining connected components of $Y$ remains
true, but to simplify the exposition we shall consider only
nondihedral forms in this paper.

The case $\mathbb{F} = \mathbb{Q}$ of Theorem \ref{thm:1} is the
celebrated theorem of Holowinsky-Soundararajan
\cite{holowinsky-soundararajan-2008}, who established a
quantitative rate of convergence in the limit (\ref{eq:1}) for a
``spanning set'' of functions $\phi$ (see
\S\ref{sec:brief-revi-holow}).  Marshall
\cite{2010arXiv1006.3305M} proved a generalization of their
result to cohomological forms over general number fields
$\mathbb{F}$ that satisfy the Ramanujan conjecture, under the
mild technical assumptions that $\mathbb{F}$ have narrow class
number one and that the weights $k_i$ (or the analogous
archimedean parameters for fields $\mathbb{F}$ with complex
places) all tend to infinity together with sufficient
uniformity, precisely that
$\min(k_1,\dotsc,k_{[\mathbb{F}:\mathbb{Q}]}) \rightarrow
\infty$ with $\min(k_1,\dotsc,k_{[\mathbb{F}:\mathbb{Q}]}) \geq
(k_1 \dotsb k_{[\mathbb{F}:\mathbb{Q}]})^{\eta}$ for some fixed
$\eta > 0$.  Since cohomological forms over
totally real and imaginary quadratic number fields are known to
satisfy the Ramanujan conjectures, his results are unconditional
in many cases and overlap\footnote{We proved a slightly weaker form
of Theorem \ref{thm:1}
in September 2009 and learned soon thereafter from Sarnak's
lecture notes \cite{sarnak-progress-que} that the overlapping
results just described had been obtained earlier that year in the 2009/2010 Princeton
PhD thesis of his student S. Marshall \cite{2010arXiv1006.3305M}.  We hope that our own
arguments differ sufficiently to be of interest.} with ours when $\mathbb{F}$ is totally
real of narrow class number one and the weights grow uniformly
in the sense just described.
The essential difference between our approaches is explained in
remark \ref{rmk:compare-marshall}.

An important ingredient in Holowinsky's contribution to proof of
Theorem \ref{thm:1} when $\mathbb{F} = \mathbb{Q}$ is his bound
\begin{equation}\label{eq:35}
  \sum_{n \leq x}
  \lambda(n) \lambda(n+l)
  \ll_\eps \tau(l)
  x \log(x)^\eps \prod_{p \leq x}
  \left(
    1 + \frac{\lambda(p) -1}{p}
  \right)^2
\end{equation}
for any multiplicative function $\lambda : \mathbb{N}
\rightarrow \mathbb{R}_{\geq 0}$ satisfying $\lambda(n) \leq
\tau_m(n)$ for some positive integer $m$ and any ``shift'' $l$
satisfying $0 \neq  |l| \leq x$ (see \S\ref{sec:holow-indep-argum}).
A generalization of \eqref{eq:35} to number fields features in
Marshall's work mentioned above.  We independently generalize
\eqref{eq:35} to number fields that are totally real, although
this restriction is not essential.  The bounds that we obtain
are stronger than those obtained by Holowinsky and Marshall in
that we have removed the factor $\tau(l)$ appearing on the RHS
of \eqref{eq:35} and its generalizations (see Theorem
\ref{thm:essential-sums} and Theorem
\ref{thm:shifted-sums-boxes}).  Although doing so is not
necessary for our present purposes, this refinement has
applications to the study of the distribution of mass of
holomorphic forms of large level \cite{PDN-HQUE-LEVEL}.

\subsection{Motivation}\label{sec:motivation}
The study of the limiting behavior of the masses of Hilbert
modular eigencuspforms is natural and interesting from several
perspectives of which we highlight two.  First, it is analogous
to a fundamental problem in quantum chaos, which concerns more
generally the limiting behavior as $\lambda \rightarrow \infty$
of eigenfunctions $\phi$
\begin{equation}\label{eq:23}
  (\Delta + \lambda ) \phi  = 0
\end{equation}
of the Laplacian $\Delta$ on a compact Riemannian manifold $M$
for which the geodesic flow is chaotic (see \cite{MR1321639}).  Here the geodesic flow on $M$ is regarded as
the Hamiltonian flow of a chaotic classical mechanical system, the
Laplacian $\Delta \circlearrowright L^2(M)$ as the Hamiltonian
operator for the corresponding quantized system, and the
eigenfunction $\phi$ (normalized so that $\int |\phi|^2 = 1$) as
the wave function for a quantum particle on $M$ of energy
$\lambda$ whose position is described in the Copenhagen
interpretation of quantum mechanics by the probability density
$|\phi|^2$.
In suitable units the Schr\"{o}dinger equation
for stationary states reads
$(\hbar^2 \Delta + \lambda)\phi = 0$,
so studying
$\phi$ in \eqref{eq:23}
as $\lambda \rightarrow \infty$
is akin to considering the semiclassical limit
$\hbar \rightarrow 0$ of the quantization
of the geodesic flow.

Among several questions that one can ask we single out that of
the behavior of the densities $|\phi|^2$ for particles of high
energy $\lambda \rightarrow \infty$.
A fundamental result in
this direction is the \emph{quantum ergodicity} theorem of
Schnirelman, Colin d{e} Verdi\`{e}re, and Zelditch
\cite{MR0402834,MR819779,MR916129}, which asserts that if the
geodesic flow on the unit cotangent bundle of $M$ is ergodic,
then for any sequence $(\phi_n)$ with $\lambda_n \rightarrow
\infty$ there exists a full-density subsequence $(\phi_{n_k})$
such that the $|\phi_{n_k}|^2$ equidistribute.\footnote{in a more
  precise sense than we describe here; see the introduction to
  \cite{MR1266075}}
In the particular case that $M$ is
negatively curved, the \emph{quantum unique ergodicity} (QUE)
conjecture of Rudnick and Sarnak \cite{MR1266075} predicts that
the \emph{full} sequence of $|\phi_n|^2$ equidistributes with respect to the
volume measure on $M$ as $\lambda \rightarrow \infty$.

The QUE conjecture is considered difficult and there
has been little progress for general $M$,
but for certain special $M$
that arise from arithmetic considerations (such as the modular
curve or the Hilbert modular varieties) there has been
significant progress on QUE and related questions
\cite{sarnak-progress-que,MR1361757,MR2195133,2009arXiv0901.4060S,MR2346281,holowinsky-soundararajan-2008}.  Such
arithmetic manifolds arise as quotients of symmetric spaces by
arithmetic groups and are characterized by the presence of
additional symmetry in the form of a large commuting family
$\mathbb{T}$ of correspondences that commute with the 
algebra $\mathcal{D}$ of invariant differential operators,
thereby providing a powerful tool for the study of common
eigenfunctions of $\mathbb{T}$ and $\mathcal{D}$.  One may hope
that such arithmetic instances of QUE provide tractable
and yet representative model cases for
the more general problem (see \cite{MR1321639}).

The variant of QUE that we consider for holomorphic Hilbert
modular eigencuspforms $f$ of increasing weight is in the spirit
of the original conjectures and was spelled out explicitly for
the modular curve ($\mathbb{F} = \mathbb{Q}$) by Luo and Sarnak
\cite{luo-sarnak-mass}; it is important here that $f$ is taken
to be an eigenform (of the Hecke algebra), since for instance
the powers of a fixed form have weight tending to infinity but
do not have equidistributed mass.

A second motivation for our considerations arises from their
connection to central problems in the analytic theory of
$L$-functions.  Watson \cite{watson-2008} showed that for $M =
\SL(2,\mathbb{Z}) \backslash \mathbb{H}$ (as well as other
``arithmetic surfaces'' $\Gamma \backslash \mathbb{H}$), the
Weyl periods for the equidistribution problem posed by QUE are
essentially products of central values $L(\onehalf)$ of
automorphic $L$-functions $L(s)$ of degree at most $6$; a
similar relation holds over totally real fields (see \S
\ref{sec:sound-indep-argum}).  The generalized Riemann hypothesis
(GRH) for such $L(s)$, which asserts that the nontrivial zeros
of $L(s)$ lie on the line $\Re(s) = \tfrac{1}{2}$, would
imply
sufficiently strong bounds on $L(\onehalf)$ to establish the QUE
conjecture for $M$.  But the bounds on $L(\onehalf)$ demanded by
QUE are considerably more tractable than those implied by the
GRH (let alone the GRH itself), and so provide
accessible problems on which to develop new techniques.

\subsection{Overview of proof}\label{sec:overview-proof}
Recall that we consider nondihedral holomorphic Hilbert modular
eigencuspforms $f$ on $\PGL_2/\mathbb{F}$ of weight
$(k_1,\dotsc,k_{[\mathbb{F}:\mathbb{Q}]})$ and full level, the
equidistribution of whose mass we seek on the (in general,
non-connected) Hilbert modular variety $Y$.  The basic strategy,
as in many equidistribution problems,
is to study the
``Weyl periods'' $\int \phi \lvert f \rvert ^2 $ as $\phi$ traverses a
\emph{convenient spanning set} of functions on $Y$,
analogous to how one uses the exponentials $\mathbb{R} / \mathbb{Z}
\ni x \mapsto e^{2
  \pi i n x}$ to prove the equidistribution of the fractional
parts
of $\alpha k$ ($k \in \mathbb{N}$) for $\alpha \in \mathbb{R} -
\mathbb{Q}$.

Indeed,
Theorem
\ref{thm:1} follows as soon as one can establish (\ref{eq:1})
for each element $\phi$ of a set the uniform closure of whose
span contains $C_c(Y)$.  Such a spanning set is furnished by the
Maass eigencuspforms and the incomplete Eisenstein series, as
defined in \S\ref{sec:automorphic-forms}.
To highlight the essential difficulties let us suppose in this
section
that $\phi$ is a Maass eigencuspform.
Then $\int \phi = 0$, so to establish (\ref{eq:1})
we must show that
\begin{equation}\label{eq:4}
  \frac{\int \phi \lvert f  \rvert ^2 }{ \int \lvert f  \rvert
    ^2 }
  \rightarrow 0
  \quad \text{ as } \max(k_1,\dotsc,k_{[\mathbb{F}:\mathbb{Q}]})
  \rightarrow \infty,
\end{equation}
where the rate of convergence
is allowed to depend upon $\phi$.

Take $\mathbb{F} = \mathbb{Q}$
and $f$ of weight $k$ for now.  Holowinsky and
Soundararajan established (\ref{eq:4}) by a remarkable synthesis
of their independent efforts \cite{holowinsky-2008,MR2680497}, which we now recall briefly, saving a
more detailed discussion for \S\ref{sec:brief-revi-holow}
and referring to the lucid expositions of
\cite{holowinsky-soundararajan-2008,sarnak-progress-que,aws2010sound} for further motivation and details.  Watson's
formula \cite{watson-2008} and work of Gelbart-Jacquet
\cite{MR546600} and Hoffstein-Lockhart-Goldfeld-Lieman
\cite{HL94} imply (see \cite[Lem 2]{holowinsky-soundararajan-2008}) that 
\begin{equation}\label{eq:7}
  \frac{\int \phi \lvert f  \rvert ^2 }{\int \lvert f \rvert ^2
  }
  \approx_\phi
  \frac{\lvert L (\phi \times \ad f, \tfrac{1}{2} ) \rvert
    ^{1/2}}{
    k ^{1/2}
  }
  \exp
  \left(-\sum_{p \leq k} \frac{1}{p} \lambda(p^2) \right),
\end{equation}
where $L(\cdot)$ denotes the finite part of the $L$-function
indicated above, $\approx_\phi$ denotes equality up to
multiplication
by a
bounded power of $\log \log(k)$ times a constant depending
upon $\phi$,
and $\lambda(n)$ is the $n$th Fourier coefficient of $f$
normalized so that the Deligne bound reads $|\lambda(p)| \leq
2$.
Soundararajan proves a ``weak
subconvexity''
bound for the central values of quite general
$L$-functions satisfying a ``weak Ramanujan hypothesis,''
specializing in the present circumstances to
$\lvert L (\phi \times \ad f, \tfrac{1}{2} ) \rvert
\ll k/\log(k)^{1-\eps}$
for any $\eps > 0$, which implies (\ref{eq:4}) provided that
\begin{equation}\label{eq:siegel-bad}
  \frac{
    \sum_{p \leq k} \tfrac{1}{p} \lambda(p^2)
  }{
    \sum_{p \leq k} \tfrac{1}{p}
  }
  \geq  -1/2 + \delta + o_{k \rightarrow \infty}(1)
  \quad \text{ for some fixed } \delta > 0.
\end{equation}
By considering Fourier expansions
at the cusps of the modular curve
and bounding the sums
(described below in more detail)
that arise,
Holowinsky
proves (following the
reformulation
of Iwaniec \cite{IwaniecQUENotes})
\begin{equation}\label{eq:holow-bound-iw}
  \frac{\int \phi |f|^2}{|f|^2}
  \ll_{\phi,\eps} \log(k)^\eps
  \exp \left( - \sum_{p \leq k}
    \frac{1}{p} (|\lambda(p)| - 1 )^2
  \right),
\end{equation}
which implies \eqref{eq:4} provided that
\begin{equation}
  \frac{
    \sum_{p \leq k}
    \tfrac{1}{p}
    (|\lambda(p)| - 1)^2
  }{
    \sum_{p \leq k}
    \tfrac{1}{p}
  }
  \geq \delta + o_{k \rightarrow \infty}(1)
  \quad \text{ for some fixed } \delta > 0.
\end{equation}
In summary, Soundararajan succeeds
unless typically $\lambda(p^2) \lessapprox -1/2$,
while Holowinsky succeeds
unless typically $|\lambda(p)| \approx 1$
(in the harmonically weighted
sense taken over $p \leq k$);
the identity $\lambda(p)^2 = \lambda(p^2) + 1$
shows that
\[
\lambda(p^2) \lessapprox -1/2
\implies |\lambda(p)| \lessapprox \sqrt{1/2}
\quad \text{ and }
\quad 
|\lambda(p)| \approx 1 \implies
\lambda(p^2) \approx 0,
\]
so in all cases at least
one of their approaches succeeds.

The basic ideas underlying our proof when $\mathbb{F}$ is
totally real are the same as those just described in the case
$\mathbb{F} = \mathbb{Q}$; the generalization is a nontrivial
and yet purely technical matter, requiring no fundamental
reworking of the overall strategy.  As we shall explain in
\S\ref{sec:brief-revi-holow}, the only part of the $\mathbb{F} = \mathbb{Q} $
argument that does not generalize transparently is Holowinsky's
proof of \eqref{eq:holow-bound-iw}.
His argument amounts to
\begin{enumerate}
\item bounding $\int \phi |f|^2 / \int |f|^2$
  from above
  in terms of the ``shifted sums''
  \begin{equation}\label{eq:holow-shifted-sums-intro}
    X^{-1} \sum_{n \in \mathbb{Z} \cap [1,X]}^{\text{smooth}}
    \lambda(n) \lambda(n+l),
  \end{equation}
  where $l \neq 0$ is a small integer and $X \approx k$,
  and
\item bounding the shifted sums
  \eqref{eq:holow-shifted-sums-intro};
  a reformulation \cite{IwaniecQUENotes} of the bound that
  Holowinsky
  obtains is
  \begin{equation}\label{eq:holow-bound-shifted-sums-intro}
    X^{-1} \sum_{n,n+l \in \mathbb{Z} \cap [1,X]}
    |\lambda(n) \lambda(n+l)|
    \ll
    \tau(l)
    \log(k)^\eps
    \prod_{p \leq k}
    \left( 1 + \frac{2 (|\lambda(p)|-1)}{p} \right),
  \end{equation}
  which is roughly the square of the bound one would expect for
  $X^{-1} \sum |\lambda(n)|$ and so may be understood as
  asserting the independence of the random variables $n \mapsto
  |\lambda(n)|$, $n \mapsto |\lambda(n+l)|$ owing to the
  independence of the prime factorizations of $n$ and $n+l$ and
  the multiplicativity of $\lambda$.  The novelty in his
  argument is that he does not exploit cancellation in the sums
  \eqref{eq:holow-shifted-sums-intro} that one would expect to
  arise from the independent variation in sign of $\lambda(n)$
  and $\lambda(n+l)$ for varying $n$ and fixed $l \neq 0$; his
  motivation for doing so came from the expectation that the
  $\lambda(p)$ follow the Sato-Tate distribution, which suggests
  that $X^{-1} \sum |\lambda(n)| \ll \log(X)^{-\delta}$ for some
  small $\delta > 0$.  See
  \cite{luo-sarnak-mass,holowinsky-soundararajan-2008,
    sarnak-progress-que,aws2010sound} and especially
  \cite{MR2484279} for further discussion.
\end{enumerate}
Now let $[\mathbb{F}:\mathbb{Q}] = d$
and take $f$ of weight $(k_1,\dotsc,k_d)$.
The most na\"{i}ve higher-dimensional generalization
of Holowinsky's method that we found requires
one to replace $X$ and $\mathbb{Z} \cap [1,X]$ in
\eqref{eq:holow-shifted-sums-intro}
by $X \approx k_1 \dotsb k_{d}$
and $\mathfrak{o} \cap \mathcal{R}$, where $\mathfrak{o}$ is
the ring of integers in $\mathbb{F}$
and $\mathcal{R}$ is the region
in the totally positive quadrant of $\mathbb{F} \otimes_\mathbb{Q}
\mathbb{R}
\cong \mathbb{R}^{d}$
bounded by the hyperbola $\{ x_1 \dotsb x_{d}
= X \}$
and the hyperplanes $\{x_i = c\}$ for some small constant $c >
0$.
Unfortunately, the volume of $\mathcal{R}$ is roughly
$X \log(X)^{d-1}$,
so even the most optimistic bounds
along the lines of \eqref{eq:holow-bound-shifted-sums-intro}
fail to produce an estimate of the quality
\eqref{eq:holow-bound-iw}
because of the unaffordable factor
$\log(X)^{d-1}$
when $d > 1$.

To circumvent this difficulty, we refine Holowinsky's upper
bound for $\int \phi |f|^2$ by a method
that when $\mathbb{F} = \mathbb{Q}$ leads
(see remark \ref{rmk:super-refined})
to the precise
asymptotic expansion
\begin{equation}\label{eq:refined-holowinsky}
  \frac{\int \phi |f|^2}{\int |f|^2}
  \sim \frac{(Y k)^{-1}}{L(\ad f, 1)}
  \mathop{\sum \sum}_{
    \substack{
      m=n+l \\
      \max(m,n) \asymp Y k
    }
  }
  \frac{\lambda_\phi(l) }{\sqrt{|l|}}
  \lambda_f(m) \lambda_f(n)
  \kappa_{\phi,\infty}
  \left( \frac{k - 1}{4 \pi }
    \left\lvert \log \frac{m}{n} \right\rvert
  \right),
\end{equation}
where $Y \geq 1$ tends slowly to infinity with $k$,
$\lambda_\phi$, and $\lambda_f$ are the normalized Fourier
coefficients of $\phi$ and $f$ respectively, $\kappa_{\phi,\infty}(y) = 2
y^{1/2} K_{i r}(2 \pi y)$ for $y > 0$ if $\tfrac{1}{4}+r^2$ is
the Laplace eigenvalue of $\phi$, and the sum is taken over
triples $(l,m,n) \in \mathbb{Z}^3$ for which $0 \neq |l| <
Y^{1+\eps}$, $m > 0, n > 0$, $m - n = l$ and $\max(m,n) \asymp Y
k$ (with the last condition imposed by a normalized smooth truncation).

We exploit (in Lemma \ref{lem:j-bound-legit} and Corollary
\ref{prop:bound-integral}; see also 
remark \ref{rmk:comparison}) what amounts to the overwhelming
decay of the Bessel factor $\kappa_{\phi,\infty}(\dotsb)$ in the
higher-dimensional generalization of
\eqref{eq:refined-holowinsky} when $m,n$ lie in the outskirts of
the region $\mathcal{R}$; the simple proof that we give amounts
to some amusing inequalities satisfied by the hypergeometric
function and ratios of pairs of Gamma functions (see \S
\ref{sec:bounds-an-integral}).  In this way we reduce to
bounding shifted sums of the form
\eqref{eq:holow-shifted-sums-intro} taken over $\mathfrak{o}
\cap \mathcal{R} '$ with $\mathcal{R} '$ the much smaller region
bounded by the hyperbola $\{ x_1 \dotsb x_{d} = X \}$ and the
hyperplanes $\{ x_i = k_i Y^{1/d} / U \}$ with $X = k_1 \dotsb
k_{d} Y$ and $U = \exp(\log(X)^\eps)$.  The volume of
$\mathcal{R}'$ is merely $ \approx X \log(U)^{d-1} = X
\log(X)^{\eps '}$ with $\eps ' = (d-1) \eps$, and this
arbitrarily small logarithmic power $\log(X)^{\eps '}$ is
negligible in seeking estimates of type
\eqref{eq:holow-bound-shifted-sums-intro} and
\eqref{eq:holow-bound-iw} which already contain such a factor.  The rest of our argument proceeds
essentially as it did for Holowinsky
upon replacing his Mellin transforms on $\mathbb{R}_+^*$
by Mellin transforms on certain quotients
of the idele class group of $\mathbb{F}$, although some new features
do arise (e.g., when $\mathbb{F}$ has general class number we
must consider Hilbert modular varieties having multiple
connected components and exclude certain dihedral forms from our
analysis).  We elaborate on these last few paragraphs in
successively greater detail in \S\ref{sec:brief-revi-holow} and
\S\ref{sec:key-arguments-our}.

\subsection{Plan for the paper}
In \S\ref{sec:preliminaries} we introduce notation that will
allow us to speak meaningfully about automorphic forms over
totally real fields.
In \S\ref{sec:brief-revi-holow} we review
the work of Holowinsky and Soundararajan
over $\mathbb{F} = \mathbb{Q} $
and reduce the proof of our main result
Theorem \ref{thm:1} to that of a
generalization (Theorem \ref{thm:holow})
of Holowinsky's bound \eqref{eq:holow-bound-iw}.
The heart of our paper is \S\ref{sec:key-arguments-our},
in which we prove Theorem \ref{thm:holow}
assuming some independent technical results
that we relegate to
\S\ref{sec:reduct-shift-sums},
\S\ref{sec:bounds-shifted-sums},
\S\ref{sec:reduct-from-class}
and
\S\ref{sec:bounds-an-integral}.

\subsection{Acknowledgements}
\thanks{We thank Dinakar Ramakrishnan for suggesting this
  problem and for his very helpful feedback and comments on
  earlier drafts of this paper.  We thank Fokko van de Bult for
  a conversation that led to a strengthening and simplification
  of the proof of Lemma \ref{lem:technical-1}.  We thank Roman
  Holowinsky, Philippe Michel, Peter Sarnak, and K. Soundararajan for their
  encouragement. We thank the referee for the careful
  reading and comments that
  have helped improve our exposition.
  This work represents part of the author's doctoral
  dissertation written at the California
  Institute of Technology.}
\section{Preliminaries}\label{sec:preliminaries}

\subsection{Number fields}\label{sec:number-fields}
Let ${\mathbb{F}}$ be a totally real number field, $\mathbb{A}$ its
adele ring, $\mathbb{A}_f \subset \mathbb{A}$ the subring of
finite
adeles, $I_{\mathbb{F}}$ the group of fractional ideals in
${\mathbb{F}}$,
$\mathbb{F}_\infty = \mathbb{F} \otimes_\mathbb{Q} \mathbb{R}$,
$0 \neq e_{\mathbb{F}} \in \Hom(\mathbb{A}/{\mathbb{F}}, S^1)$ the standard
nontrivial additive character
(i.e., normalized so that its restriction $e_{\mathbb{F}_\infty}$
to $\mathbb{F}_{\infty} = \mathbb{F}_{\infty} \times \{0\}
\subset
\mathbb{F}_{\infty} \times \mathbb{A}_f = \mathbb{A}$ is given by
$e_{\mathbb{F}_\infty}(x) = e^{2 \pi i \Tr(x)}$),
${\mathbb{F}}_{\infty_+}^*$ the
connected component of the identity in ${\mathbb{F}}_\infty^*$,
$\mathfrak{o}$ the ring of integers in ${\mathbb{F}}$, $\hat{\mathfrak{o}}^* =
\prod_{v < \infty} \mathfrak{o}_v^* < \mathbb{A}_f^*$
the maximal compact subgroup of the finite ideles, and $\mathfrak{o}_+^* =
\mathfrak{o}^* \cap {\mathbb{F}}_{\infty_+}^*$ the group of totally
positive units of $\mathfrak{o}$, which is free abelian of rank
$[\mathbb{F}:\mathbb{Q}]-1$.
Let
$C_\mathbb{F} =
\mathbb{F}^* \backslash \mathbb{A}^*$ denote the
idele class group of $\mathbb{F}$
and
$C_\mathbb{F}^1 \leq C_{\mathbb{F}}$
the (compact) kernel of the adelic absolute value.

Let $\div \alpha \in
I_{\mathbb{F}}$ denote the fractional ideal generated by
an idele $\alpha \in \mathbb{A}^*$
and $\norm(\mathfrak{a})$ the (absolute) norm
of a fractional ideal $\mathfrak{a}$.
Let $\mathfrak{d}$ be the different of $\mathbb{F}$, so that
$\mathfrak{d}^{-1}$ is the dual of $\mathfrak{o}$ with respect
to the bilinear form $\mathbb{F} \times \mathbb{F} \ni (x,y) \mapsto e_{\mathbb{F}}(x
y)$
and $\Delta_\mathbb{F} = \norm(\mathfrak{d})$ is the discriminant
of $\mathbb{F}$.
Let $h(\mathbb{F})$ be the (finite) narrow class number of
$\mathbb{F}$ and $\mathfrak{z}_1, \dotsc,
\mathfrak{z}_{h(\mathbb{F})}$ a set of representatives
for the group of narrow ideal classes.
Choose finite ideles 
$d_{\mathbb{F}}, z_1, z_2, \dotsc, z_{[\mathbb{F}:\mathbb{Q}]} \in
\mathbb{A}_f^*$ such that
$\div d_{\mathbb{F}} = \mathfrak{d}$ and
$\div z_j = \mathfrak{z}_j$ for $j = 1, \dotsc,
h(\mathbb{F})$.
Then we have natural identifications
\begin{equation}\label{eq:100}
  \mathbb{A}^* = \sqcup_{j=1}^{h(\mathbb{F})}
  {\mathbb{F}}^*({\mathbb{F}}_{\infty+}^* \times
  z_j^{-1} \widehat{\mathfrak{o}}^*),
  \quad
  {\mathbb{F}}^* \backslash \mathbb{A}^* / \hat{\mathfrak{o}}^*
  = \sqcup_{j=1}^{h(\mathbb{F})} \left(
    ({\mathbb{F}}_{\infty_+}^* / \mathfrak{o}^*_+)
    \times z_j^{-1} \right).
\end{equation}
We let $\mathfrak{p}$ denote a typical prime ideal of
$\mathfrak{o}$
and $v$ a typical place of $\mathbb{F}$.

\subsection{Asymptotic notation}\label{sec:asymptotic-notation}
We use the asymptotic notation $\ll, \asymp$, $O()$ in the
strong sense that certain inequalities should hold for all
values of the parameters under consideration and not merely
eventually with respect to some limit.  For instance, we write
$f(x,y,z) \ll_{x,y} g(x,y,z)$ to indicate that there exists a
positive real $C(x,y)$, possibly depending upon $x$ and $y$ but
not upon $z$, such that $| f(x,y,z)| \leq C(x,y) |g(x,y,z)|$ for
all $x,y$ and $z$ under consideration; here $C(x,y)$ is called
an \emph{implied constant}.  We write $f(x,y,z) =
O_{x,y}(g(x,y,z))$ synonymously for $f(x,y,z) \ll_{x,y}
g(x,y,z)$ and write $f(x,y,z) \asymp_{x,y} g(x,y,z)$
synonymously for $f(x,y,z) \ll_{x,y} g(x,y,z) \ll_{x,y}
f(x,y,z)$.  On the other hand, the notation $f(x) = o(g(x))$
only
makes sense in the context of a limit, and we give it the
standard meaning $f(x)/g(x) \rightarrow 0$.

We regard the number field ${\mathbb{F}}$ as fixed, so that any
implied constants may depend on it without mention.  We
similarly regard the choice of narrow ideal class
representatives $\mathfrak{z}_1, \dotsc,
\mathfrak{z}_{h(\mathbb{F})}$ as fixed.  We let $\eps \in
(0,0.01)$ denote a sufficiently small parameter and $A \geq 100$
a sufficiently large parameter, which we allow to assume
finitely many distinct values throughout our analysis.  We allow
our implied constants to depend on $\eps$ and $A$ without
mention.

\subsection{Real embeddings}\label{sec:real-embeddings}
Set $d = [\mathbb{F}:\mathbb{Q}]$ for now.
An ordering
on the real embeddings $\infty_1, \dotsc,
\infty_{d}$ of ${\mathbb{F}}$ determines
a linear inclusion ${\mathbb{F}} \hookrightarrow
\mathbb{R}^{d}$
(the Minkowski embedding), which we fix.
For $x
\in \mathbb{R}^{d}$ write $x_i$ for its
$i$th component, so that $x_i = x^{\infty_i}$ when $x \in
{\mathbb{F}}$.  For $x,y \in
\mathbb{R}^{d}$ and $\alpha \in
\mathbb{R}_{>0}^{d}$ we define
$\mathbf{max}(x,y), \mathbf{min}(x,y), |x| \in
\mathbb{R}^{d}$ and $x^\alpha \in
\mathbb{R}$ by
\begin{align*}
  \mathbf{max}(x,y) &= (\max(x_1,y_1),\dotsc,\max(x_d,y_d)), \\
  \mathbf{min}(x,y) &= (\min(x_1,y_1),\dotsc,\min(x_d,y_d)), \\
  |x| &= (|x_1|,\dotsc,|x_d|), \\
  x^\alpha &= x_1^{\alpha_1} \dotsb x_d^{\alpha_d}.
\end{align*}
These definitions apply in particular when $x, y \in
{\mathbb{F}} \hookrightarrow
\mathbb{R}^{d}$.
We write simply
\[
\mathbf{1} = (1,\dotsc,1),
\quad 
\mathbf{0} = (0,\dotsc,0),
\]
so that $x^{\mathbf{1}} = x_1 \dotsb x_d$ for $x \in
\mathbb{R}^d$.  We extend the Gamma function multiplicatively to
$\vecgamma : (\mathbb{C} - \mathbb{Z}_{\leq 0})^d \rightarrow
\mathbb{C}$ by the formula $\vecgamma(z) = \Gamma(z_1) \dotsb
\Gamma(z_d)$ for $z \in (\mathbb{C} - \mathbb{Z}_{\leq 0})^d$.
As an example of our notation, for $k = (k_1,\dotsc,k_d)
\in (2 \mathbb{Z}_{\geq 1})^d$
we have
\[
\frac{(4 \pi
  \mathbf{1})^{k-\mathbf{1}}}{\vecgamma(k-\mathbf{1})}
=
\frac{(4 \pi)^{k_1-1}}{\Gamma(k_1-1)}
\dotsb
\frac{(4 \pi)^{k_d-1}}{\Gamma(k_d-1)}.
\]
We extend the relations $R \in \{ <, \leq , \geq , > \}$
componentwise to partial orders on
$\mathbb{R}^{d}$, writing $x \, R \, y$ to
denote that $x_i \, R \, y_i$ for all $i \in \{1, \dotsc,
d\}$; in particular, $x > \mathbf{0}$
signifies that $x_i > 0$ for all $i$, i.e., that $x$ is totally
positive.

\subsection{Groups}\label{sec:groups}
Let $G = \GL(2)/\mathbb{Q}$ with the usual subgroups
\[
B = \{
\left( \begin{smallmatrix} * & * \\ & *
  \end{smallmatrix} \right) \}, \quad 
N = \{ \left( \begin{smallmatrix} 
    1 & * \\ 
    & 1 
  \end{smallmatrix} \right) \} 
, \quad 
A = \{ \left( \begin{smallmatrix} 
    * &  \\ 
    & * 
  \end{smallmatrix} \right) \}, \quad 
Z = \{ \left( \begin{smallmatrix} 
    z &  \\ 
    & z 
  \end{smallmatrix} \right) \}
\]
and the accompanying notation
\[
n(x) = \left( \begin{smallmatrix} 
    1 & x \\ 
    & 1 
  \end{smallmatrix} \right)
\in N(\mathbb{A}),
\quad a(y) = \left( \begin{smallmatrix} y & \\ & 1
  \end{smallmatrix} \right) \in A(\mathbb{A})
\]
for $x \in \mathbb{A}$ and $y \in \mathbb{A}^*$.
Put ${\mathbf{X}} = Z(\mathbb{A}) G({\mathbb{F}})
\backslash G(\mathbb{A})$.

Let $K_\infty = \SO(2)^{[\mathbb{F}:\mathbb{Q}]}$ be the standard maximal
compact (connected) subgroup of $G(\mathbb{F}_\infty)$,
let
\[K_{\fin} = \prod_{v<\infty}
\left\{ 
  \left(
    \begin{smallmatrix}
      a &b\\
      c&d
    \end{smallmatrix}
  \right) \in G(\mathbb{F}_v) : a,d \in \mathfrak{o}_v, b \in
  \mathfrak{d}_v^{-1}, c \in \mathfrak{d}_v \right\}),
\]
and let $K = K_\infty \times K_{\fin}$.  Then $K$ is the
conjugate by $a(1 \times d_{\mathbb{F}}^{-1})$ of the standard maximal
compact subgroup of $G(\mathbb{A})$.  Our choice of $K_{\fin}$
follows Shimura \cite{MR507462} and is convenient because
the restriction
to $G(\mathbb{F}_{\infty})$
of a right-$K_{\fin}$-invariant automorphic form on
$G(\mathbb{A})$
has a Fourier expansion indexed by
the ring of integers $\mathfrak{o}$ rather than
by the inverse different $\mathfrak{d}^{-1}$.

By the Iwasawa decompositon
$G(\mathbb{A}) = N(\mathbb{A}) A(\mathbb{A}) K$,
we may define a function on $G(\mathbb{A})$
by prescribing the values it
takes on elements of the form $g = n(x) a(y) k z$ with $x \in
\mathbb{A}$, $y \in \mathbb{A}^*$, $k \in K$, and $z \in
Z(\mathbb{A})$, provided that these values do not depend upon
the choice of $x,y,k,z$ in expressing $g = n(x) a(y) k z$.

\subsection{Measures}\label{sec:measures}
We normalize Haar measures on the locally compact groups
$\mathbb{A}$, $\mathbb{A}^*$,
and $K$
by requiring that
\[
\vol(\mathbb{A}/\mathbb{F}) =  \vol( (1,e)^{[\mathbb{F}:\mathbb{Q}]} \times
\hat{\mathfrak{o}}^*)
= \vol(K) = 1.
\]
We give $\mathbb{A}/\mathbb{F}$ and $C_{\mathbb{F}} = \mathbb{A}^*/\mathbb{F}^*$
the quotient measures defined with respect to the counting
measures on the discrete subgroups $\mathbb{F}$, $\mathbb{F}^*$;
more generally we give discrete groups such as $N(\mathbb{F}),
B(\mathbb{F}), A(\mathbb{F})$, and $G(\mathbb{F})$  the counting
measure and normalize accordingly the Haar measures
on quotients thereof.  We
normalize the Haar measure on
$Z(\mathbb{A}) \backslash G(\mathbb{A})$ by requiring that
\begin{equation}\label{eq:31}
  \int_{Z(\mathbb{A}) B(\mathbb{Q}) G(\mathbb{A})}
  \phi 
  = \int_{x \in \mathbb{F} \backslash \mathbb{A}}
  \int_{y \in \mathbb{F}^* \backslash \mathbb{A}^*}
  \int_{k \in K}
  \phi(n(x) a(y) k) \, d x \, \frac{d^\times y}{|y|_\mathbb{A} } \, d k
\end{equation}
for all compactly supported continuous functions
$\phi$ on $Z(\mathbb{A}) B(\mathbb{Q}) \backslash G(\mathbb{A})$.
This choice defines a quotient measure $\mu$ on $\mathbf{X}
= Z(\mathbb{A}) G(\mathbb{F}) \backslash G(\mathbb{A})$.
Finally, we choose a Haar measure on $C_{\mathbb{F}}^1$
so that the corresponding quotient measure
on $C_{\mathbb{F}} / C_{\mathbb{F}}^1 \cong \mathbb{R}_+^*$
is the standard Haar measure $d^\times t = t^{-1} \, d t$.

\subsection{Characters}\label{sec:characters}
We introduce some notation related to the
Fourier transform on the idele class group $C_{\mathbb{F}} =
\mathbb{F}^* \backslash \mathbb{A}^*$, and in particular
its ``unramified'' quotient $C_{\mathbb{F}} / \hat{\mathfrak{o} }^*$.

Let $\mathfrak{X}(H)$ denote the group of (quasi-)characters
on a topological abelian group $H$, thus $\mathfrak{X}(H)$ is
the group of continuous homomorphisms $\chi : H \rightarrow
\mathbb{C}^*$; a character having image in
the circle group $S^1$
will be called a \emph{unitary} character.
For a quotient group $H'' = H / H'$ with $H'$
closed in $H$, identify $\mathfrak{X}(H'')$
with the subgroup of $\mathfrak{X}(H)$ consisting of those
characters having trivial restriction to $H'$.

Let the group $\mathfrak{X}(C_{\mathbb{F}})$ of idele class
characters on $\mathbb{F}$ carry the structure of a complex
manifold whose connected components are the cosets of the
subgroup $\mathfrak{X}(C_{\mathbb{F}}/C_{\mathbb{F}}^1) =\{|.|^s
: s \in \mathbb{C} \}$ on which the complex structure is given
by $s$; here $|.| = |.|_{\mathbb{A}}$ is the adelic absolute
value $C_{\mathbb{F}} \ni (x_v)_v \mapsto \prod |x_v|_v \in
\mathbb{R}^*_+$ with $|.|_v$ the standard absolute value on the
completion $\mathbb{F}_v$ of $\mathbb{F}$, so that
multiplication by $x_v$ scales the Haar measure on
$\mathbb{F}_v$ by $|x_v|_v$.

Since $C_{\mathbb{F}}^1$ is compact, for each $\chi \in
\mathfrak{X}(C_{\mathbb{F}})$ we have $|\chi| = |.|^\sigma$
for some $\sigma \in \mathbb{R}$, which we call the \emph{real
  part} of $\chi$ and denote by $\sigma = \Re(\chi)$.  Let
$\mathfrak{X}(C_{\mathbb{F}})(c)$ denote the set of
idele class characters having real part $c$.

Let
\[
\mathfrak{X}(C_{\mathbb{F}})[2]
:= \{ \chi_0 \in \mathfrak{X}(C_{\mathbb{F}}): \chi_0^2 = 1
\}
\]
denote the group of \emph{quadratic} idele class characters.
This is not to be confused with the set
$\mathfrak{X}(C_{\mathbb{F}})(2)$ of idele class characters
$\chi$ having real part $\Re(\chi) = 2$.

Let $\chi_{\infty} \in
\mathfrak{X}(\mathbb{F}_{\infty}^*)$ denote
the restriction of an idele
class character $\chi \in \mathfrak{X}(C_{\mathbb{F}})$ to
$\mathbb{F}_{\infty}^*$.
Then $\chi_{\infty}$ is of the form
\begin{equation}\label{eq:33}
  y \mapsto \prod_{i=1}^{[\mathbb{F}:\mathbb{Q}]}
  \sgn(y_j)^{\eps_j} |y_j|^{i r_j}
  \quad \text{ if }
  y =
  (y_1,\dotsc,y_{[\mathbb{F}:\mathbb{Q}]})
  \in (\mathbb{R}^{[\mathbb{F}:\mathbb{Q}]})^* =
  \mathbb{F}_{\infty}^*
\end{equation}
for some $\eps_j \in \{0, 1\}$ and $r_j \in \mathbb{C}$; the
character $\chi_\infty$ is unitary if and only if each $r_j \in
\mathbb{R}$.  For a place $v$ of $\mathbb{F}$, let $\chi_v$ be
the restriction of $\chi$ to $\mathbb{F}_{v}^* \hookrightarrow
\mathbb{A}^*$; in particular, $\chi_{\infty_j} =[y_j \mapsto
\sgn(y_j)^{\eps_j} |y_j|^{i r_j}]$ is the
restriction of $\chi_\infty$ as above to the $j$th factor of
$(\mathbb{R}^{[\mathbb{F}:\mathbb{Q}]})^*$,

The group $\mathfrak{X}(C_{\mathbb{F}}/\hat{\mathfrak{o} }^*)$
of unramified idele class characters $\chi$
is a subgroup of the group $\mathfrak{X}(C_{\mathbb{F}})$
of all idele class characters;
here and elsewhere unramified means ``unramified at all finite
places.''
Set
$\mathfrak{X}(C_{\mathbb{F}}/\hat{\mathfrak{o} }^*)(c)
:= \mathfrak{X}(C_{\mathbb{F}}/\hat{\mathfrak{o} }^*)
\cap \mathfrak{X}(C_{\mathbb{F}})(c)$
for any $c \in \mathbb{R}$
and $\mathfrak{X}(C_{\mathbb{F}}/\hat{\mathfrak{o} }^*)[2]
:= \mathfrak{X}(C_{\mathbb{F}}/\hat{\mathfrak{o} }^*)
\cap \mathfrak{X}(C_{\mathbb{F}})[2]$.

Let
\[
\xi_{\mathbb{F}} :
\mathfrak{X}(C_{\mathbb{F}} / \hat{\mathfrak{o}}^*)
\rightarrow \mathbb{P}^1(\mathbb{C})
\]
be the (completed) Dedekind zeta function,
defined for unramified idele class characters
of real part $\Re(\chi) > 1$ by the Euler product
$\xi_{\mathbb{F}}(\chi) = \prod_v \zeta_v(\chi_v)$
and in general by meromorphic continuation,
where $\zeta_\mathfrak{p}(v) =
(1-\chi_\mathfrak{p}(\varpi_\mathfrak{p}))^{-1}$
for $\varpi_\mathfrak{p}$ a generator of $\mathfrak{p} \subset
\mathbb{F}_\mathfrak{p}$
and $\zeta_{\infty_j}(\chi_{\infty_j}) = \Gamma_{\mathbb{R}}(i
r_j + \eps_j)$
if $\chi_\infty$ is given by \eqref{eq:33};
here $\Gamma_{\mathbb{R}}(s) = \pi^{-s/2} \Gamma(s/2)$.
For $s \in \mathbb{C}$ let $\xi_{\mathbb{F}}(s)
:= \xi_{\mathbb{F}}(|.|^s)$, which agrees with the usual definition.
Hecke proved that $\xi_{\mathbb{F}}$ is holomorphic away
from its simple pole at $\chi = |.|$ and satisfies a functional
equation relating its values
at $\chi$ and $|.| \chi^{-1}$.

Let $\Psi \in
C_c^\infty(C_{\mathbb{F}}/ \hat{\mathfrak{o} }^*)$
be a test function.
For each character $\chi \in
\mathfrak{X}(C_{\mathbb{F}}/\hat{\mathfrak{o} }^*)$
let $\Psi^\wedge(\chi)$ be the Fourier-Mellin transform
of $\Psi$ at $\chi$ normalized 
so that the inversion formula
\begin{equation}
  \Psi(y) = \int_{\mathfrak{X}(C_{\mathbb{F}}/\hat{\mathfrak{o} }^*)(c)}
  \Psi^\wedge(\chi) \chi(y) \, \frac{d \chi }{2 \pi i}
\end{equation}
holds,
where $\int_{\mathfrak{X}(C_{\mathbb{F}}/\hat{\mathfrak{o}
  }^*)(c)}$
denotes the contour integral over unramified idele class
characters $\chi$ having real part $c > 1$
taken in the usual vertical sense,
precisely
\[
\int_{\mathfrak{X}(C_{\mathbb{F}}/\hat{\mathfrak{o} }^*)(c)}
\Psi^\wedge(\chi) \chi(y) \, \frac{d \chi }{2 \pi i}
:=
\sum_{
  \chi_0 \in
  \frac{
    \mathfrak{X}
    (
    C_{\mathbb{F}}/\hat{\mathfrak{o}}^*
    )(0)
  }
  {
    \mathfrak{X}(C_{\mathbb{F}} / C_{\mathbb{F}}^1)
  }
  % \mathfrak{X}
  % (
  % C_{\mathbb{F}}^1/\hat{\mathfrak{o}}^*
  % )
}
\int_{(c)}
\Psi^\wedge(\chi_0 |.|^s) \chi_0(y) |y|_{\mathbb{A}}^s
\, \frac{d s}{2 \pi i},
\]
where $\int_{(c)}$ denotes the vertical contour integral
taken over $\Re(s) = c$ 
from $c - i \infty$ to $c + i \infty$,
and as representatives for the quotient
$\mathfrak{X}(C_{\mathbb{F}}/\hat{\mathfrak{o} }^*)
/ \mathfrak{X}(C_{\mathbb{F}} / C_{\mathbb{F}}^1)$
one may take the image of
the discrete group $\mathfrak{X}(C_{\mathbb{F}}^1/\hat{\mathfrak{o} }^*)$
under pullback by a section of the inclusion $C_{\mathbb{F}}^1
\hookrightarrow C_{\mathbb{F}}$.
By our normalization of measures (see \S\ref{sec:measures}),
the forward transform is given explicitly by
\begin{equation}
  \Psi^\wedge(\chi)
  = \frac{1}{\vol(C_{\mathbb{F}}^1)}
  \int_{C_{\mathbb{F}}} \Psi(y) \chi^{-1}(y) \, d^\times y.
\end{equation}

The analytic conductor \cite{MR1826269} of an unramified idele
class character $\chi \in
\mathfrak{X}(C_{\mathbb{F}}/\hat{\mathfrak{o} }^*)$
having archimedean component \eqref{eq:33}
is defined
to be
\begin{equation}
  C(\chi)
  = \prod_{i=1}^{[\mathbb{F}:\mathbb{Q}]}
  (3 + |r_j|);
\end{equation}
the number $3$ is unimportant and present only so that
$\log C(\chi)$ is never too small.
Repeated ``partial integration'' shows that
$\Psi^\wedge(\chi) \ll_{\Psi,A} C(\chi)^{-A}$
for any test function $\Psi \in
C_c^\infty(C_{\mathbb{F}}/\hat{\mathfrak{o} }^*)$
and any positive integer $A$,
uniformly for $\Re(\chi)$ in any bounded set.
Concretely, we have
natural short exact sequences
\[
1 \rightarrow \mathbb{F}_{\infty+}^* /  \mathfrak{o}^*_+
\rightarrow C_{\mathbb{F}} / \hat{\mathfrak{o} }^*
\rightarrow \Cl_{\mathbb{F}}^+ \rightarrow 1,
\]
and
\[
1 \rightarrow \mathbb{F}_{\infty+}^1 / \mathfrak{o}_+^*
\rightarrow \mathbb{F}_{\infty+}^* / \mathfrak{o}_+^*
\xrightarrow{x \mapsto x^{\mathbf{1}}} \mathbb{R}_+^* \rightarrow 1,
\]
where $\Cl_{\mathbb{F}}^+ = C_{\mathbb{F}} /
(\mathbb{F}_{\infty+}^*
\times \hat{\mathfrak{o} }^*)$ is the (finite) narrow
class group of $\mathbb{F}$
and $\mathbb{F}_{\infty+}^1$ is the subgroup
$\{(x_i) : \prod x_i = 1\}$ of $\mathbb{F}_{\infty+}^*$.
Thus $C_{\mathbb{F}} /
\hat{\mathfrak{o} }^*$
is an extension of a finite group
by
an extension of
$\mathbb{R}^*_+$ by a compact torus,
so the assertion
$\Psi^\wedge(\chi) \ll_{\Psi,A} C(\chi)^{-A}$
reduces to the familiar decay properties
of the Fourier transform
of a test function on
a finite product of Euclidean lines and circles.

\subsection{Fourier expansions}\label{sec:fourier-expansions}
Suppose that
$\phi : \mathbf{X} \rightarrow \mathbb{C}$ is continuous
and right-$K$-invariant.
By the Iwasawa decomposition, $\phi$ is determined
by the values $\phi(n(x) a(y))$
for $x \in \mathbb{A}, y \in \mathbb{A}^*$.
If $\phi$ is assumed merely to be right-$K_{\fin}$-invariant
but transforms under a unitary character of $K_{\infty}$,
then $|\phi|^2$ is still determined by the values
$\phi(n(x) a(y))$.
In either case,
the left-$B(\mathbb{F})$-invariance of $\phi$
implies a Fourier expansion
\begin{equation}\label{eq:17}
  \phi(n(x) a(y))
  = \phi_0(y) + \sum_{n \in \mathbb{F}^*}
  \kappa_\phi(n y) e_{\mathbb{F}}(n x)
\end{equation}
for some functions $\phi_0$ on $C_{\mathbb{F}} /
\hat{\mathfrak{o}}^* = \mathbb{F}^* \backslash \mathbb{A}^* /
\hat{\mathfrak{o}}^*$ and $\kappa_\phi$ on $\mathbb{A}^* /
\hat{\mathfrak{o}}^*$ (see \cite{MR1610485}).

We say that the Fourier expansion \eqref{eq:17}
of $\phi$ is \emph{factorizable}
if for each $y \times z \in \mathbb{F}_{\infty}^*
\times \mathbb{A}_f^* = \mathbb{A}^*$
we have
\begin{equation}
  \kappa_\phi(y \times z)
  = \kappa_{\phi,\infty}(y)
  \frac{\lambda_\phi(\div z)}{\norm(\div z)^{1/2}},
\end{equation}
where $\lambda_\phi : I_{\mathbb{F}}
\rightarrow \mathbb{C}$
is a weakly multiplicative function supported
on the monoid of integral ideals
and $\kappa_{\phi,\infty}(y)
= \prod_{j=1}^{[\mathbb{F}:\mathbb{Q}]}
\kappa_{\phi,\infty_j}(y_j)$
for some functions $\kappa_{\phi,\infty_j} : \mathbb{R}^*
\rightarrow \mathbb{C}$.

\subsection{Automorphic forms}\label{sec:automorphic-forms}
We shall consider various kinds of automorphic forms
throughout this paper.  In this
section we give them convenient names and state their relevant
properties.
\subsubsection{Holomorphic eigencuspforms}\label{sec:holom-eigenc}
By a \emph{holomorphic eigencuspform} $f : \mathbf{X} \rightarrow
\mathbb{C}$ of weight $k =
(k_1,\dotsc,k_{[\mathbb{F}:\mathbb{Q}]})$
(here and always each $k_j$ is a \emph{positive} even integer,
for simplicity)
we mean an arithmetically normalized cuspidal
holomorphic Hilbert modular form of weight $k$, full level, and
trivial central character, that is furthermore an eigenfunction
of the algebra of Hecke operators.
Precise definitions in both the classical and adelic languages
appear in Shimura's paper \cite{MR507462};
for our purposes, it is necessary
to know only that $f$ is right $K_{\fin}$-invariant, transforms under a
(specific) unitary character of $K_{\infty}$, and has a factorizable
Fourier expansion \eqref{eq:17} with $f_0 \equiv 0$ and
\begin{equation}
  \kappa_{f,\infty_j}(y)
  =
  \begin{cases}
    y^{k_j/2} e^{- 2 \pi y} & \text{ for } y > 0, \\
    0 & \text{ for } y < 0
  \end{cases}
\end{equation}
for each infinite place
$\infty_j$ of $\mathbb{F}$.
The ``Ramanujan bound'' for $f$ \cite{MR2327298}
asserts\footnote{the parity conditions on the weight of $f$
  are satisfied
  because $f$ has trivial central character,
  hence the $k_i$ are all even}
that
$|\lambda_f(\mathfrak{a})| \leq \tau(\mathfrak{a})$
for each integral ideal $\mathfrak{a}$,
where $\tau$ is the
divisor function (multiplicative, $\mathfrak{p}^k \mapsto
k+1$); this improves an earlier result of Brylinski-Labesse,
which asserts that $|\lambda_f(\mathfrak{p})| \leq 2$ for a full density
set of primes $\mathfrak{p}$.

To $f$ and an unramified idele class character
$\chi \in \mathfrak{X}(C_{\mathbb{F}} / \hat{\mathfrak{o} }^*)$
of sufficiently large real part
we associate the finite part of the adjoint
$L$-function
\[
L(\ad f, \chi) = \prod_{\mathfrak{p}}
L_\mathfrak{p}(\ad f,\chi)
\]
and its completion $\Lambda(\ad f, \chi) = L_\infty(\ad f, \chi)
L(\ad f, \chi ) = \prod_v L_v(\ad f, \chi)$, where the local
factors are as in \cite[\S 3.1.1]{watson-2008}.  It is known
\cite{Sh75,MR533066} that $\chi \mapsto L(\ad f, \chi)$
continues meromorphically to a function on
$\mathfrak{X}(C_{\mathbb{F}} / \hat{\mathfrak{o} }^*)$ whose
only possible poles are simple and at $\chi = \chi_0 |.|$ for
$\chi_0 \in \mathfrak{X}(C_{\mathbb{F}}/\hat{\mathfrak{o}
}^*)[2]$ a quadratic character.  Call
$f$ \emph{nondihedral} if $L(\ad f, \cdot) :
\mathfrak{X}(C_{\mathbb{F}}/\hat{\mathfrak{o} }^*) \rightarrow
\mathbb{P}^1(\mathbb{C})$ is entire; this is known to be the
case precisely when $f$ is not induced from an idele class
character of a quadratic extension of $\mathbb{F}$
\cite{MR533066,MR540902}.  Note that unlike when $\mathbb{F} =
\mathbb{Q}$ or $h(\mathbb{F}) = 1$, in general (e.g.,
for $\mathbb{F} = \mathbb{Q}(\sqrt{3})$) there may exist
dihedral cusp forms of full level and trivial central character,
which we shall exclude from our analysis.

\subsubsection{Maass eigencuspforms}\label{sec:maass-eigencuspforms}
By a \emph{Maass eigencuspform} $\phi : \mathbf{X} \rightarrow
\mathbb{C}$ of Laplace eigenvalue $(\tfrac{1}{4} + r_1^2,
\dotsc, \tfrac{1}{4} + r_{[\mathbb{F}:\mathbb{Q}]}^2) \in
\mathbb{R}_{>0}^{[\mathbb{F}:\mathbb{Q}]}$
and parity $(\eps_1,\dotsc,\eps_{[\mathbb{F}:\mathbb{Q}]})
\in \{0,1\}^{[\mathbb{F}:\mathbb{Q}]}$
we mean an arithmetically normalized
Hilbert-Maass
cusp form
on $\mathbf{X}$ of given Laplace eigenvalues
and parity, full level and
trivial central character, that is furthermore an eigenfunction
of the algebra of Hecke operators.
For our purposes this means that
$\phi$ is right-$K$-invariant
and has a factorizable
Fourier expansion \eqref{eq:17}
with $\phi_0 \equiv 0$
and
\begin{equation}
  \kappa_{\phi,\infty_j}(y) = 2 |y|^{1/2} K_{i r_j}(2 \pi |y|)
  \sgn(y)^{\eps_j}
\end{equation}
for each infinite place $\infty_j$ and all $y \in \mathbb{R}^*$;
here $K_{i r_j}$ is the modified Bessel function
of the second kind.
The trivial ``Hecke bound''
asserts that
$\lambda_\phi(\mathfrak{a}) \leq
\tau(\mathfrak{a}) \norm(\mathfrak{a})^{1/2}$.
The ``Rankin-Selberg bound,''
also known as the ``Ramanujan bound on average,''
asserts that
\begin{equation}\label{eq:27}
  \sum_{\norm(\mathfrak{a}) \leq x}
  | \lambda_\phi(\mathfrak{a})|^2
  \ll_\phi x
\end{equation}
and follows
as in \cite[\S 8.2]{MR1942691}
from the analytic properties
of the Rankin-Selberg $L$-series
attached to $\phi \times \phi$
\cite{Ja72}.

\subsubsection{Eisenstein series}\label{sec:eisenstein-series}
Let $\chi \in \mathfrak{X}(C_\mathbb{F} / \hat{\mathfrak{o}}^*)$
be an unramified idele class character.
Writing
$y(g) = y$ for $g = n(x) a(y) k z$,
the map
$B(\mathbb{F}) \backslash G(\mathbb{A}) \ni g \mapsto \chi(y(g))$ is well-defined.
The \emph{Eisenstein series}
\begin{equation}
  E(\chi,g) = \sum_{\gamma \in B(\mathbb{F}) \backslash G(\mathbb{F})}
  \chi(y(\gamma g))
\end{equation}
converges normally in $g$
and uniformly in $\chi$ for $\Re(\chi) \geq 1 + \delta > 0$,
and continues
meromorphically to the union of half-planes
on which $\Re(\chi) \geq \onehalf$, where
$\chi \mapsto E(\chi,\cdot)$
is holomorphic with the exception of simple poles at $\chi =
|.| \chi_0$ of locally constant residue
proportional to $g \mapsto \chi_0(\det(g))$ for each unramified \emph{quadratic}
idele class character $\chi_0 \in \mathfrak{X}(C_\mathbb{F} /
\hat{\mathfrak{o} }^*)[2]$
(see \cite{MR546600}).  The functions $E(\chi,\cdot) : g
\mapsto E(\chi,g)$ descend to $\mathbf{X} = Z(\mathbb{A})
G(\mathbb{F}) \backslash G(\mathbb{A}) $ and are
right-$K$-invariant by construction.

The scaled Eisenstein series
$\phi =
\Delta_{\mathbb{F}}^{-1} \chi(d_{\mathbb{F}})^{-2} \xi_{\mathbb{F}}(\chi^2) E(\chi,\cdot)$ admits a
factorizable
Fourier expansion \eqref{eq:17} 
with
\begin{equation}\label{eq:30}
  \phi_0(y) =
  \Delta_{\mathbb{F}}^{-1} \chi(d_{\mathbb{F}})^{-2}
  \xi_\mathbb{F}(\chi^2)
  \chi(y)
  +
  \Delta_{\mathbb{F}}^{-1/2}
  \xi_\mathbb{F}(\chi^2 |.|^{-1})
  \chi^{-1}(y) |y|,
\end{equation}
\[
\kappa_\phi(y \times z)
= \kappa_{{(\chi|.|^{-1/2})}_\infty}(y)
\frac{
  \lambda_{(\chi|.|^{-1/2})}(\div z)
}
{
  \norm(\div z)^{1/2}
}
\]
as in \S\ref{sec:fourier-expansions},
where for $\chi \in
\mathfrak{X}(C_{\mathbb{F}}/\hat{\mathfrak{o} }^*)$
with $\chi_\infty$ given by \eqref{eq:33},
we set
\begin{equation}
  \kappa_{\chi_{\infty_j}}(y)
  = 2 |y|^{1/2} K_{i r_j}(2 \pi |y|) \sgn(y)^{\eps_j},
  \quad
  \lambda_\chi(\mathfrak{p}^k)
  = \sum_{i=0}^k \chi(\mathfrak{p})^i
  \chi^{-1}(\mathfrak{p})^{k-i};
\end{equation}
for a convenient tabulation of such
Fourier expansions of Eisenstein series
see \cite{2009arXiv0904.2429B}.

If $\chi |.|^{-1/2}$ is a unitary character (equivalently,
$\Re(\chi) = \onehalf$, i.e., $\chi \in
\mathfrak{X}(C_{\mathbb{F}}/\hat{\mathfrak{o} }^*)(\onehalf)$),
call $E(\chi,g)$ a \emph{unitary Eisenstein series}; in that
case $|\lambda_{\chi|.|^{-1/2}}(\mathfrak{a})| \leq
\tau(\mathfrak{a})$.

\subsubsection{Incomplete Eisenstein series}\label{sec:incompl-eisenst-seri}
To a test function $\Psi \in
C_c^\infty(C_{\mathbb{F}}/\hat{\mathfrak{o} }^*)$
attach the \emph{incomplete Eisenstein series}
$E(\Psi,\cdot) : \mathbf{X} \rightarrow \mathbb{C}$
by the formula
\begin{equation}
  E(\Psi,g) = \sum_{\gamma \in B(\mathbb{F}) \backslash
    G(\mathbb{F})}
  \Psi(y(\gamma g))
\end{equation}
with $y(\gamma g)$ as in \S\ref{sec:eisenstein-series}.
Write $\phi = E(\Psi,\cdot)$.
We have $\Psi^\wedge(|.|) \res_{s=1} E(|.|^s,\cdot)
= \mu(\phi) / \mu(1)$ (see \S\ref{sec:holow-sound-synth}),
so by shifting the contour in the integral representation
$E(\Psi,\cdot) =
\int_{\mathfrak{X}(C_{\mathbb{F}}/\hat{\mathfrak{o} }^*)(2)}
\Psi^\wedge(\chi) E(\chi,\cdot) \, \frac{d \chi }{2 \pi i}$
to the union of lines $\Re(\chi) = \onehalf$
(see \cite{MR546600} and \cite[\S 7.3]{MR1942691}), we obtain
\begin{equation}
  \begin{split}
    E(\Psi,g)
    &= \frac{\mu(\phi)}{\mu(1)}
    + \sum_{1 \neq \chi_0 \in
      \mathfrak{X}(C_{\mathbb{F}}/\hat{\mathfrak{o} }^*)[2]}
    c_\Psi(\chi_0) \chi_0(\det g) \\
    &\quad + \int_{\mathfrak{X}(C_{\mathbb{F}}/\hat{\mathfrak{o} }^*)(1/2)}
    \Psi^\wedge(\chi) E(\chi,g) \, \frac{d \chi }{2 \pi i}
  \end{split}
\end{equation}
for some constants $c_\Psi(\chi_0) = \mu(1)^{-1} \int_{\mathbf{X}} E(\Psi,\cdot)
(\chi_0 \circ \det)$ whose precise values are
not important for our purposes.
Taking the Fourier expansions of both sides
gives
\begin{equation}\label{eq:phi0-useful-formula-0}
  \phi_0(y)
  = \frac{\mu(\phi)}{\mu(1)}
  + \sum_{1 \neq \chi_0 \in
    \mathfrak{X}(C_{\mathbb{F}}/\hat{\mathfrak{o}}^*)[2]}
  c_{\Psi}(\chi_0)
  \chi_0(y)
  + O_{\phi}(|y|^{1/2}),
\end{equation}
\begin{equation}\label{eq:phi0-useful-formula-1}
  \kappa_\phi(y \times z)
  =
  \int_{\mathfrak{X}(C_{\mathbb{F}}/\hat{\mathfrak{o} }^*)(0)}
  \frac{
    \Psi^\wedge(|.|^{1/2} \chi)
  }{
    \xi_{\mathbb{F}}(|.| \chi^2) \chi(d_{\mathbb{F}})^{-2}
  }
  \kappa_{\chi,\infty}(y)
  \frac{\lambda_{\chi}(\div z)}{\norm(\div z)^{1/2}}
  \, \frac{d \chi }{2 \pi i}.
\end{equation}

\subsection{Masses}\label{sec:masses}
Recall the measure $\mu$ defined
on the space $\mathbf{X} = Z(\mathbb{A}) G(\mathbb{F})
\backslash G(\mathbb{A})$ in \S\ref{sec:measures}.
For $\phi \in L^1(\mathbf{X},\mu)$ let
$\mu(\phi) = \int_{\mathbf{X}} \phi \, d \mu$.
To our varying nondihedral holomorphic eigencuspform $f$
we associate the finite measure $d \mu_f = |f|^2 \, d \mu$
and write accordingly $\mu_f(\phi) = \int_{\mathbf{X}} \phi
|f|^2 \, d \mu$.
In particular, writing $1$ for the constant function
on $\mathbf{X}$, we see that
$\mu(1)$ is the volume of $\mathbf{X}$ 
and $\mu_f(1)$ the mass of $f$,
i.e., its squared norm in $L^2(\mathbf{X},\mu)$.
With this notation,
the conclusion of Theorem \ref{thm:1} is that
for any compactly supported,
continuous, right-$K$-invariant function $\phi$
on $\mathbf{X}$, we have
\[
\frac{\mu_f(\phi)}{ \mu_f(1)}
\rightarrow \frac{\mu(\phi) }{ \mu(1)}
\]
as any of the weight components of $f$ tend to $\infty$.
It suffices to show this for $\phi$ a Maass
eigencuspform or incomplete Eisenstein series
as in \S\ref{sec:maass-eigencuspforms}
and \S\ref{sec:incompl-eisenst-seri}.

The special value $L(\ad f, 1)$ enters our analysis through the
Rankin-type formula
\begin{equation}\label{eq:26}
  \mu_f(1) =
  \frac{\vecgamma(k)}{c_1(\mathbb{F}) (4 \pi
    \mathbf{1})^{k-\mathbf{1}}}
  L(\ad f, 1),
  \quad c_1(\mathbb{F})
  := \frac{(4 \pi^2)^{[\mathbb{F}:\mathbb{Q}]}}{2 \Delta_{\mathbb{F}}^{3/2}}.
\end{equation}
We sketch the standard calculation.  Recall the measure normalization
\eqref{eq:31} and the choice of compact subgroup $K$ (\S\ref{sec:groups}) on which we base our definition (\S\ref{sec:eisenstein-series})
of $E(s,\cdot)$.
For $\Re(s) > 1$ we find
by unfolding that
\begin{eqnarray*}
  \mu_f(E(s,\cdot))
  &=& \int_{Z(\mathbb{A}) B(\mathbb{F})
    \backslash G(\mathbb{A})} |y(g)|_\mathbb{A}^s |f|^2(g) \, d g \\
  &=& \int_{x \in \mathbb{F} \backslash \mathbb{A}}
  \int_{y \in \mathbb{F}^* \backslash \mathbb{A}^*}
  |y|_\mathbb{A}^{s-1} |f|^2(n(x) a(y)) \, d x \, d^\times y \\
  &=& \prod_v \int_{y \in \mathbb{Q}_v^*} |y|_v^{s-1} | \kappa_f(y)
  |^2 \, d^\times y \\
  &=&
  \Lambda(\ad f, s)
  \frac{\xi_{\mathbb{F}}(s)}{\xi_{\mathbb{F}}(2 s)}
  \prod_{i=1}^{[\mathbb{F}:\mathbb{Q}]} 2^{-k_i-1}
\end{eqnarray*}
by local calculations as conveniently tabulated in \cite[\S
3.2.1]{watson-2008}.
Since the Fourier expansion \eqref{eq:30} implies
\[
\res_{s=1} E(s,\cdot)
= \Delta_{\mathbb{F}}^{-3/2} \frac{\res_{s=1}
  \xi_{\mathbb{F}}(s)}{2 \xi_{\mathbb{F}}(2)}
\]
and by definition
\cite[\S 3.1.1]{watson-2008}
\[
L_\infty(\ad f,1)
\prod_{i=1}^{[\mathbb{F}:\mathbb{Q}]}
2^{-k_i-1}
= (4 \pi ^2)^{-[\mathbb{F}:\mathbb{Q}]} \frac{\vecgamma(k)}{(4
  \pi \mathbf{1} )^{k-\mathbf{1}}},
\]
we obtain the claimed formula \eqref{eq:26}.

\section{Brief review of Holowinsky-Soundararajan}\label{sec:brief-revi-holow}
In this section we summarize the Holowinsky-Soundararajan
\cite{holowinsky-soundararajan-2008} proof
of Theorem \ref{thm:1}
when $\mathbb{F} = \mathbb{Q}$
and indicate which of their arguments require generalization
when $\mathbb{F}$ is a general totally real number field.
Their proof combines
\begin{enumerate}
\item the independent arguments of Holowinsky \cite{holowinsky-2008}, and
\item the independent arguments of Soundararajan \cite{MR2680497},
\item the joint Holowinsky-Soundararajan synthesis of (1) and (2).
\end{enumerate}
As we shall see,
Soundararajan's independent arguments
and the Holowinsky-Soundararajan synthesis
generalize painlessly, so the essential difficulty
is to generalize Holowinsky's arguments.
In this section,
$f$ is a holomorphic eigencuspform
of weight $k = (k_1,\dotsc,k_{[\mathbb{F}:\mathbb{Q}]})$.
Recall from \S\ref{sec:real-embeddings}
that $k^{\mathbf{1}} := k_1 \dotsc k_{[\mathbb{F}:\mathbb{Q}]}$,
thus when $\mathbb{F} = \mathbb{Q}$
we have $k = (k_1)$ and $k^{\mathbf{1}} = k_1$.

\subsection{Holowinsky's independent arguments}\label{sec:holow-indep-argum}
We begin by simultaneously recalling Holowinsky's main result
\cite[Cor 3]{holowinsky-2008} and stating our generalization
thereof.  Define for each holomorphic eigencuspform $f$ and each
real number $x \geq 2$ the quantities
\begin{equation}\label{eq:review:mf-defn}
  M_f(x) =
  \frac{\log(x)^{-2}}{ L(\ad f, 1)}
  \prod_{\norm (\mathfrak{p}) \leq x}
  \left( 1 + \frac{2 \lvert \lambda_f(\mathfrak{p}) \rvert}{\norm (\mathfrak{p}) } \right),
\end{equation}
\begin{equation}\label{eq:review:rf-defn}
  R_f(x)
  = \frac{x^{-1/2}}{L(\ad f, 1)}
  \sum_{\chi_0 \in \mathfrak{X}(C_{\mathbb{F}}/\hat{\mathfrak{o}
    }^*)[2]}
  \int_{(1/2)}
  \left\lvert \frac{L(\ad f, \chi_0 |.|^s )}{C(\chi_0 |.|^s)^{10}} \right\rvert
  \, | d s|.
\end{equation}
Here $C(\chi_0|.|^s) \asymp |s|^{[\mathbb{F}:\mathbb{Q}]}$
since $\chi_0$ is quadratic.
\begin{theorem}\label{thm:holow}
  Let $f$ be a nondihedral holomorphic eigencuspform
  of weight $k = (k_1,\dotsc,k_{[\mathbb{F}:\mathbb{Q}]})$.
  If $\phi$ is a Maass eigencuspform, then
  \[
  \frac{\mu_f (\phi ) }{ \mu _f (1)}
  \ll_{\phi,\eps} \log(k^{\mathbf{1}})^\eps M_f(k^{\mathbf{1}})^{1/2}.
  \]
  If $\phi$ is an incomplete Eisenstein series, then
  \[
  \frac{\mu_f (\phi ) }{ \mu _f (1)}
  - \frac{\mu (\phi ) }{ \mu (1)}
  \ll_{\phi,\eps} \log(k^{\mathbf{1}})^\eps M_f(k^{\mathbf{1}})^{1/2}
  (1 + R_f(k^{\mathbf{1}})).
  \]
\end{theorem}
We prove Theorem \ref{thm:holow} in \S\ref{sec:key-arguments-our}
by combining the independent results of \S\ref{sec:reduct-shift-sums},
\S\ref{sec:bounds-an-integral}
and
\S\ref{sec:bounds-shifted-sums}; doing so is our main task in this paper.
Holowinsky
\cite[Cor 3]{holowinsky-2008} established
the case
$\mathbb{F} = \mathbb{Q}$ of Theorem \ref{thm:holow},
in which the ``nondihedral'' hypothesis
is vacuously satisfied.  We briefly recall his argument.
Take $\mathbb{F} = \mathbb{Q}$ and denote by $k$ the
weight of $f$.  Suppose for simplicity that $\phi$ is a Maass
eigencuspform.  Holowinsky defines for a fixed test function $h
\in C_c^\infty(\mathbb{R}_+^*)$ the integral
\[
S_l(Y) = \int_{y \in \mathbb{R}^*_+} h(Y y)
\int_{x \in \mathbb{R}/\mathbb{Z}} (\phi_l |f|^2)(x+iy)
\, \frac{d x \, d y}{y^2},
\]
where $\phi(z) = \sum_l \phi_l(z)$ with
$\phi_l(z+\xi) = e^{2 \pi i l \xi} \phi_l(z)$
for $\xi \in \mathbb{R} $,
and establishes \cite[Theorem 1]{holowinsky-2008}
for any $Y \geq 1$ and $\eps > 0$
the asymptotic formula
\begin{equation}\label{eq:review:holow-eval-sums}
  \frac{\int \phi |f|^2}{\int |f|^2}
  = c Y^{-1} \sum_{0 < |l| < Y^{1+\eps}} S_l(Y)
  + O_{\phi,\eps}(Y^{-1/2})
\end{equation}
where $c$ is an explicit nonzero constant depending only upon the test
function $h$; he shows moreover that
\begin{equation}\label{eq:review:holow-reduc-sums}
  \frac{S_l(Y)}{Y}
  \ll_{\phi,\eps} \frac{|\phi_l(a(Y^{-1}))|}{L(\ad f,1)}
  \left[ \frac{1}{Y k}
    \sum_{
      \substack{
        n  \in \mathbb{N} \\
        m := n+l \in \mathbb{N} 
      }
    }
    | \lambda_f(m) \lambda_f(n) |
    h \left( \frac{Y \left(\frac{k-1}{4 \pi}\right)}{\frac{m+n}{2}}
    \right)
    + \frac{(Y k)^\eps}{k}
  \right].
\end{equation}
He then proves
\cite[Theorem 2]{holowinsky-2008}
(in somewhat greater generality)
that for each $\eps \in (0,1)$,
each $x \gg_\eps 1$,
and each $l \in \mathbb{Z}$ for which
$0 \neq |l| \leq x$, we have
\begin{equation}\label{eq:review:holow-bound-sums}
  \sum_{n \leq x} |\lambda_f(m) \lambda_f(n)|
  \ll \tau(l) \frac{x}{\log(x)^{2-\eps}}
  \prod_{p \leq X} \left( 1 + \frac{2 |\lambda_f(p)|}{p} \right).
\end{equation}
From this he deduces the cuspidal case of Theorem
\ref{thm:holow} for $\mathbb{F} = \mathbb{Q}$.  We generalize
and refine (\ref{eq:review:holow-eval-sums}),
(\ref{eq:review:holow-reduc-sums}) and
(\ref{eq:review:holow-bound-sums}) in
\S\ref{sec:reduct-shift-sums},
\S\ref{sec:bounds-an-integral} and
\S\ref{sec:bounds-shifted-sums}, respectively; among other
refinements, we show that (a generalization to totally real
number fields of) the bound (\ref{eq:review:holow-bound-sums})
holds without the factor $\tau(l)$.  The main complication is
the manner in which these ingredients fit together to yield
Theorem \ref{thm:holow} when $\mathbb{F} \neq \mathbb{Q}$; this
is the crux of our argument, which we present in
\S\ref{sec:key-arguments-our}.  Specifically,
recall from \S\ref{sec:overview-proof} that for a totally real
number field $\mathbb{F}$ of degree $d =
[\mathbb{F}:\mathbb{Q}]$, our na\"{i}ve generalization of
(\ref{eq:review:holow-eval-sums}) and
(\ref{eq:review:holow-reduc-sums}) leaves us with the task of
showing that a sum of roughly $x \log(x)^{d-1}$ terms is small
relative to $x$ (with $x$ a bit larger than $k^{\mathbf{1}}$),
which seems beyond the limits of any method that does not
exploit cancellation in the sum of $\lambda_f(m) \lambda_f(n)$.
By discarding a large number of these terms trivially through a
refinement of (\ref{eq:review:holow-reduc-sums}), we reduce to
the more tractable
problem of showing that a sum of roughly $x \log(x)^{\eps}$ terms is small
relative to $x$.

\subsection{Soundararajan's independent arguments}\label{sec:sound-indep-argum}
Let $\phi$ be a Maass eigencuspform,
and suppose that $\mathbb{F} = \mathbb{Q}$.
Watson's formula \cite[Theorem 3]{watson-2008}
asserts that
\begin{equation}\label{eq:review:22}
  \left \lvert \frac{\mu_f(\phi)}{\mu_f(1)} \right \rvert^2
  = c(\mathbb{F},\phi)
  \frac{\Lambda(\phi \times f \times f,\tfrac{1}{2})}{\Lambda(\ad
    f,1)^2}
\end{equation}
where $c(\mathbb{Q},\phi) = \mu(|\phi|^2) / 8\Lambda(\ad
\phi,1)$
is a nonzero
constant unimportant for our purposes
and $\Lambda(\dotsb,s)$ is the completed $L$-function
for $L(\dotsb,s)$ with
local factors as in \cite[\S 3.1.1]{watson-2008}.
The identity
(\ref{eq:review:22}) with
$c(\mathbb{F},\phi) \neq 0$
holds for totally real $\mathbb{F}$
by Ichino's general triple product formula
\cite{MR2449948} together with Watson's
calculations of the local zeta integrals
of  Harris-Kudla \cite{harris-kudla-1991}
at the real places.
When $\mathbb{F} = \mathbb{Q}$,
Soundararajan \cite[Ex 2]{MR2680497}
proves that
\begin{equation}
  L(\phi \times f \times f,\tfrac{1}{2})
  \ll_{\phi,\eps} \frac{k^{\mathbf{1}} }{ \log (k^{\mathbf{1}})^{1-\eps}}.
\end{equation}
His argument applies verbatim when $\mathbb{F}$ is totally real:
it relies only upon the Ramanujan bound
for the local components of $f$ and the Rankin-Selberg
theory for $\phi \times \phi$,
noting that the analytic conductor of $\phi \times f \times f$
is $\asymp_{\phi} (k^{\mathbf{1}})^4$.
By Stirling's formula as in the $\mathbb{F} = \mathbb{Q}$
case, we obtain
\begin{equation}\label{eq:review:24}
  \frac{\int \phi |f|^2}{\int |f|^2}
  \ll_{\phi,\eps} \frac{\log (k^{\mathbf{1}})^{-1/2+\eps} }{L(\ad f,1)}.
\end{equation}

Now let $\phi = E(\chi,\cdot)$ be the unitary Eisenstein
series associated as in \S\ref{sec:eisenstein-series}
to an unramified idele class character
$\chi \in \mathfrak{X}(C_\mathbb{F}/\hat{\mathfrak{o}
}^*)(\onehalf)$ of real part $\onehalf$,
and suppose that $\mathbb{F} = \mathbb{Q} $.
(Since $C_{\mathbb{Q}} / \hat{\mathbb{Z} }^*
\cong \mathbb{R}_+^*$, we have
$\chi = |.|^{1/2+it}$ for some $t \in \mathbb{R}$.)
Soundararajan \cite[p7]{MR2680497} shows
by the unfolding method, Stirling's formula
and his weak subconvex bounds
for $L(\ad f, \chi)$ \cite[Ex 1]{MR2680497},
the last of which makes use of the known Ramanujan bound
for $f$,
that
\begin{equation}\label{eq:review:26}
  \frac{\mu_f(\phi)}{\mu_f(1)}
  \ll_\eps C(\chi)^2 \frac{\log( k^{\mathbf{1}})^{-1+\eps}}{L(\ad f,1)},
\end{equation}
and \cite[p2]{MR2680497}
\begin{equation}\label{eq:review:27}
  \left\lvert L(\ad f,\chi)
  \right\rvert
  \ll_\eps \frac{(k^{\mathbf{1}})^{1/2} C(\chi)^{3/4}}{\log(k^{\mathbf{1}})^{1-\eps}}.
\end{equation}
By the modularity of $L(\ad f,\chi)$ as the $L$-function
of an automorphic
form on $GL(3)$ \cite{MR533066}, its Rankin-Selberg
theory, and the lower bound
\begin{equation}
  L(\ad f,1) \gg \log(
  k^{\mathbf{1}})^{-1}
\end{equation}
due to
Hoffstein-Lockhart-Goldfeld-Hoffstein-Lieman \cite{HL94}
(which is available for general $\mathbb{F}$,
see \cite[\S 2.9]{2009arXiv0904.2429B}), Soundararajan
deduces \cite[Lem 1]{holowinsky-soundararajan-2008} in his
joint paper with Holowinsky that
\begin{equation}\label{eq:review:bound-for-R_k}
  R_f(k^{\mathbf{1}}) \ll_\eps
  \frac{\log (k^{\mathbf{1} })^\eps}{\log( k^{\mathbf{1}}) L(\ad f,1)}
  \ll \log( k^{\mathbf{1}})^{\eps}.
\end{equation}
The same argument establishes \eqref{eq:review:26}, \eqref{eq:review:27},
\eqref{eq:review:bound-for-R_k}
for general totally real number fields $\mathbb{F}$.

\subsection{The Holowinsky-Soundararajan synthesis}\label{sec:holow-sound-synth}
In their joint work \cite{holowinsky-soundararajan-2008},
Holowinsky and Soundararajan show \cite[Lem 3]{holowinsky-2008}
for $\mathbb{F} = \mathbb{Q}$ that
\begin{equation}\label{eq:review:Mk-bound}
  M_k(f) \ll \log (k^{\mathbf{1}})^{1/6} \log \log (k^{\mathbf{1}})^{9/2}
  L(\ad f,1)^{1/2},
\end{equation}
and their proof applies for general
$\mathbb{F}$.
Subsituting the bound \eqref{eq:review:Mk-bound} into
Theorem \ref{thm:holow}
and combining with
Soundararajan's estimate
\eqref{eq:review:24} yields
for each Maass eigencuspform $\phi$  that
\begin{equation}\label{eq:review:25}
  \frac{\mu_f(\phi)}{\mu_f(1)}
  \ll_{\phi,\eps} \min \left( \frac{\log(k^{\mathbf{1}})^{-1/2+\eps}}{ L(\ad f,1)},
    \log(k^{\mathbf{1}})^{1/12+\eps} L(\ad f,1)^{1/4} \right).
\end{equation}
It follows as in \cite[Proof of Thm
1]{holowinsky-soundararajan-2008} that $\mu_f(\phi) / \mu_f(1)
\ll_{\phi,\eps} \log(k^{\mathbf{1}})^{-1/30+\eps} = o(1)$, and the same
argument applies in the totally real case as soon as one has
established Theorem \ref{thm:holow}.

Holowinsky and Soundararajan show
\cite[p10]{holowinsky-soundararajan-2008}
that Soundararajan's bound
\eqref{eq:review:26} for unitary Eisenstein series
also applies to incomplete Eisenstein series
via the Mellin inversion formula.
Specifically, they show for $\mathbb{F} = \mathbb{Q}$
and $\phi = E(\Psi,\cdot)$
that
\begin{equation}\label{eq:review:28}
  \left\lvert
    \frac{\mu_f(\phi)}{\mu_f(1)}
    - \frac{\mu(\phi)}{\mu(1)}
  \right\rvert
  \ll_{\phi,\eps} \frac{\log(k^{\mathbf{1}})^{-1+\eps}}{L(\ad f,1)}.
\end{equation}
Their argument generalizes to the totally real case by replacing
the Mellin inversion on $\mathbb{R}_+^* \cong
C_{\mathbb{Q}}/\hat{\mathbb{Z} }^*$ with that on $C_\mathbb{F} /
\hat{\mathfrak{o} }^*$,
as we now describe.
Let $\Psi \in C_c^\infty(C_{\mathbb{F}}/\hat{\mathfrak{o} }^*)$
and $\phi = E(\Psi,\cdot)$.
By the Mellin formula (see \S\ref{sec:characters})
\[
\phi =
\int_{\mathfrak{X}(C_{\mathbb{F}}/\hat{\mathfrak{o} }^*)(2)}
\Psi^\wedge(\chi) E(\chi,\cdot) \, \frac{d \chi }{2 \pi i}
\]
and the meromorphic nature of $E(\chi,\cdot)$ (see \S\ref{sec:eisenstein-series}
or \cite{MR546600}),
we have
\begin{equation}\label{eq:review:expand-inc-eis}
  \begin{split}
    \mu_f(\phi)
    &=  \sum_{
      \chi_0 \in \mathfrak{X}(C_\mathbb{F}/\hat{\mathfrak{o} }^*)[2]
    }
    \Psi^\wedge(\chi_0)
    \res_{s=1} \mu_f(E(\chi_0 |.|^s,\cdot))
    \\
    & \quad + \int_{
      \mathfrak{X}(C_{\mathbb{F}}/\hat{\mathfrak{o} }^*)(1/2)
    }
    \Psi^\wedge(\chi)
    \mu_f(E(\chi,\cdot))
    \,
    \frac{d \chi }{2 \pi i },
  \end{split}
\end{equation}
where the interchanges here and those that follow are justified
by absolute convergence owing to the rapid decay of $f$ and
$\Psi$ and the moderate growth of $E(\chi,\cdot)$.  By the
unfolding method as in \S\ref{sec:masses}, the residue
$\res_{s=1} \mu_f(E(\chi_0 |.|^s,\cdot))$ coincides with
$\res_{s=1} \Lambda(\ad f, \chi_0 |.|^s)
\xi_{\mathbb{F}}(\chi_0|.|^s)$ up to a nonzero scalar.  Suppose
now that $f$ is nondihedral in the sense of 
\S\ref{sec:holom-eigenc},
so that $s \mapsto \Lambda(\ad f, \chi_0 |.|^s)$
is entire.  Then since $\xi_{\mathbb{F}}$ is
holomorphic away from its pole at $\chi = |.|$, we see that
$\res_{s=1} \mu_f(E(\chi_0 |.|^s,\cdot)) = 0$ if $\chi_0 \neq
1$.  If $\chi_0 = 1$, then
\[
\Psi^\wedge(|.|) \res_{s=1} \mu_f(E(|.|^s,\cdot))
= \mu_f(1) \Psi^\wedge(|.|) \res_{s=1} E(|.|^s,\cdot).
\]
We have $\Psi^\wedge(|.|) \res_{s=1} E(|.|^s,\cdot)
= \mu(\phi) / \mu(1)$ because both sides are equal
to the coefficient of the constant function $1$
in the spectral decomposition of $\phi \in L^2(\mathbf{X},\mu)$ \cite[\S 4]{MR546600}.
Thus for $f$ nondihedral, we obtain
\begin{equation}\label{eq:1a}
  \frac{\mu_f(\phi)}{\mu_f(1)}
  - \frac{\mu(\phi)}{\mu(1)}
  =
  \int_{
    \mathfrak{X}(C_{\mathbb{F}}/\hat{\mathfrak{o} }^*)(1/2)
  }
  \Psi^\wedge(\chi)
  \frac{\mu_f(E(\chi,\cdot))}{\mu_f(1)}
  \,
  \frac{d \chi }{2 \pi i }.
\end{equation}
Soundararajan's bound (\ref{eq:review:26})
for unitary Eisenstein series shows that the right-hand side
of (\ref{eq:1a}) is
\[
\ll_{\eps}
\int_{
  \mathfrak{X}(C_{\mathbb{F}}/\hat{\mathfrak{o} }^*)(1/2)
}
\left\lvert \Psi^\wedge(\chi)
  \frac{C(\chi)^2 \log(k^{\mathbf{1}})^{-1+\eps}}{L(\ad f, 1)}
\right\rvert
\, |d \chi|
\ll_{\phi}
\frac{\log(k^{\mathbf{1}})^{-1+\eps}}{L(\ad f, 1)},
\]
where in the final step we invoked the rapid decay of
$\Psi^\wedge$ (see \S\ref{sec:characters}).
Thus we obtain the estimate (\ref{eq:review:28})
for nondihedral forms over a totally real field.

By combining Holowinsky's Theorem \ref{thm:holow}
with Soundararajan's \eqref{eq:review:bound-for-R_k}
and \eqref{eq:review:28}, Holowinsky and Soundararajan obtain,
for $\mathbb{F} = \mathbb{Q}$ and $\phi = E(\Psi,\cdot)$,
the bound
\begin{equation}
  \left\lvert
    \frac{\mu_f(\phi)}{\mu_f(1)}
    - \frac{\mu(\phi)}{\mu(1)}
  \right\rvert
  \ll_{\phi,\eps}
  \min
  \left(
    \frac{\log (k^{\mathbf{1}})^{-1+\eps}}{L(\ad f,1)},
    \log (k^{\mathbf{1}})^{1/12+\eps} L(\ad f,1)^{1/4}
  \right),
\end{equation}
which is $o(1)$ (or even $\ll
\log(k^{\mathbf{1}})^{-2/15+\eps}$) by examination (see
\cite[Proof of Thm 1]{holowinsky-soundararajan-2008}).  The same
estimate follows in the totally real case as soon as one has
established Theorem \ref{thm:holow}.

\section{The key arguments in our generalization}\label{sec:key-arguments-our}
We saw in \S\ref{sec:brief-revi-holow}
that our main result Theorem \ref{thm:1}
follows from the generalization of Holowinsky's
work asserted by Theorem \ref{thm:holow}.
We now describe the key arguments
that reduce our proof of Theorem \ref{thm:holow}
to several technical results that we shall prove
in the remaining sections of this paper;
those results are independent of one another and do not depend
upon any work in this section, so there is no circularity
in our discussion.

Recall that Theorem \ref{thm:holow}
claims to bound $\mu_f(\phi)/\mu_f(1) - \mu(\phi)/\mu(1)$,
for $f$ a nondihedral holomorphic eigencuspform
of weight $k$
and $\phi$ either a Maass eigencuspform
or an incomplete Eisenstein series,
in terms of certain quantities $M_f(k^{\mathbf{1}})$
and $R_f(k^{\mathbf{1}})$ \eqref{eq:review:mf-defn}--\eqref{eq:review:rf-defn}.
\begin{definition}\label{defn-shifted-sums}
  Fix a nonnegative test function $h \in
  C_c^\infty(\mathbb{R}_+^*)$ with Mellin transform
  \[h^\wedge(s)
  = \int_0^\infty h(y) y^{-s} \, d^\times y
  \]
  normalized so that
  $h^\wedge(1) \res_{s=1} E(s,\cdot) = 1$.
  Recall from \S\ref{sec:number-fields} that we have fixed
  representatives $\mathfrak{z}_j = \div z_j$ for the narrow
  class group of $\mathbb{F}$; here $j \in \{1, \dotsc,
  h(\mathbb{F})\}$ and $z_j \in \mathbb{A}_f^*$.
  For each unramified idele
  class character $\chi \in
  \mathfrak{X}(C_{\mathbb{F}}/\hat{\mathfrak{o} } ^\ast )$
  and each $x \geq 2$,
  define the \emph{shifted sums}
  \begin{equation}
    S_\chi(x) =
    \sum_{j=1}^{h(\mathbb{F})}
    \sum
    _{
      \substack{
        l \in \mathfrak{o} _+ ^\ast \backslash \mathfrak{z} _j \\
        0 \neq |l ^{\mathbf{1}}| < x ^{1 + \eps }
      }
    }
    \frac{
      \lambda_{\chi}(\mathfrak{z}_j^{-1} l)
    }
    {
      \norm(\mathfrak{z}_j^{-1} l)^{1/2}
    }
    S_{\chi_\infty}(\mathfrak{z}_j,l,x),
  \end{equation}
  where
  \begin{equation}\label{eq:11}
    S_{\chi_\infty}(\mathfrak{z},l,x) =
    \sum _{
      \substack{
        n \in \mathfrak{z}
        \cap \mathbb{F}_{\infty+}^* \\
        m := n + l \in \mathfrak{z}
        \cap \mathbb{F}_{\infty+}^* \\
      }}
    \frac{
      \lambda_{f}(\mathfrak{z}^{-1} m)
    }
    {
      \norm(\mathfrak{z}^{-1} m)^{1/2}
    }
    \frac{
      \lambda_{f}(\mathfrak{z}^{-1} n)
    }
    {
      \norm(\mathfrak{z}^{-1} n)^{1/2}
    }
    \frac{I_{\chi_\infty}(l,n,\norm(\mathfrak{z}) x)}{\norm(\mathfrak{z})},
  \end{equation}
  and (here $m := n + l$ as always)
  \begin{equation}\label{eq:10}
    I_{\chi_\infty}(l,n,x)
    =
    \frac{(4 \pi \mathbf{1})^{k-\mathbf{1}}}{\vecgamma(k-\mathbf{1})}
    \int_{\mathbb{F}_{\infty+}^*}
    h(x y^1) \kappa_{\chi,\infty}(l y)
    \kappa_{f,\infty}(m y)
    \kappa_{f,\infty}(n y) \,
    \frac{d^\times y}{y^{\mathbf{1}}}.
  \end{equation}
  If $\phi$ is a Maass eigencuspform of eigenvalue
  $(\tfrac{1}{4}+r_1^2,\dotsc,\tfrac{1}{4}+r_{[\mathbb{F}:\mathbb{Q}]}^2)$ and
  parity $(\eps_1,\dotsc,\eps_{[\mathbb{F}:\mathbb{Q}]})$,
  define analogously $S_\phi(x)$,
  $S_{\phi_\infty}(\mathfrak{z},l,x)$ and
  $I_{\phi_\infty}(l,n,x)$ by replacing $\kappa_{\chi,\infty}$
  and $\lambda_\chi$ with
  $\kappa_{\phi,\infty}$ and $\lambda _\phi $ above; note then that
  $S_{\phi_\infty}(\mathfrak{z},l,x)$ is the special case
  of $S_{\chi_\infty}(\mathfrak{z},l,x)$ obtained by
  taking $\chi_\infty$ to be the
  (conceivably non-unitary)
  character $[y \mapsto \prod \sgn(y_j)^{\eps_j} |y_j|^{i r_j}]
  \in \mathfrak{X}(\mathbb{F}_{\infty}^*)$ as in \eqref{eq:33}.
\end{definition}

\begin{proposition}\label{thm:reduction}
  Let $f$ be as in the statement of Theorem \ref{thm:holow}
  and let $Y \geq 1$.
  If $\phi$ is a Maass eigencuspform,
  then
  \[
  \frac{\mu_f(\phi)}{\mu_f(1)}
  =
  \frac{c_1(\mathbb{F})}{L(\ad f,1)}
  \frac{  S_\phi(Y)}{(k-\mathbf{1})^{\mathbf{1}} Y}
  + O_{\phi,\eps} (Y ^{- 1/2}).
  \]
  If $\phi = E(\Psi,\cdot)$ is an incomplete
  Eisenstein series (recall that $f$ is not dihedral),
  then
  \begin{equation*}
    \begin{split}
      \frac{\mu_f(\phi)}{\mu_f(1)}
      - \frac{\mu(\phi)}{\mu(1)}
      &=
      \frac{c_1(\mathbb{F})}{L(\ad f,1)}
      \int _{\mathfrak{X}(C_{\mathbb{F}}/\hat{\mathfrak{o} } ^\ast )(0)}
      \frac{
        \Psi^\wedge(|.|^{1/2}\chi)
      }{
        \xi _{\mathbb{F} } (|.|\chi^2 )    \chi(d_{\mathbb{F}})^{-2}
      }
      \frac{  S_\chi(Y) }{(k-\mathbf{1})^{\mathbf{1}} Y }
      \, \frac{d \chi }{2 \pi i} \\
      &\quad 
      + O_{\phi,\eps} \left( \frac{1 + R_f(k^{\mathbf{1}})}{Y ^{
            1/2}} \right).
    \end{split}
  \end{equation*}
  The constant $c_1(\mathbb{F})$ is as in the formula
  \eqref{eq:26}.
\end{proposition}
\begin{proof}
  See \S\ref{sec:reduct-shift-sums}.
  The proof is a straightforward
  and na\"{i}ve generalization of Holowinsky's
  arguments in the $\mathbb{F}=\mathbb{Q}$ case.
\end{proof}
Proposition \ref{thm:reduction} shows that Theorem \ref{thm:holow} follows from
sufficiently strong bounds for the shifted sums $S_\phi(Y)$
for $\phi$ a Maass eigencuspform and $S_\chi(Y)$ for $\chi
\in \mathfrak{X}(C_{\mathbb{F}}/\hat{\mathfrak{o} }^*)(0)$ an
unramified unitary idele class character.

We bound the sums $S_\phi(Y)$ and $S_\chi(Y)$
by bounding their summands
$S_{\chi_\infty}(\mathfrak{z},l,x)$
for each narrow ideal class representative $\mathfrak{z} =
\mathfrak{z}_j$ ($j \in \{1, \dotsc, \mathbb{F} \}$),
each nonzero shift
$l \in \mathfrak{z} \cap \mathbb{F}^*$,
and each character $\chi_{\infty} \in \mathfrak{X}(\mathbb{F}_{\infty}^*)$;
recall from Definition \ref{defn-shifted-sums}
that
\begin{equation}\label{eq:34}
  S_{\phi_\infty}(\mathfrak{z},l,x)
  = S_{\chi_\infty}(\mathfrak{z},l,x)
\end{equation} for a suitable
character $\chi_{\infty} \in
\mathfrak{X}(\mathbb{F}_{\infty}^*)$.
For this reason it suffices
to bound $S_{\chi_\infty}(l,n,x)$
when
$\chi_{\infty}$ is either unitary
or
of the form \eqref{eq:33} for some Maass eigencuspform $\phi$,
so that in particular each $r_j \in \mathbb{R} \cup
i(-\onehalf,\onehalf)$;
we assume henceforth that this is the case.

The sums $S_{\chi_\infty}(\mathfrak{z},l,x)$
are weighted by an integral $I_{\chi_\infty}(l,n,x)$,
which we treat as follows.
By the Mellin formula $h(y) =
\int_{(c)} h^\wedge(s) y^s \, \tfrac{d s}{2 \pi i}$ with
$h^\wedge(s) = \int_0^\infty  h(y) y^{-s} \, d^\times y$
and $c \geq 0$, we may factor
$I_{\chi_\infty}(l,n,x)$ as a product of local integrals
\begin{equation}\label{eq:32}
  I_{\chi_\infty}(l,n,x)
  = \int_{(c)} h^\wedge(s) x^s
  \left(
    \prod_{j=1}^{[\mathbb{F}:\mathbb{Q}]}
    J_{i r_j}(l_j,n_j,s)
  \right)
  \frac{d s}{2 \pi i},
\end{equation}
where
\[
J_{i r_j}(l_j,n_j,s)
:=
\frac{(4 \pi)^{k_j-1}}{\Gamma(k_j-1)}
\int_{\mathbb{R}_+^*}
y^{s-1} \kappa_{\chi,\infty_j}(l_j y)
\kappa_{f,\infty_j}(m_j y)
\kappa_{f,\infty_j}(n_j y)
\, d^\times y.
\]
The ``trivial'' bound for $J_{i r_j}$
obtained by applying the inequality
$|\kappa_{\chi,\infty_j}(l_j y)| \leq 1$
to the integrand
and evaluating the resulting gamma integral
is
\begin{equation}\label{eq:j-bound-trivial}
  |J_{i r_j}(l_j,n_j,s)|
  \leq
  \frac{\Gamma(k_j - 1+\sigma)}{\Gamma(k_j-1)}
  \frac{
    \sqrt{m_j n_j}
  }
  {
    \left(4 \pi \left( \frac{m_j + n_j }{2} \right) \right)^\sigma
  }
  \left( \frac{\sqrt{m_j n_j}}{ \left( \frac{m_j + n_j}{2}
      \right)} \right)^{k_j-1},
\end{equation}
where $s = \sigma + i t$.
However, \eqref{eq:j-bound-trivial} would \emph{not suffice}
for our purposes, as we shall explain after
proving the following refinement.
\begin{lemma}\label{lem:j-bound-legit}
  For $i r_j \in i \mathbb{R} \cup
  (-\tfrac{1}{2},\tfrac{1}{2})$,
  $l_j \neq 0$, $n_j > 0$,
  $m_j = n_j + l_j > 0$,
  $k_j \geq 2$,
  and $s = \sigma + i t$
  with $\sigma \geq -\tfrac{1}{2}$, we have
  \begin{equation}\label{eq:j-bound-legit}
    |J_{i r_j}(l_j,n_j,s)|
    \leq
    \frac{\Gamma(k_j - 1 + \sigma)}{\Gamma(k_j-1)}
    \frac{
      \sqrt{m_j n_j}
    }
    {
      \left(4 \pi \max(m_j,n_j) \right)^\sigma
    }
    \left(
      \frac{\min(m_j,n_j)}{\max(m_j,n_j)}
    \right)^{\frac{k_j-1}{2}}.
  \end{equation}
\end{lemma}
\begin{proof}
  By the integral formula \cite[6.621.3]{GR} and the
  transformation formula \cite[9.131]{GR} in Gradshteyn-Ryzhik,
  we have explicitly
  \begin{equation}\label{eq:huge-equation}
    \begin{split}
      J_{i r_j}(l_j,n_j,s)
      &=
      \pm
     \frac{\Gamma(k_j - 1 + s)}{\Gamma(k_j-1)}
      \frac{
        \sqrt{m_j n_j}
      }
      {
        \left(4 \pi \max(m_j,n_j) \right)^s
      }
      \left(
        \frac{\min(m_j,n_j)}{\max(m_j,n_j)}
      \right)^{\frac{k_j-1}{2}} \\
      &\quad
      \cdot
      \frac{
        \Gamma(k_j+s-\tfrac{1}{2}+i r_j)
        \Gamma(k_j+s-\tfrac{1}{2}-i r_j)
      }{
        \Gamma(k_j+s-1)
        \Gamma(k_j+s)
      } \\
      &\quad
      \cdot 
      {}_2 F_1
      \left( \efrac{\tfrac{1}{2} - i r_j, \tfrac{1}{2} + i r_j}{
          k_j + s} ; - \frac{\min(m_j,n_j)}{|m_j-n_j|}  \right)
    \end{split}      
  \end{equation}
  where ${}_2 F_1$ is the Gauss hypergeometric function
  and the sign is given by
  $\prod \sgn(l_j)^{\eps_j}$.
  By the technical lemmas proved
  in \S\ref{sec:bounds-an-integral},
  the factors on the second and third lines
  of \eqref{eq:huge-equation}
  are each bounded in absolute value by $1$,
  so the claim follows from the basic inequality
  $|\Gamma(k_j-1+s)| \leq \Gamma(k_j-1+\sigma)$.
\end{proof}

\begin{corollary}\label{prop:bound-integral}
  Let $\chi_\infty \in \mathfrak{X}(\mathbb{F}_{\infty}^*)$ be
  of the form \eqref{eq:33} with each $i r_j \in i \mathbb{R}
  \cup (-\onehalf,\onehalf)$.  Then
  \begin{equation}\label{eq:14}
    I_{\chi_\infty}(l,n,x)
    \ll_{A}
    \sqrt{m^{\mathbf{1}} n^{\mathbf{1}}}
    \left(
      \frac{\mathbf{min}(m,n)}{\mathbf{max}(m,n)}
    \right)^{\frac{k-\mathbf{1}}{2}}
    \min \left( 1, \frac{k^{\mathbf{1}} x}{\mathbf{max}(m,n)^{\mathbf{1}}} \right)^{A}.
  \end{equation}
\end{corollary}
\begin{proof}
  Substitute \eqref{eq:j-bound-legit}
  into \eqref{eq:32}, taking $c \in \{0,A\}$ and invoking
  the well known estimate $\Gamma(k_j-1+\sigma) / \Gamma(k_j-1)
  \ll_\sigma k_j^{\sigma}$ \cite[Ch 7,
  Misc. Ex 44]{MR0178117}.
\end{proof}
\begin{remark}\label{rmk:super-refined}
  With more effort
  (e.g., by studying the asymptotics
  of the expression \eqref{eq:huge-equation})
  one can show that if the
  components of the weight $k$ increase in such a way that
  $\min(k_1,\dotsc,k_{[\mathbb{F}:\mathbb{Q}]}) \gg
  (k^{\mathbf{1}})^{\delta_0}$ for some $\delta_0 > 0$, then
  (setting $\mathbf{log}(x) =
  (\log x_1,\dotsc,\log x_{[\mathbb{F}:\mathbb{Q}]})$
  for $x \in \mathbb{F}_{\infty+}^* \cong (\mathbb{R}_+^*)^{[\mathbb{F}:\mathbb{Q}]}$)
  \begin{equation*}
    \begin{split}
      I_{\chi_\infty}(l,n,x)
      = \sqrt{m^{\mathbf{1}} n^{\mathbf{1}}}
      &\left[
        \kappa_{\chi,\infty}
        \left( \frac{k - \mathbf{1} }{4 \pi }
          \left\lvert \mathbf{log} \frac{m}{n} \right\rvert
        \right)
        h \left( \frac{x \left( \frac{k - \mathbf{1} }{4 \pi }
            \right)^{\mathbf{1}}}{\mathbf{max}(m,n)^{\mathbf{1}}}
        \right)
      \right.
      \\
      &
      \left.
        \quad
        \, + \, \,
        O_{\chi_\infty}
        \left(
          (k^{\mathbf{1}})^{-\delta_0} \left( \frac{k^{\mathbf{1}}
              x}{\mathbf{max}(m,n)^{\mathbf{1}}} 
          \right)^{1+\eps}
        \right)
      \right].
    \end{split}
  \end{equation*}
  It follows with some work
  that for $\phi$ a Maass
  eigencuspform and $Y \geq 1$, we have
  \begin{equation*}
    \begin{split}
      \frac{\mu_f(\phi)}{\mu_f(1)}
      =
      O_{\phi}(Y^{-1/2}) +
      \frac{
        c_1(\mathbb{F})
      }{
        k^{\mathbf{1}} Y L(\ad f,1)
      }
      &
      \sum_{j=1}^{h(\mathbb{F})}
      \sum_{
        \substack{
          l \in \mathfrak{o} _+ ^\ast \backslash \mathfrak{z} _j \\
          0 \neq |l ^{\mathbf{1}}| < Y ^{1 + \eps }
        }
      }
      \frac{
        \lambda_{\phi}(\mathfrak{z}_j^{-1} l)
      }
      {
        \norm(\mathfrak{z}_j^{-1} l)^{1/2}
      }
      \\
      &\cdot 
      \sum_{
        \substack{
          n \in \mathfrak{z}
          \cap \mathbb{F}_{\infty+}^* \\
          m := n + l \in \mathfrak{z}
          \cap \mathbb{F}_{\infty+}^* \\
        }
      }
      \lambda_{f}(\mathfrak{z}^{-1} m)
      \lambda_{f}(\mathfrak{z}^{-1} n)
      \\
      &\cdot 
      \kappa_{\phi,\infty}\left( \frac{k - \mathbf{1} }{4 \pi }
        \left\lvert \mathbf{log} \frac{m}{n} \right\rvert
      \right)
      \frac{h \left( \frac{Y \norm(\mathfrak{z}) \left( \frac{k - \mathbf{1} }{4 \pi }
            \right)^{\mathbf{1}}}{\mathbf{max}(m,n)^{\mathbf{1}}} \right)
      }
      {
        \norm(\mathfrak{z})
      }.
    \end{split}
  \end{equation*}
  This refinement is not necessary for our
  purposes, so we omit the proof; the simpler upper bound given by
  Corollary \ref{prop:bound-integral} suffices because we do
   not exploit cancellation in the shifted sums, and has the
  advantage of being completely uniform in $\chi_{\infty}$.
\end{remark}

\begin{corollary}\label{cor:bound-shifted-sums}
  Let $\chi_\infty \in \mathfrak{X}(\mathbb{F}_{\infty}^*)$
  satisfy the hypotheses of Corollary
  \ref{prop:bound-integral}.
  Then the shifted sums
  $S_{\chi_\infty}(\mathfrak{z},l,Y)$ are bounded
  up to a multiple depending only upon $\mathfrak{z}$ and $A$
  by the quantity
  \begin{equation}\label{eq:24}
    \sum _{
      \substack{
        n \in \mathfrak{z}
        \cap \mathbb{F}_{\infty+}^* \\
        m := n + l \in \mathfrak{z}
        \cap \mathbb{F}_{\infty+}^* \\
      }}
    \left\lvert 
      \lambda_{f}(\mathfrak{z}^{-1} m)
      \lambda_{f}(\mathfrak{z}^{-1} n)
    \right\rvert
    \left(
      \frac{\mathbf{min}(m,n)}{\mathbf{max}(m,n)}
    \right)^{\frac{k-\mathbf{1}}{2}}
    \min \left( 1, \frac{k^{\mathbf{1}}
        Y}{\mathbf{max}(m,n)^{\mathbf{1}}} \right)^{A}.
  \end{equation}
\end{corollary}
\begin{proof}
  Substitute Corollary \ref{prop:bound-integral}
  into Definition \ref{defn-shifted-sums}.
\end{proof}

\begin{remark}\label{rmk:comparison}
  When $\mathbb{F} = \mathbb{Q}$, Holowinsky
  applies what amounts to the trivial bound
  \eqref{eq:j-bound-trivial}, which gives something like
  \eqref{eq:24}
  upon replacing
  \begin{equation}\label{eq:silly-factor-comparison}
    \left(
      \frac{\mathbf{min}(m,n)}{\mathbf{max}(m,n)}
    \right)^{\frac{k-\mathbf{1} }{2}}
    =
    \prod_{j=1}^{[\mathbb{F}:\mathbb{Q}]}
    \left(
      \frac{\min(m_j,n_j)}{\max(m_j,n_j)}
    \right)^{\frac{k_j-1}{2}}
    \text{ by }
    \prod_{j=1}^{[\mathbb{F}:\mathbb{Q}]}
    \left(
      \frac{\sqrt{m_j n_j}}{\left(\frac{m_j+n_j}{2}\right)}
    \right)^{k-\mathbf{1}}.
  \end{equation}
  He then bounds the factor on the RHS of
  \eqref{eq:silly-factor-comparison}
  by $1$.
  Now, bounding either of the factors
  in \eqref{eq:silly-factor-comparison} is harmless
  when
  $\mathbb{F}  = \mathbb{Q}$:
  if $f$ has weight $k$,
  then in the sum
  \eqref{eq:24}
  we typically have $m,n \asymp k Y$,
  so
  for $|l| = O(1)$ both factors
  in \eqref{eq:silly-factor-comparison} are typically
  $\asymp 1$.
  On the other hand, when $d = [ \mathbb{F} : \mathbb{Q}] > 1$
  it is costly to apply such bounds prematurely:
  the sum \eqref{eq:24}
  then has roughly $x \log(x)^{d-1}$
  nonnegligible
  terms
  with $x = k^{\mathbf{1}} Y$, and this extra logarithmic factor
  ``$\log(x)^{d-1}$''
  turns out to be unaffordable in the application
  to mass equidistribution.
  One can show that the savings obtained by treating nontrivially
  the factor on the RHS of \eqref{eq:silly-factor-comparison}
  are negligible even for $d > 1$.
  Thus the success of our
  method when $\mathbb{F} \neq \mathbb{Q}$
  depends crucially on the more careful treatment afforded by Corollary
  \ref{prop:bound-integral}.
  In fact,
  the key to our whole
  argument is that the factor
  on the LHS of \eqref{eq:silly-factor-comparison}
  is very small if \emph{any component} of $\mathbf{max}(m,n)$ is not too
  large,
  as we quantify in Lemma \ref{lem:sum-dyadic}.
\end{remark}

\begin{definition}\label{defn:essential-sums}
  Given parameters
  $T = (T_1, \dotsc, T_d) \in \mathbb{R}_{\geq 1}^{[\mathbb{F}:\mathbb{Q}]}$
  and $U \in \mathbb{R}_{\geq 1}$,
  let
  \begin{equation*}
    \mathcal{R}_{T,U} =
    \left\{ x \in \mathbb{R}^{[\mathbb{F}:\mathbb{Q}]}:
      x^{\mathbf{1}} \leq
      T^{\mathbf{1}},
      x \geq  T/U
    \right\}
  \end{equation*}
  be the subregion of $\mathbb{R}_{>0}^{[\mathbb{F}:\mathbb{Q}]}$
  bounded by the hyperbola
  $\{ \prod x_i = \prod T_i\}$
  and the hyperplanes $\{x_i = T_i/U\}$.
  For a multiplicative
  function $\lambda : I_\mathbb{F} \rightarrow \mathbb{C}$,
  an ideal $\mathfrak{z}$ in $\mathbb{F}$ and
  an element $l \in \mathfrak{z}$,
  let
  \begin{equation}\label{eq:12}
    \Sigma_\lambda(\mathfrak{z},l,T,U)
    := \sum
    _{
      \substack{
        n \in \mathfrak{z} \\
        m := n + l \in \mathfrak{z} \\
        \mathbf{max}(m,n) \in \mathcal{R}_{T,U}
      }
    }
    \lvert \lambda (\mathfrak{z} ^{-1} m ) \lambda(\mathfrak{z}
    ^{-1} n)
    \rvert.
  \end{equation}
\end{definition}

\begin{lemma}\label{lem:sum-dyadic}
  Let $\chi \in \mathfrak{X}(\mathbb{F}_{\infty+}^*)$ satisfy
  the hypotheses of Corollary \ref{prop:bound-integral},
  let
  \[
  d = [\mathbb{F}:\mathbb{Q}],
  \quad
  T = (T_1,\dotsc,T_d)
  \text{ with }T_i = k_i Y^{1/d},
  \quad
  X = T_1 \dotsc T_d = k^{\mathbf{1}} Y,
  \]
  and let $U = \exp(\log(X)^\eps)$.
  Suppose that  $1 \leq Y \ll \log(X)^{O(1)}$.
  Then for any ideal $\mathfrak{z}$, any nonzero
  shift $l \in \mathfrak{z} \cap \mathbb{F}^*$,
  and any positive integer $A$, we have
  \begin{equation}\label{eq:sum-dyadic}
    S_{\chi_\infty}(l,n,Y)
    \ll_{\mathfrak{z},A} X^{-A} + \sum_{r=0}^\infty
    2^{- r d A}
    \Sigma_{\lambda_f}(\mathfrak{z},l,2^{r+1} T, 2^{r+1} U).
  \end{equation}
\end{lemma}
\begin{proof}
  We work with the bound asserted by Corollary
  \ref{cor:bound-shifted-sums}.  Partition those $m,n$ in
  \eqref{eq:24} for which $\mathbf{max}(m,n) \geq T/U$ according to the
  least integer $r \geq 0$ such that $\mathbf{max}(m,n)^{\mathbf{1}} \leq 2^r
  X$; their contribution is bounded by the second term on the RHS
  of \eqref{eq:sum-dyadic}.  It remains to consider those $m,n$
  for which
  \begin{equation}\label{eq:13}
    \max(m_i,n_i) \leq T_i/U
  \end{equation}
  for some index $i \in \{1,\dotsc,d\}$.
  The elementary inequality $1
  - x \leq \exp(-x)$ and the tautology $\mathbf{min}(m,n) + |l| =
  \mathbf{max}(m,n)$ show that
  \begin{equation*}
    \left(
      \frac{\mathbf{min}(m,n)}{\mathbf{max}(m,n)}
    \right)^{\frac{k-\mathbf{1}}{2}}
    \leq \exp \left(
      - \sum_{j=1}^d \frac{k_j-1}{2} \frac{|l_j|}{\max(m_j,n_j)}
    \right),
  \end{equation*}
  so the assumption \eqref{eq:13} implies
  \begin{equation}\label{eq:16}
    \left(
      \frac{\mathbf{min}(m,n)}{\mathbf{max}(m,n)}
    \right)^{\frac{k-\mathbf{1} }{2}}
    \leq \exp
    \left( - \frac{|l_i| U}{3 Y^{1/d}} \right).
  \end{equation}
  Here we may and shall assume that the shift $l$ is balanced in
  the sense that $|l_i| \asymp_{\mathfrak{z}} |l_j|$ for all $i, j
  \in \{1, \dotsc, \mathbb{F} \}$ since $S_{\chi_\infty}(\eta
  l,n,Y)= S_{\chi_\infty}(l,n,Y)$ for any totally positive unit
  $\eta \in \mathfrak{o}_+^*$; in particular, we may assume that there
  exists a positive number $c$, depending only upon the fixed
  number field $\mathbb{F}$ and the fixed set of representatives
  $\{\mathfrak{z}_1, \dotsc, \mathfrak{z}_{h(\mathbb{F})}\}$ for
  the narrow class group, such that $|l_i| \geq c$ for each $i$.
  Since $Y \ll \log(X)^{O(1)}$ by assumption,
  our choice $U= \exp(\log(X)^\eps)$ is (more than) large enough
  that
  for each positive real
  $A$ the inequality
  \[
  \frac{c U}{3 Y^{1/d}}
  \geq  A \log(X)
  \]
  holds eventually
  (i.e., for $\max(k_1,\dotsc,k_d) \gg 1$),
  so by \eqref{eq:16} we obtain
  \begin{equation}\label{eq:15}
    \left(
      \frac{\mathbf{min}(m,n)}{\mathbf{max}(m,n)}
    \right)^{\frac{k-\mathbf{1} }{2}}
    \ll_A X^{-A}.
  \end{equation}
  By the trivial ``Hecke'' bound $\lambda_f(\mathfrak{a}) \ll
  \norm(\mathfrak{a})^{1/2+\eps}$, the contribution to
  \eqref{eq:24} of $n$ satisfying
  \eqref{eq:13} is
  \begin{align}\label{eq:25}
    \nonumber &\ll
    X^{-A'}
    \sum _{
      \substack{
        n \in \mathfrak{z}
        \cap \mathbb{F}_{\infty+}^* \\
        m := n + l \in \mathfrak{z}
        \cap \mathbb{F}_{\infty+}^* \\
      }}
    \left\lvert 
      \lambda_{f}(\mathfrak{z}^{-1} m)
      \lambda_{f}(\mathfrak{z}^{-1} n)
    \right\rvert
    \min \left( 1, \frac{X}{\mathbf{max}(m,n)^{\mathbf{1}}} \right)^{A}  \\
    &\ll
    X^{-A'}
    \sum _{
      \substack{
        n \in \mathfrak{z}
        \cap \mathbb{F}_{\infty+}^* \\
        m := n + l \in \mathfrak{z}
        \cap \mathbb{F}_{\infty+}^* \\
      }}
    (m^{\mathbf{1}} n^{\mathbf{1}})^{1/2+\eps}
    \min \left( 1, \frac{X}{\mathbf{max}(m,n)^{\mathbf{1}}} \right)^{A}
  \end{align}
  for any $A, A' > 0$.  Since $|l|_i \geq c$, the number of $n
  \in \mathfrak{z} \cap \mathbb{F}_{\infty+}^*$ for which $n + l
  \in \mathfrak{z} \cap \mathbb{F}_{\infty+}^*$ and
  $\mathbf{max}(m,n)^{\mathbf{1}} \leq 2^r X$ ($r \geq 0$) is $\ll (2^r
  X)^{d}$.  Choosing $A = 1+2 \eps +
  d + 1$, summing dyadically, and taking
  $A'$ to be sufficiently large, we see that \eqref{eq:25} is
  $\ll_{A''} X^{-A''}$ for any positive constant $A''$,
  as desired.
\end{proof}

The volume of $\mathcal{R}_{T,U}$ is approximately $X
\log(U)^{d-1} = X \log(X)^{(d-1)\eps}$.  Since the number of
nonnegligible terms appearing in $S_{\chi_\infty}(l,n,Y)$ is approximately
$X \log(X)^{d-1}$, we see that Lemma \ref{lem:sum-dyadic}
allows us to discard the vast majority of those terms.
We treat the remaining
$\approx X \log(X)^{\eps '}$ terms by the following
generalization of Holowinsky's
bound for shifted sums
of multiplicative functions \cite[Thm 2]{holowinsky-2008}.
\begin{theorem}\label{thm:essential-sums}
  Let $T \in \mathbb{R}_{\geq 1}^{[\mathbb{F}:\mathbb{Q}]}$,
  $U \in \mathbb{R}_{\geq 1}$, $\mathfrak{z}$, $l$ and $\lambda
  : I_{\mathbb{F}} \rightarrow \mathbb{C}$ be as in
  Definition \ref{defn:essential-sums}.  Suppose that $l \neq 0$
  and that $|\lambda(\mathfrak{a})| \leq \tau(\mathfrak{a})$ for
  all integral ideals $\mathfrak{a}$.
  Set $X = T^{\mathbf{1}}$ and $d =
  [\mathbb{F}:\mathbb{Q}]$.  Then
  \begin{equation}\label{eq:21}
    \Sigma_{\lambda}(\mathfrak{z},l,T,U)
    \ll_{\mathfrak{z},\eps}
    \frac{\log(eU)^{d-1} X}{\log(e X)^{2-\eps}}
    \prod _{
      \substack{
        \norm (\mathfrak{p}) \leq X \\
        % \mathfrak{p} \sim \mathfrak{z}^{-1}
      }
    }
    \left( 1
      + \frac{2 |\lambda(\mathfrak{p})|}{\norm( \mathfrak{p}) } \right).
  \end{equation}
  Here the product is taken over prime ideals of
  norm at most $X$.
\end{theorem}
\begin{proof}
  See \S\ref{sec:bounds-shifted-sums}.
\end{proof}
\begin{remark}
  Holowinsky \cite[Thm 2]{holowinsky-2008} established a
  slightly weaker form of the case $d = 1$ of Theorem
  \ref{thm:essential-sums} by an application of the large sieve;
  in his inequality \eqref{eq:35} an additional factor of
  $\tau(l)$ appears on the RHS.  We prove Theorem
  \ref{thm:essential-sums} by adapting his approach, with the
  only difficulty being that the regions $\mathcal{R}_{T,U}$ are
  shaped quite differently when $d > 1$.

  If one is willing to sacrifice uniformity in the shift $l$,
  then alternate proofs of the corresponding
  weakening of Holowinsky's \cite[Thm
  2]{holowinsky-2008} and (probably) our Theorem
  \ref{thm:essential-sums} can be obtained by the general
  estimates due to Nair \cite{MR1197420} and Nair-Tenenbaum
  \cite{MR1618321} for sums $\sum_n \lambda(|P(n)|)$ with $P$ a
  (primitive, possibly multivariate) polynomial (for example,
  $P(n) = n(n+l)$) and $n$ traversing a box;
  note that in all of the
  bounds asserted by Nair and Nair-Tenenbaum, the implied
  constants depend in an unspecified manner upon the
  discriminant and degree of $P$.  This seems insufficient in
  the application to QUE where the shift $l$ must vary
  (particularly when $\phi$ is an incomplete Eisenstein series,
  see \cite{aws2010sound}).

  We refer to \cite[Rmk 3.11]{PDN-HQUE-LEVEL} for a further
  discussion of variations on the $d=1$ case of
  Theorem \ref{thm:essential-sums} that may be derived from other works
  and particularly their applicability to QUE in the level
  aspect.
  % By working through the details of
  % Nair's paper we have shown that his arguments allow one to
  % deduce Holowinsky's bound \eqref{eq:35} uniformly for $0 \neq
  % |l| < x^{1/16-\eps}$ upon replacing the factor $\tau(l)$ on
  % the RHS of \eqref{eq:35} by $\tau_m(l)$ for some $m \geq 2$
  % (probably $m=2$); such uniformity suffices for the application
  % to QUE in the weight aspect, but not in the level aspect
  % \cite{PDN-HQUE-LEVEL}, where the shift $l$ can be nearly as
  % large as $x$ and can have many divisors.
\end{remark}

\begin{proof}[Proof of Theorem \ref{thm:holow}]
  Let $Y \geq 1$ be a parameter
  (to be chosen at the end of the proof)
  that satisfies $Y \ll \log(k^{\mathbf{1}})^{O(1)}$.
  Preserve the hypotheses and notation $d=[\mathbb{F}:\mathbb{Q}]$,
  $T = Y^{1/d} k$, $X = T^{\mathbf{1}} = k^{\mathbf{1}} Y$
  and $U = \exp(\log(X)^\eps)$ from above.
  Lemma \ref{lem:sum-dyadic}
  and Theorem \ref{thm:essential-sums} show that
  \begin{equation}
    S_{\chi_\infty}(l,n,Y)
    \ll_{A,\eps}
    X^{-A}
    + \sum_{r=0}^{\infty} 2^{-r d A}
    \frac{\log(2^r e U)^{d-1} 2^{r d} X}{ \log(2^{r d} X)^{2-\eps}}
    \prod _{
      \substack{
        \norm (\mathfrak{p}) \leq 2^r X \\
        % \mathfrak{p} \sim \mathfrak{z}^{-1}
      }
    }
    \left( 1
      + \frac{2 |\lambda_f(\mathfrak{p})|}{\norm( \mathfrak{p}) }
    \right).
  \end{equation}
  Taking $A = 2$ and using that
  \[
  \sum_{r=0}^\infty
  2^{r d -r d A}
  \log(2^r e U)^{d-1}
  \prod _{
    \substack{
      X < \norm (\mathfrak{p}) \leq 2^r X \\
      % \mathfrak{p} \sim \mathfrak{z}^{-1}
    }
  }
  \left( 1
    + \frac{4}{\norm( \mathfrak{p}) } \right)
  \ll_\eps \log(X)^{(d-1) \eps}
  \]
  gives
  \[
  S_{\chi_\infty}(l,n,Y)
  \ll_\eps
  \frac{X}{\log(X)^{2-\eps'}}
  \prod _{
    \substack{
      \norm (\mathfrak{p}) \leq X \\
      % \mathfrak{p} \sim \mathfrak{z}^{-1}
    }
  }
  \left( 1
    + \frac{2 |\lambda_f(\mathfrak{p})|}{\norm( \mathfrak{p}) }
  \right),
  \]
  where $\eps ' = d \eps$.
  Thus
  \begin{equation}
    S_\phi(Y)
    \ll_{\phi,\eps}
    \frac{k^{\mathbf{1}}
      Y^{3/2+\eps}}{\log(k^{\mathbf{1}})^{2-\eps'} }
    \prod_{\norm(\mathfrak{p}) \leq k^{\mathbf{1}}}
    \left( 1 + \frac{2 \lvert \lambda_f(\mathfrak{p}) \rvert}{\norm(\mathfrak{p})} \right),
  \end{equation}
  since the sum over $l$
  in
  Definition
  \ref{defn-shifted-sums}
  introduces the additional factor
  \[
  \sum_{
    \substack{
      0 \neq \mathfrak{a} \subset \mathfrak{o} \\
      \norm(\mathfrak{a}) < Y^{1+\eps}
    }
  }
  \frac{|\lambda_\phi(\mathfrak{a})|}{\norm(\mathfrak{a})^{1/2}}
  \leq 
  \left(
    \sum_{
      \substack{
        0 \neq \mathfrak{a} \subset \mathfrak{o} \\
        \norm(\mathfrak{a}) < Y^{1+\eps}
      }
    }
    | \lambda_\phi(\mathfrak{a})|^2
    \sum_{
      \substack{
        0 \neq \mathfrak{b} \subset \mathfrak{o} \\
        \norm(\mathfrak{b}) < Y^{1+\eps}
      }
    }
    \frac{1}{\norm(\mathfrak{b})}
  \right)^{1/2}
  \ll_\phi Y^{1/2+\eps}
  \]
  by the Cauchy-Schwarz inequality and the Rankin-Selberg bound
  \eqref{eq:27}; similarly, using that
  $|\lambda_\chi(\mathfrak{a})| \leq \tau(\mathfrak{a})$
  for a \emph{unitary} character
  $\chi \in \mathfrak{X}(C_{\mathbb{F}}/\hat{\mathfrak{o}
  }^*)(0)$, 
  we find that
  \begin{equation}
    S_\chi(Y) \ll_{\eps}
    \frac{k^{\mathbf{1}}
      Y^{3/2+\eps}}{\log(k^{\mathbf{1}})^{2-\eps'} }
    \prod_{\norm(\mathfrak{p}) \leq k^{\mathbf{1}}}
    \left( 1 + \frac{2 \lvert \lambda_f(\mathfrak{p}) \rvert}{\norm(\mathfrak{p})} \right),
  \end{equation}
  where we emphasize that the implied constant does not depend
  upon $\chi$.
  By Proposition \ref{thm:reduction}
  and the definitions \eqref{eq:review:mf-defn}--\eqref{eq:review:rf-defn}
  of $M_f(x)$ and $R_f(x)$,
  we deduce for $\phi$ a Maass eigencuspform that
  \begin{equation}\label{eq:29}
    \frac{\mu_f(\phi)}{\mu_f(1)}
    \ll_{\phi,\eps}
    Y^{1/2+\eps} \log(k^{\mathbf{1}})^{\eps'} M_f(k^{\mathbf{1}})
  \end{equation}
  and for $\phi = E(\Psi,\cdot)$ an incomplete Eisenstein series
  that
  \begin{equation}\label{eq:28}
    \begin{split}
      \frac{\mu_f(\phi)}{\mu_f(1)} - \frac{\mu(\phi)}{\mu(1)}
      &\ll_{\phi,\eps}
      Y^{1/2+\eps}
      \log(k^{\mathbf{1}})^{\eps '} M_f(k^{\mathbf{1}})
      \int_{\mathfrak{X}(C_{\mathbb{F}}/\hat{\mathfrak{o}}^*)(0)}
      \left\lvert \frac{\Psi^\wedge(|.|^{1/2}
          \chi)}{\xi_{\mathbb{F}}(|.|^1 \chi^2)}
      \right\rvert
      \, |d \chi | \\
      &\quad 
      + \frac{1 + R_f(k^{\mathbf{1}})}{Y^{1/2}}.
    \end{split}
  \end{equation}
  The integral in \eqref{eq:28} converges by the rapid decay of
  $\Psi^\wedge$ (see \S\ref{sec:characters}).
  Choosing (as Holowinsky does) $Y = \max(1,M_f(k^{\mathbf{1}})^{-1})
  \ll \log(k^{\mathbf{1}}) ^{O (1) }$ in \eqref{eq:29}
  and \eqref{eq:28},
  we conclude the proof
  of Theorem \ref{thm:holow}.
\end{proof}

\section{Reduction to shifted sums weighted by an
  integral}\label{sec:reduct-shift-sums}
In this section we establish Proposition \ref{thm:reduction},
which reduces our study of $\mu_f(\phi)$ to that of the shifted
sums $S_\phi(Y)$ and $S_\chi(Y)$; here and throughout this
section $Y \geq 1$ is a (small) parameter, $f$ is a nondihedral
holomorphic eigencuspform of weight $k =
(k_1,\dotsc,k_{[\mathbb{F}:\mathbb{Q}]})$, $\phi$ is a Maass
eigencuspform or incomplete Eisenstein series, and $h \in
C_c^\infty(\mathbb{R}_+^*)$ is a fixed test function with Mellin
transform $h^\wedge(s) = \int_0^\infty h(y) y^{-s} \, d^\times
y$ normalized as in Definition \ref{defn-shifted-sums} so that
\begin{equation}\label{eq:h-normalization}
  h^\wedge(1) \res_{s=1} E(s,\cdot) = 1.
\end{equation}
Let $h_Y$ be the
function $y \mapsto h(Y y)$ and
let
\[
E(h_Y,\cdot) :
G(\mathbb{A}) \ni g \mapsto \sum_{\gamma \in B(\mathbb{F})
  \backslash G(\mathbb{F})} h_Y(|y(\gamma g)|)
\]
be the incomplete
Eisenstein series attached by the recipe of \S\ref{sec:incompl-eisenst-seri}
to the test function $h_Y \circ |.|
\in C_c^\infty(C_{\mathbb{F}} / C_{\mathbb{F}}^1 )
\hookrightarrow C_c^\infty(C_{\mathbb{F}} / \hat{\mathfrak{o}
}^*)$.

\begin{lemma}\label{lem:5.1}
  We have the approximate formula
  \[
  \frac{\mu_f(\phi)}{\mu_f(1)}
  = \frac{\mu_f(E(h_Y,\cdot) \phi)}{Y \mu_f(1)}
  + O_\phi(Y^{-1/2}).
  \]
\end{lemma}
\begin{proof}
  The starting point is the consequence
  \begin{equation}\label{eq:starting-pt}
    \mu_f(E(h_Y,\cdot) \phi) = Y \mu_f(\phi) + \int_{(1/2)}
    h_Y^\wedge(s) \mu_f(E(s,\cdot) \phi) \, \frac{d s}{2 \pi i},
  \end{equation}
  of Mellin inversion, Cauchy's
  theorem and our normalization (\ref{eq:h-normalization}).
  We need a crude bound of the form
  \begin{equation}\label{eq:easy-bound-for-eis}
    E(s,g) \phi(g) \ll_\phi
    |s|^{2[\mathbb{F}:\mathbb{Q}]+\eps}
    \quad \text{ for $\Re(s) = \onehalf, g \in G(\mathbb{A})$,}
  \end{equation}
  where the precise exponent is not important.  To establish
  this, recall first that if $c > 0$ is chosen small enough,
  then the Siegel set $\mathfrak{S}$ consisting of those $g =
  n(x) a(y) k z \in G(\mathbb{A})$ for which $|y| \geq c$
  satisfies $G(\mathbb{A}) = G(\mathbb{F}) \mathfrak{S}$.  Since
  $E(s,\cdot) \phi$ is $Z(\mathbb{A})$-invariant and right
  $K$-invariant, it suffices to establish
  (\ref{eq:easy-bound-for-eis}) for $g = n(x) a(y \times
  z_j^{-1})$ where $x \in \mathbb{A}$, $y \in
  \mathbb{F}_{\infty+}^*$ with $y^{\mathbf{1}} \geq c$ and $j
  \in \{1, \dotsc, h(\mathbb{F})\}$.  For $s = \onehalf + it$ the
  Fourier expansion of $E(s,\cdot)$, given in 
  \S\ref{sec:eisenstein-series},
  shows
  that
  \begin{equation}
    |E(s,n(x) a(y \times z_j^{-1})))| \ll
    (y^{\mathbf{1}})^{1/2}
    + \sum_{n \in \mathbb{F}^* \cap \mathfrak{z}_j}
    \left\lvert
      \frac{
        \kappa_{i t,\infty}( n y)
      }
      {
        \xi_{\mathbb{F}}(1 + 2 i t)
      }
      \frac{\lambda_{i t}(\mathfrak{z}_j^{-1} n)
      }{
        \norm(\mathfrak{z}_j^{-1} n)^{1/2}
      }
    \right\rvert,
  \end{equation}
  where for simplicity we write $\kappa_{i t,\infty} :=
  \kappa_{|.|^{it},\infty}$ and $\lambda_{it} :=
  \lambda_{|.|^{it}}$.  The straightforward analysis of
  \cite[\S3.6]{MR882550} applied to $\zeta_\mathbb{F}$ in
  place of $\zeta_\mathbb{Q}$ shows that\footnote{We believe
    that the stronger bound with $(1+|t|)^\eps$ replaced by
    $\log(1+|t|)$ holds, but could not quickly locate a
    reference.}
  \[
  \xi_\mathbb{F}(1 + 2 i t)^{-1} \ll
  \frac{
    (1+|t|)^\eps
  }
  {
    \Gamma_\mathbb{R}(1+2it)^{[\mathbb{F}:\mathbb{Q}]}
  },
  \]
  and it is noted in \cite[page 6]{holowinsky-2008}
  that the integral formula for $K_{i t}$ implies
  \[
  \frac{K_{i t}(y)}{\Gamma_\mathbb{R}(1+2it)}
  \ll
  \left( \frac{1 + |t|}{y} \right)^A
  \left( 1 + \frac{1 + |t|}{y} \right)^\eps
  \quad \text{ for any } A \in \mathbb{Z}_{\geq 0},
  \, \eps>0,
  \]
  thus (writing $d = [\mathbb{F}:\mathbb{Q}]$,
  $\eps ' = (d+1) \eps$,
  and using that $|n^{\mathbf{1}}| y^{\mathbf{1}} \gg 1$)
  \[
  \left\lvert
    \frac{
      \kappa_{i t,\infty}( n y)
    }
    {
      \xi_{\mathbb{F}}(1 + 2 i t)
    }
    \frac{\lambda_{i t}(\mathfrak{z}_j^{-1} n)
    }{
      \norm(\mathfrak{z}_j^{-1} n)^{1/2}
    }
  \right\rvert
  \ll (y^{\mathbf{1}})^{1/2}
  (1+|t|)^{2 d + \eps'}
  \frac{|n^{\mathbf{1}}|^\eps }{(\mathbf{max}(\mathbf{1},|n| y)^{\mathbf{1}})^{A}}.
  \]
  Take $A = 2$.
  We have
  \begin{equation}\label{eq:convergent-sum-of-n}
    \sum_{n \in \mathbb{F}^* \cap \mathfrak{z}_j}
    \frac{|n^{\mathbf{1}}|^{\eps}}{(\mathbf{max}(\mathbf{1},|n|
      y)^{\mathbf{1}})^{2}}
    \ll (y^{\mathbf{1}})^{-2}
  \end{equation}
  because the LHS of (\ref{eq:convergent-sum-of-n})
  is invariant under multiplying $y$
  by an element of $\mathfrak{o}_+^*$, so we may
  assume that $y$ is balanced ($y_i \asymp y_j$
  for all $i,j$) with each component bounded uniformly
  from below, in which case (\ref{eq:convergent-sum-of-n})
  may be compared with a convergent integral.
  Thus $|E(s,n(x) a(y \times z_j^{-1}))|
  \ll (y^{\mathbf{1}})^{1/2}
  + |s|^{2 d + \eps '} (y^{\mathbf{1}})^{-3/2}$.
  Since $\phi$ satisfies\footnote{For a Maass eigencuspform,
    this is well known \cite[Prop
    10.7]{MR0401654}; an incomplete
    Eisenstein
    series vanishes off a compact subset of $\mathbf{X}$.}
  $\phi(n(x) a(y \times z_j^{-1})) \ll_\phi (y^{\mathbf{1}})^{-A}$,
  we obtain the crude bound (\ref{eq:easy-bound-for-eis}).
  
  By the rapid decay of $h^\wedge$
  and the identity $h_Y^\wedge(s) = Y^s h^\wedge(s)$, we deduce
  from (\ref{eq:easy-bound-for-eis})
  that the error term in (\ref{eq:starting-pt})
  satisfies
  \[
  \int_{(1/2)}
  h_Y^\wedge(s) \mu_f(E(s,\cdot) \phi) \, \frac{d s}{2 \pi i}
  \ll Y^{1/2} \mu_f(1).
  \]
  The lemma follows upon dividing through by $Y \mu_f(1)$.
\end{proof}

Fix now a nice fundamental domain
$[\mathbb{F}_{\infty+}^* / \mathfrak{o}_+^*]$
for the quotient
$\mathbb{F}_{\infty+}^* / \mathfrak{o}_+^*$
with the property that 
$y \in [\mathbb{F}_{\infty+}^* / \mathfrak{o}_+^*]$
implies $y_i \asymp y_j$ for all $i,j \in \{1, \dotsc,
[\mathbb{F}:\mathbb{Q}]\}$.
Write the Fourier expansions of $\phi$ and $f$
in the form
\begin{equation}
  \phi = \sum_{l \in \mathbb{F}} \phi_l,
  \quad f = \sum_{n \in \mathbb{F}^*} f_n,
\end{equation}
where $\phi_l : G(\mathbb{A}) \rightarrow \mathbb{C}$ satisfies
$\phi_l(n(x) g) = e_{\mathbb{F}}(l x) \phi_l(g)$ for all $x \in
\mathbb{A}$ and $f_n$ satisfies the analogous condition.
\begin{lemma}\label{lem:unfold}
  We have $\mu_f(E(h_Y,\cdot) \phi)
  = \mathcal{S}_0 + \mathcal{S}_1 + \mathcal{S}_2$,
  where
  \begin{equation}\label{eq:s0-defn}
    \mathcal{S}_0
    =
    \sum_{j=1}^{h(\mathbb{F})}
    \int_{y \in [\mathbb{F}_{\infty+}^*/\mathfrak{o}_+^*]}
    \frac{h_Y(y^{\mathbf{1}}
      \norm(\mathfrak{z}_j))}{\norm(\mathfrak{z}_j)}
    \int_{x \in \mathbb{F} \backslash \mathbb{A}}
    (\phi_0 |f|^2)(n(x) a(y \times z_j^{-1})) \, d x \,
    \frac{d^\times y}{y^{\mathbf{1}}};
  \end{equation}
  for $\phi$ a Maass eigencuspform,
  \[
  \mathcal{S}_1
  =
  \frac{\vecgamma(k-\mathbf{1})}{(4 \pi
    \mathbf{1})^{k-\mathbf{1}}}
  S_\phi(Y);
  \]
  for $\phi = E(\Psi,\cdot)$ an incomplete Eisenstein series,
  \[
  \mathcal{S}_1
  =
  \frac{\vecgamma(k-\mathbf{1})}{(4 \pi
    \mathbf{1})^{k-\mathbf{1}}}
  \int_{\mathfrak{X}(C_{\mathbb{F}}/\hat{\mathfrak{o} }^*)(0)}
  \frac{\Psi^\wedge(|.|^{1/2} \chi )}{
    \xi_{\mathbb{F}}(|.|
    \chi^2)
    \chi(d_{\mathbb{F}})^{-2}
  }
  S_\chi(Y)
  \, \frac{d \chi }{ 2 \pi i };
  \]
  and
  \begin{equation}\label{eq:s2-initial-bound}
    |\mathcal{S}_2|
    \leq \mu_f(E(h_Y,\cdot))
    \sum_{j=1}^{h(\mathbb{F})}
    \sup _{
      \substack{
        y \in [\mathbb{F}_{\infty+}^* / \mathfrak{o}_+^*] \\
        h_Y(y^{\mathbf{1}} N(\mathfrak{z}_j)) \neq 0
      }
    }
    \sum _{
      \substack{
        l \in \mathfrak{z} _j \\
        |l ^{\mathbf{1} }| \geq Y ^{1 + \eps }
      }
    }
    |\phi_l(a(y \times z_j^{-1}))|.
  \end{equation}
  The shifted sums $S_\phi(Y)$ and $S_\chi(Y)$
  are as in Definition \ref{defn-shifted-sums}.
\end{lemma}
\begin{proof}
  By the formula \eqref{eq:31}
  for integration over $Z(\mathbb{A}) B(\mathbb{F}) \backslash
  G(\mathbb{A})$,
  we see that
  \begin{equation}\label{eq:42}
    \begin{split}
      &\mu_f(E(h_Y,\cdot) \phi) \\
      &\quad =
      \sum_{j=1}^{h(\mathbb{F})}
      \int_{y \in \mathbb{F}_{\infty+}^*/\mathfrak{o}_+^*}
      \frac{h_Y(y^{\mathbf{1}}
        \norm(\mathfrak{z}_j))}{\norm(\mathfrak{z}_j)}
      \int_{x \in \mathbb{F} \backslash \mathbb{A}}
      (\phi |f|^2)(n(x) a(y \times z_j^{-1})) \, d x \,
      \frac{d^\times y}{y^{\mathbf{1}}}.
    \end{split}
  \end{equation}
  We now integrate
  in $y$ over the fundamental domain
  $[\mathbb{F}_{\infty+}^* / \mathfrak{o}_+^*]$
  and substitute for $\phi$ its Fourier series
  $\sum \phi_l$.
  Note that $\phi_l(n(x) a(y \times z_j^{-1}))
  = 0$ unless $l \in \mathfrak{z}_j$.
  The contribution to (\ref{eq:42})
  of the constant term $\phi_0$ is precisely $\mathcal{S}_0$.
  Let $\mathcal{S}_2$ denote the contribution of those
  $\phi_l$ for which $|l^{\mathbf{1}}| \geq Y^{1+\eps}$,
  so that the bound (\ref{eq:s2-initial-bound})
  follows from the formula for $\mu_f(E(h_Y,\cdot))$
  given by (\ref{eq:42}) with $\phi = 1$.
  Let $\mathcal{S}_1$ denote the remaining
  contribution of those
  $l \in \mathfrak{z}_j$ for which
  $0 \neq |l^{\mathbf{1}}| < Y^{1+\eps}$.
  Substituting the Fourier series $f = \sum f_n$
  (in which $f_n(y \times z_j^{-1}) = 0$
  unless $n \in \mathfrak{z}_j \cap \mathbb{F}_{\infty+}^*$)
  and integrating in $x$, we obtain
  \begin{equation}\label{eq:36}
    \mathcal{S}_1
    =
    \sum_{j=1}^{h(\mathbb{F})}
    \mathop{\sum \sum}_{
      \substack{
        (l,n) \in (\mathbb{F}^* \cap
        \mathfrak{z}_j)^2 \\
        l^{\mathbf{1}} < Y^{1+\eps} \\
        n \in \mathbb{F}_{\infty+}^* \\
        m := n+l \in \mathbb{F}_{\infty+}^* \\
      }
    }
    \int_{y \in [\mathbb{F}_{\infty+}^*/\mathfrak{o}_+^*]}
    \frac{
      h_Y(y^{\mathbf{1}}
      \norm(\mathfrak{z}_j))
    }
    {
      \norm(\mathfrak{z}_j)
    }
    (\phi_l \overline{f_m} f_n)(a(y \times z_j^{-1}))
    \, \frac{d^\times y}{y^{\mathbf{1}}}.
  \end{equation}
  If $\eta \in \mathfrak{o}_+^*$,
  then $(\phi_{\eta l} \overline{f_{\eta m}} f_{\eta n})(a(y
  \times z_j^{-1}))
  = (\phi_{ l} \overline{f_{ m}} f_{ n})(a(\eta y
  \times z_j^{-1}))$ (see \S\ref{sec:fourier-expansions}),
  so we may break the sum into orbits
  for $(l,n)$ under the diagonal action of $\mathfrak{o}_+^*$
  and unfold the integral over $y$ to
  all of $\mathbb{F}_{\infty+}^*$:
  \begin{equation}\label{eq:37}
    \mathcal{S}_1
    =
    \sum_{j=1}^{h(\mathbb{F})}
    \mathop{\sum \sum}_{
      \substack{
        (l,n) \in \mathfrak{o}_+^* \backslash (\mathbb{F}^* \cap
        \mathfrak{z}_j)^2 \\
        l^{\mathbf{1}} < Y^{1+\eps} \\
        n \in \mathbb{F}_{\infty+}^* \\
        m := n+l \in \mathbb{F}_{\infty+}^* \\
      }
    }
    \int_{y \in \mathbb{F}_{\infty+}^*}
    \frac{
      h_Y(y^{\mathbf{1}}
      \norm(\mathfrak{z}_j))
    }
    {
      \norm(\mathfrak{z}_j)
    }
    (\phi_l \overline{f_m} f_n)(a(y \times z_j^{-1}))
    \, \frac{d^\times y}{y^{\mathbf{1}}}.
  \end{equation}
  Take as representatives
  for $\mathfrak{o}_+^* \backslash (\mathbb{F}^* \cap
  \mathfrak{z}_j)^2$
  the pairs $(l,n)$ with $l$ traversing
  any set of representatives for
  $\mathfrak{o}_+^* \backslash (\mathbb{F}^* \cap
  \mathfrak{z}_j)$
  and $n$ traversing
  the set $\mathbb{F}^* \cap \mathfrak{z}_j$.
  Recalling the formulas for $f_n$ and $\phi_l$
  given in \S\ref{sec:holom-eigenc},
  \S\ref{sec:maass-eigencuspforms}
  and \S\ref{sec:incompl-eisenst-seri}  
  and the definitions of $S_\phi(Y)$ and $S_\chi(Y)$,
  we obtain the claimed expressions
  for $\mathcal{S}_1$.
\end{proof}

\begin{lemma}
  We have
  \[
  \frac{\mathcal{S}_0}{Y \mu_f(1)}
  = \frac{\mu(\phi)}{\mu(1)}
  + O_\phi
  \left( \frac{1 + \delta_\phi R_f(k^{\mathbf{1}})}{Y^{1/2}} \right),
  \]
  where $\delta_\phi = 0$ or $1$ according
  as $\phi$ is a Maass eigencuspform
  and or an incomplete Eisenstein series.
\end{lemma}
\begin{proof}
  If $\phi$ is cuspidal, then
  $\mathcal{S}_0 = \mu(\phi) = 0$,
  so there is nothing to show.
  Suppose that $\phi = E(\Psi,\cdot)$.
  If $y^{\mathbf{1}} \asymp Y^{-1}$, then it follows from
  \eqref{eq:phi0-useful-formula-0}
  that
  \begin{equation}\label{eq:phi0-useful-formula}
    \phi_0(y \times z_j^{-1})
    = \frac{\mu(\phi)}{\mu(1)}
    + \sum_{1 \neq \chi_0 \in
      \mathfrak{X}(C_{\mathbb{F}}/\hat{\mathfrak{o}}^*)[2]}
    c_{\Psi}(\chi_0)
    \chi_0(y \times z_j^{-1})
    + O_{\phi}(Y^{-1/2}).
  \end{equation}
  We have
  \begin{equation}\label{eq:integral-ehy-1}
    \begin{split}
      &\sum_{j=1}^{h(\mathbb{F})}
      \int_{y \in [\mathbb{F}_{\infty+}^*/\mathfrak{o}_+^*]}
      \frac{h_Y(y^{\mathbf{1}}
        \norm(\mathfrak{z}_j))}{\norm(\mathfrak{z}_j)}
      \int_{x \in \mathbb{F} \backslash \mathbb{A}}
      |f|^2(n(x) a(y \times z_j^{-1})) \, d x \,
      \frac{d^\times y}{y^{\mathbf{1}}} \\
      &\quad =
      \mu_f(E(h_Y,\cdot))
      = \int_{(2)} h_Y^\wedge(s) \mu_f(E(s,\cdot)) \, \frac{d s}{2 \pi i},
    \end{split}
  \end{equation}
  and similarly for $1 \neq \chi_0  \in
  \mathfrak{X}(C_{\mathbb{F}}/\hat{\mathfrak{o}}^*)[2]$,
  \begin{equation}\label{eq:integral-ehy-2}
    \begin{split}
      &\sum_{j=1}^{h(\mathbb{F})}
      \int_{y \in [\mathbb{F}_{\infty+}^*/\mathfrak{o}_+^*]}
      \frac{h_Y(y^{\mathbf{1}}
        \norm(\mathfrak{z}_j))}{\norm(\mathfrak{z}_j)}
      \chi_0(y \times z_j^{-1})
      \int_{x \in \mathbb{F} \backslash \mathbb{A}}
      |f|^2(n(x) a(y \times z_j^{-1})) \, d x \,
      \frac{d^\times y}{y^{\mathbf{1}}} \\
      &\quad
      = \int_{(2)} h_Y^\wedge(s) \mu_f(E(|.|^s \chi_0,\cdot)) \,
      \frac{d s}{2 \pi i}.
    \end{split}
  \end{equation}
  Substituting
  (\ref{eq:phi0-useful-formula}) into (\ref{eq:s0-defn})
  and applying \eqref{eq:integral-ehy-1} and \eqref{eq:integral-ehy-2},
  we obtain
  \begin{equation}\label{eq:s0-contour-integral}
    \begin{split}
      \mathcal{S}_0
      &= \left( \frac{\mu(\phi)}{\mu(1)}
        + O_\phi(Y^{-1/2}) \right)
      \int_{(2)} h_Y^\wedge(s) \mu_f(E(s,\cdot)) \, \frac{d s}{2
        \pi i} \\
      &\quad + \sum_{1 \neq \chi_0  \in
        \mathfrak{X}(C_{\mathbb{F}}/\hat{\mathfrak{o}}^*)[2]}
      c_\Psi(\chi_0)
      \int_{(2)} h_Y^\wedge(s) \mu_f(E(|.|^s \chi_0,\cdot)) \,
      \frac{d s}{2 \pi i}.
    \end{split}
  \end{equation}
  Shift the contours in (\ref{eq:s0-contour-integral}) to the
  line $\Re(s) = \onehalf$; for $\chi_0 \neq 1$ we do not pick up
  a pole of $\mu_f(E(|.|^s \chi_0, \cdot))$
  because $f$ is nondihedral.
  Thus
  \begin{equation}\label{eq:s0-with-unsimplified-error}
    \begin{split}
      \mathcal{S}_0
      &= Y \mu_f(1) \left( \frac{\mu(\phi)}{\mu(1)}
        + O_\phi(Y^{-1/2}) \right) \\
      &\quad 
      + O_\phi 
      \left( \sum_{\chi_0 \in
          \mathfrak{X}(C_{\mathbb{F}}/\hat{\mathfrak{o}}^*)[2]}
        \int_{(1/2)}
        \left\lvert
          h _Y ^\wedge (s) \mu _f (E (\chi _0 |.|^s,\cdot))
        \right\rvert
        \, |d s|
      \right).
    \end{split}
  \end{equation}
  To simplify the error term, we apply the formula
  \begin{equation}
    \begin{split}
      &\frac{\mu_f(E(\chi_0 |.|^s,\cdot))}{\mu_f(1)}
      \\
      &\quad 
      = c_1(\mathbb{F})
      \int _{(1/2)}
      h^\wedge(s)
      \left( \frac{Y }{ 4 \pi ^{[\mathbb{F} :
            \mathbb{Q} ]}} \right) ^s 
      \frac{\vecgamma (k + (s - 1 ) \mathbf{1} )}{ \vecgamma (k )}
      \frac{\zeta_\mathbb{F} (\chi_0 |.|^s) }{\zeta _{\mathbb{F} } (2 s )}
      \frac{L (\ad f, \chi_0 |.|^s)}{ L (\ad f, 1 )}
      \, \frac{d s}{2 \pi i}
    \end{split}
  \end{equation}
  which follows from the unfolding method and analytic
  continuation as in \S\ref{sec:masses}.  By the standard estimates
  $|\Gamma(k_j - \tfrac{1}{2} + i t)| \leq \Gamma(k_j -
  \tfrac{1}{2}) \ll k_j^{-1/2} \Gamma(k_j)$,
  $\zeta_{\mathbb{F}}(\chi_0 |.|^s) \ll
  |s|^{[\mathbb{F}:\mathbb{Q}]/4}$ and $|\zeta_\mathbb{F}(2 s)|
  \gg |s|^{-\eps}$ for $\Re(s) = \onehalf$ (see also Soundararajan's
  arguments \cite[p7]{MR2680497} when $\mathbb{F} =
  \mathbb{Q}$), we deduce that the error term in
  (\ref{eq:s0-with-unsimplified-error}) satisfies
  \begin{equation}
    \sum_{\chi_0 \in
      \mathfrak{X}(C_{\mathbb{F}}/\hat{\mathfrak{o}}^*)[2]}
    \int_{(1/2)}
    \left\lvert
      h _Y ^\wedge (s) \mu _f (E (\chi _0 |.|^s,\cdot))
    \right\rvert
    \, |d s|
    \ll Y^{1/2} \mu_f(1)
    R_f(k^{\mathbf{1}}),
  \end{equation}
  with $R_f$ given by \eqref{eq:review:rf-defn}.
  The lemma follows upon dividing
  through by $Y \mu_f(1)$.
\end{proof}

\begin{lemma}\label{lem:bound-s2}
  We have
  \[
  \frac{|\mathcal{S}_2|}{Y \mu_f(1)} \ll Y^{-10}.
  \]
\end{lemma}
\begin{proof}
  Set $d = [\mathbb{F}:\mathbb{Q}]$,
  and note that
  each $l$ arising in the sum
  (\ref{eq:s2-initial-bound})
  satisfies
  \begin{equation}\label{eq:l-max-dyadic}
    2^r (Y^{1+\eps})^{1/d}
    \leq \max(|l_1|,\dotsc,|l_d|)
    < 
    2^{r+1} (Y^{1+\eps})^{1/d}
  \end{equation}
  for some nonnegative integer $r$.
  More generally,
  there are $\ll 2^{r d} Y^{1+\eps}$ elements $l \in
  \mathfrak{z}_j$ for which (\ref{eq:l-max-dyadic})
  holds.
  For each $y \in [\mathbb{F}_{\infty+}^* / \mathfrak{o}_+^*]$
  such that $h_Y(y^{\mathbf{1}} N(\mathfrak{z}_j)) \neq 0$,
  we have $y^{\mathbf{1}} \asymp Y^{-1}$
  and $y_i \asymp y_j$ for $i,j \in \{1,\dotsc,
  [\mathbb{F}:\mathbb{Q}]\}$,
  thus
  \begin{equation}\label{eq:y-each-component-small}
    y_i \asymp Y^{-1/d} \quad \text{ for each $i$.}
  \end{equation}

  Suppose that  $\phi$ is
  a Maass eigencuspform, so that
  \[\phi_l(a(y \times z_j^{-1})) =
  \kappa_{\phi,\infty}(l y)
  \frac{
    \lambda_{\phi}(l z_j^{-1})
  }{
    \norm(l z_j^{-1})^{1/2}
  }.\]
  We
  have  $\lambda_\phi(\mathfrak{a}) \leq \tau(\mathfrak{a}) \norm(\mathfrak{a})^{1/2} \ll \norm(\mathfrak{a})^{1/2+\eps}$
  and
  $\kappa_{\phi,\infty}(l y) = \prod_{i=1}^d
  \kappa_{\phi,\infty_i}(l_i y_i)$ with 
  \[
  \kappa_{\phi,\infty_i}(l_i y_i)
  =
  \pm 2 (|l_i| y_i)^{1/2} K_{i r_i}(2 \pi |l_i| y_i),
  \]
  where
  $|\kappa_{\phi,\infty_i}(l_i y_i)| \leq 1$
  and
  \begin{equation}\label{eq:k-bessel-bound}
    K_{i r}(x) \ll \left( \frac{1 + |r|}{x} \right)^{A'}
    \quad \text{uniformly for $r \in \mathbb{R} \cup i(-\onehalf,\onehalf)$
      and $x \geq \delta > 0$.}
  \end{equation}
  Thus if   $l \in \mathfrak{z}_j$
  and $y \in \mathbb{F}_{\infty+}^*$ satisfy
  (\ref{eq:l-max-dyadic})--(\ref{eq:y-each-component-small}),
  we obtain
  \begin{equation}\label{eq:bound-phi-l}
    \lvert \phi _l (a (y \times z_j^{-1})) \rvert
    \ll
    (1 + |r|^{\mathbf{1}})^{O(1)}
    (2^r Y^{\eps/d})^{-A}
  \end{equation}
  for any positive $A$.  The dependence of the bound
  (\ref{eq:bound-phi-l}) on $\phi$ is polynomial in the
  archimedean parameters $r_i$, so
  (\ref{eq:bound-phi-l}) extends to the case that $\phi =
  E(\Psi,\cdot)$ is an incomplete Eisenstein series by the
  integral formula \eqref{eq:phi0-useful-formula-1} for its Fourier coefficients and the
  rapid decay of the test function $\Psi^\wedge$.

  Taking $A$ sufficiently large in (\ref{eq:bound-phi-l})
  and summing over $l \in \mathfrak{z}_j$ that satisfy
  the condition (\ref{eq:l-max-dyadic})
  for some $r \in \mathbb{Z}_{\geq 0}$, we deduce
  \begin{equation}\label{eq:22}
    |\mathcal{S}_2| \ll  Y^{-12} \mu_f(E(h_Y,\cdot)).
  \end{equation}
  The function $h$ is bounded,
  so
  \begin{equation}\label{eq:9}
    E(h_Y,g)
    = \sum_{\gamma \in B(\mathbb{F}) \backslash G(\mathbb{F})}
    h(Y |y(\gamma g)|)
    \ll \# \{\gamma \in B(\mathbb{F}) \backslash G(\mathbb{F})
    : | y(\gamma g)| \asymp Y^{-1}
    \}.
  \end{equation}
  By \cite[Lem 8.7]{venkatesh-2005}, the cardinality on the
  RHS of \eqref{eq:9} is $\ll Y^{1+\eps}$, uniformly in $g$.
  Thus $E(h_Y,\cdot) \ll Y^{1+\eps}$ and $\mu_f(E(h_Y,\cdot))
  \ll Y^{1+\eps} \mu_f(1)$,
  so \eqref{eq:22} gives
  $|\mathcal{S}_2| \ll Y^{-10} \mu_f(1)$.
\end{proof}

\begin{proof}[Proof of Proposition \ref{thm:reduction}]
  Follows immediately from the sequence of lemmas proved in this
  section together with the consequence
  \[
  \frac{1}{Y \mu_f(1)}
  \frac{\vecgamma(k-\mathbf{1})}{(4 \pi \mathbf{1})^{k-\mathbf{1}}}
  = \frac{c_1(\mathbb{F})}{L(\ad f, 1)}
  \frac{1}{(k-\mathbf{1})^{\mathbf{1}} Y}
  \]
  of the formula \eqref{eq:26}.
\end{proof}

\begin{remark}\label{rmk:compare-marshall}
  Let us point out the essential difference between our
  method and that of Marshall \cite{2010arXiv1006.3305M}.
  Recall that starting from Lemma \ref{lem:5.1}, we have
  integrated $\phi |f|^2$
  against the incomplete Eisenstein series $E(h,\cdot)$
  attached to a test function $h \in C_c^\infty(C_{\mathbb{F}} /
  C_{\mathbb{F}}^1) = C_c^\infty(\mathbb{R}_+^*)$.  Marshall
  instead integrates against what he calls a ``unipotent
  Eisenstein series,'' which (reinterpreted adelically) amounts to
  the incomplete Eisenstein series $E(H,\cdot)$ attached to
  the
  test function $H \in
  C_c^\infty(C_{\mathbb{F}}/\hat{\mathfrak{o} }^*)$ given by
  $H(y) = \sum_{\alpha \in \mathbb{F}^*} h(\alpha y)$
  for some pure tensor $h = \prod h_v \in
  C_c^\infty(\mathbb{A}^* / \hat{\mathfrak{o} }^*)$.
  Suppose that $\phi$ is cuspidal;
  the case that $\phi = E(\Psi,\cdot)$ is an incomplete
  Eisenstein series proceeds similarly after separating out the
  constant term and appealing to the formula
  \eqref{eq:phi0-useful-formula-1}.
  Then
  \begin{eqnarray*}
    \mu_f(E(H,\cdot) \phi)
    &=& \int_{Z(\mathbb{A}) B(\mathbb{F}) \backslash
      G(\mathbb{A})} H \phi |f|^2 \\
    &=& \int_{y \in \mathbb{F}^* \backslash \mathbb{A}^*}
    \left( \sum_{\alpha \in \mathbb{F}^*} h(\alpha y) \right)
    \int_{x \in \mathbb{F} \backslash \mathbb{A}} (\phi
    |f|^2)(n(x) a(y))
    \, d x \, \frac{d^\times y}{|y|} \\
    &=& \int_{y \in \mathbb{A}^*}
    h(y)
    \int_{x \in \mathbb{F} \backslash \mathbb{A}} (\phi
    |f|^2)(n(x) a(y))
    \, d x \, \frac{d^\times y}{|y|} \\
    &=&
    \mathop{\sum \sum}_{
      \substack{
        (l,n) \in \mathbb{F}^* \times \mathbb{F}^* \\
        m := n + l \in \mathbb{F}^*
      }
    }
    \int_{y \in \mathbb{A}^*} h(y) \kappa_\phi(l y) \kappa_f(m y)
    \kappa_f(n y) \, \frac{d^\times y}{|y|}.
  \end{eqnarray*}
  The integral in the final expression factorizes over the
  places of $\mathbb{F}$; taking each $h_\mathfrak{p}$ to be the
  characteristic function of $\mathfrak{o}_p^*$ and
  $h_{\infty_j}(y) = h_0(Y y)$ for some fixed $h_0 \in
  C_c^\infty(\mathbb{R}_+^*)$ gives
  \begin{equation}
    \begin{split}
      \mu_f(E(H,\cdot) \phi)
      &= \mathop{\sum \sum}_{
        \substack{
          (l,n) \in (\mathbb{F}^* \cap \mathfrak{o})^2 \\
          m := n +l \in
          \mathbb{F}^* \cap \mathfrak{o}}
      }
      \frac{
        \lambda_\phi(l)
        \lambda_f(m) \lambda_f(n)
      }{
        \sqrt{
          |
          l^{\mathbf{1}}
          m^{\mathbf{1}}
          n^{\mathbf{1}}
          |
        }
      }
      \\
      &\quad \times
      \prod_{j=1}^{[\mathbb{F}:\mathbb{Q}]}
      \int_{y \in \mathbb{R}_+^*}
      h_0(Y y)
      \kappa_{\phi,\infty_j}(l_j y) \kappa_{f,\infty_j}(m_j y)
      \kappa_{f,\infty_j}(n_j y)
      \, \frac{d^\times y}{y}.
    \end{split}
  \end{equation}
  The integrals here, which may be treated either by bounding
  $\kappa_{\phi,\infty_j}$ trivially as in
  \eqref{eq:j-bound-trivial} (which is basically what
  Holowinsky and Marshall do) or by our sharp refinement
  given in Lemma \ref{lem:j-bound-legit}, essentially truncate the sum
  over $l$ and $n$ to a pair of boxes rather than regions
  bounded by a hyperbola and hyperplanes as in our approach.
\end{remark}

\section{Bounds for shifted sums under
  hyperbolas}\label{sec:bounds-shifted-sums}
In this section we establish Theorem \ref{thm:essential-sums},
whose hypotheses we now recall.
Let $d = [\mathbb{F}:\mathbb{Q}]$ be the degree
of our totally real number field $\mathbb{F}$,
so that $\mathbb{F}_\infty \cong \mathbb{R}^d$
(see \S\ref{sec:real-embeddings}).
Let $T \in
\mathbb{R}_{\geq 1}^{d}$
and $U \in
\mathbb{R}_{\geq 1}$ be parameters to which we associate the
region
\[
\mathcal{R}_{T,U} =
\left\{ x \in \mathbb{R}^{d}:
  x^{\mathbf{1}} \leq
  X,
  \,
  x \geq  T/U
\right\},
\quad X := T^{\mathbf{1}}.
\]
Let $\mathfrak{z} \subset \mathbb{F}$ be a fractional ideal and
$l \in \mathbb{F}^* \cap \mathfrak{z}$ a nonzero ``shift.''  Let
$\lambda : I_{\mathbb{F}} \rightarrow \mathbb{C}$ be a weakly
multiplicative function that satisfies $|\lambda(\mathfrak{a})|
\leq \tau(\mathfrak{a})$.  We would like to bound certain sums
\begin{equation}\label{eq2:sums-restate}
  \Sigma_\lambda(\mathfrak{z},l,T,U)
  :=
  \sum
  _{
    \substack{
      n \in \mathfrak{z}\\
      m := n + l \in \mathfrak{z} \\
      \mathbf{max}(m,n) \in \mathcal{R}_{T,U}
    }
  }
  \lvert \lambda (\mathfrak{z} ^{-1} m ) \lambda(\mathfrak{z}
  ^{-1} n)
  \rvert.
\end{equation}

Our strategy for doing so generalizes Holowinsky's.  By the
assumption $|\lambda(\mathfrak{a})| \leq \tau(\mathfrak{a})$ we
reduce to quantifying the ``independence'' of the
small prime factors
of $m$ and $n$, which in turn reduces to a classical sieving
problem (estimating how many lattice points in a region satisfy
some congruence conditions).  By general machinery due to
Linnik, R\'{e}nyi, Bombieri and Davenport, Montgomery and others
in the case $\mathbb{F} = \mathbb{Q}$ (see \cite[\S 27]{Dav80},
\cite[p180]{MR2061214} and \cite{MR2426239}), such classical sieving problems
follow from additive large sieve inequalities (quantifying the
approximate orthogonality of a family of additive characters on
a lattice when restricted to the intersection of that lattice
with a sufficiently smooth region), which in turn follow from
bounds for sums over well-spaced points in the support
$\mathcal{R}_{T,U}^\wedge$ of the Fourier transform of a smooth
majorizer for the region $\mathcal{R}_{T,U}$.

Some care is required when $[\mathbb{F}:\mathbb{Q}] > 1$ because
then $\mathcal{R}_{T,U}^\wedge$ will have long and thin regions
that (unfortunately) accomodate many well-spaced points.  In our
intended application the parameter $U$ is small enough that one
can successfully analyze $\mathcal{R}_{T,U}^\wedge$ without
using any properties of $\mathfrak{z}$ beyond that it
is a lattice, but to simplify our treatment and allow arbitrary
values of $U$ we instead exploit the symmetries of
the fractional ideal $\mathfrak{z}$
coming from the action of the units $\mathfrak{o}_+^*$.  First, we cover
$\mathcal{R}_{T,U}$ by $\ll \log(e U)^{n-1}$ boxes of volume
$X =T^{\mathbf{1}}$:
\begin{lemma}\label{lem:boxes-cover}
  There exists a finite collection
  $(\mathcal{R}_{\alpha})_{\alpha \in A}$
 of boxes
  \[
  \mathcal{R}_{\alpha} = [a_{\alpha,1},b_{\alpha,1}]
  \times \dotsb \times
  [a_{\alpha,d},b_{\alpha,d}]
  \subset \mathbb{R}_{\geq 0}^{d},
  \quad 0 \leq  a_{\alpha,j} < b_{\alpha,j}
  \]
  whose union contains $\mathcal{R}_{T,U}$
  with
  $\# A \ll \log(e U)^{d-1}$
  such that
  $\vol(\mathcal{R}_\alpha) = X$
  and $b_{\alpha,1} \dotsb b_{\alpha,d} \ll X$ for
  each $\alpha \in A$.
\end{lemma}
\begin{proof}
  Let $x \in \mathcal{R}_{T,U}$,
  so that $x_1 \dotsb x_d \leq T_1 \dotsb T_d$
  and $x_i \geq T_i/U$.
  By the pigeonhole principle, we have
  $\prod_{j \neq i} x_j \leq \prod_{j \leq i} T_j$
  for some index $i$;
  to simplify notation, suppose that $i =1$,
  so that $x_2 \dotsb x_d \leq T_2 \dotsb T_d$.
  Choose integers $a_2, \dotsc, a_d$
  so that
  \begin{equation*}
    \frac{T_i}{2^{a_i}}
    \leq x_i \leq \frac{T_i}{2^{a_i-1}}.
  \end{equation*}
  Since
  $0
  \leq 
  x_1 \leq T_1 T_2 \dotsb T_d / x_2 \dotsb x_d
  \leq 2^{a_2 + \dotsb + a_d} T_1$,
  we see that
  $x$ is contained in the box
  \[
  \mathcal{R} = \left[
    0,2^{a_2 + \dotsb + a_d} T_1
  \right]
  \times
  \left[
    \frac{T_2}{2^{a_2}},
    \frac{T_2}{2^{a_2-1}},
  \right]
  \times \dotsb
  \times
  \left[
    \frac{T_d}{2^{a_d}},
    \frac{T_d}{2^{a_d-1}},
  \right],
  \]
  which satisfies the desiderata of the lemma.
  Since $x_2 \dotsb x_d \leq T_2 \dotsb T_d$
  implies
  \[
  \frac{T_2}{2^{a_2}} \dotsb \frac{T_d}{2^{a_d}}
  \leq x_2 \dotsb x_d \leq T_2 \dotsb T_d,
  \]
  and because $x_i \geq T_i / U$,
  we deduce that
  \begin{equation}\label{eq:conditions-on-a}
    a_i \leq \lceil \log_2 U \rfloor
    \text{ for } i = 2, \dotsc, d
    \quad
    \text{ and }
    \quad 
    a_2 + \dotsb a_d
    \geq 0.
  \end{equation}
  There are $\ll \log(e U)^{d-1}$ tuples $(a_2,\dotsc,a_d) \in
  \mathbb{Z}^{d-1}$
  satisfying the conditions \eqref{eq:conditions-on-a}.
\end{proof}
Next,
because $\lambda$ and $\mathfrak{z}$ are invariant
under $\mathfrak{o}_+^*$,
we see that for any (totally positive) unit $\eta \in \mathfrak{o}_+^*$
and any region $\mathcal{R} \subset \mathbb{R}^d$,
we have
\[
\sum
_{
  \substack{
    n \in \mathfrak{z}\\
    m := n + l \in \mathfrak{z} \\
    \mathbf{max}(m,n) \in \mathcal{R}
  }
}
\lvert \lambda (\mathfrak{z} ^{-1} m ) \lambda(\mathfrak{z}
^{-1} n)
\rvert
=
\sum
_{
  \substack{
    n \in \mathfrak{z}\\
    m := n + \eta^{-1} l \in \mathfrak{z} \\
    \mathbf{max}(m,n) \in \eta \mathcal{R}
  }
}
\lvert \lambda (\mathfrak{z} ^{-1} m ) \lambda(\mathfrak{z}
^{-1} n)
\rvert
\]
where $\eta \mathcal{R} = \{\eta x: x \in \mathcal{R} \}$.
The $\mathfrak{o}_+^*$-orbit of any box
$\mathcal{R}_{\alpha}$ as in Lemma \ref{lem:boxes-cover}
contains a representative $[a_1,b_1] \times \dotsb \times
[a_d,b_d]$ for which $|a_i - b_i| \asymp |a_j - b_j| \asymp
X^{1/d}$ for all $i,j \in \{1, \dotsc, d\}$.
Thus
\begin{equation}
  \Sigma_\lambda(\mathfrak{z},l,T,U)
  \ll 
  \log(eU)^{d-1}
  \sup_{\mathcal{R}}
  \sup_{\eta \in \mathfrak{o}_+^*}
  \sum_{
    \substack{
      n \in \mathfrak{z}\\
      m := n + \eta^{-1} l \in \mathfrak{z} \\
      \mathbf{max}(m,n) \in \mathcal{R}
    }
  }
  |\lambda(\mathfrak{z}^{-1} m) \lambda(\mathfrak{z}^{-1}n)|
\end{equation}
where the supremum is taken over all boxes $\mathcal{R} =
[a_1,b_1] \times \dotsb \times [a_d,b_d]$ for which
$\vol(\mathcal{R}) = X$, $|a_i-b_i| \asymp X^{1/d}$,
$0 \leq  a_i < b_i$
and $\max(b_1,\dotsc,b_d) \ll X^{1/d}$, with the
implied constants depending only upon the field $\mathbb{F}$.
Finally, if $\mathbf{max}(m,n)$ belongs to such a box $\mathcal{R}$
with $m,n \in \mathbb{F}_{\infty+}^*$, then
both $m$ and $n$ belong to
the box $(0,b_1] \times \dotsb \times (0,b_d]$.
Therefore Theorem \ref{thm:essential-sums}
reduces to the following result,
which we shall establish in the remainder of this section.
\begin{theorem}\label{thm:shifted-sums-boxes}
  Let $\mathbb{F}$ be a totally real number field of degree $d =
  [\mathbb{F}:\mathbb{Q}]$, let $\lambda : I_{\mathbb{F}}
  \rightarrow \mathbb{R}_{\geq 0}$ be a nonnegative-valued
  multiplicative function that satisfies $\lambda(\mathfrak{a})
  \leq \tau(\mathfrak{a})$ for all $\mathfrak{a} \in
  I_{\mathbb{F}}$, let $\mathfrak{z}$
  be a fractional ideal in $\mathbb{F}$, let $\lambda^0 :
  \mathfrak{z} \rightarrow \mathbb{R}_{\geq 0}$ be the function
  $\lambda^0(n) = \lambda(\mathfrak{z}^{-1} n)$,
  let $X \geq 2$, and let
  \begin{equation}\label{eq:defn-R-X-z}
    \mathcal{R}_{X,\mathfrak{z}}
    =
    (0,(\norm(\mathfrak{z})X)^{1/d}] \times \dotsb \times
    (0,(\norm(\mathfrak{z})X)^{1/d}] \subset
    \mathbb{R}^d.
  \end{equation}
  Then for $l \in
  \mathfrak{z} \cap \mathbb{F}^*$,
  we have
  \begin{equation}
    \sum
    _{
      \substack{
        n \in \mathfrak{z} \cap \mathcal{R}_{X,\mathfrak{z}}\\
        m := n + l \in \mathfrak{z} \cap \mathcal{R}_{X,\mathfrak{z}}
      }
    }
    \lambda^0(m) \lambda^0(n)
    \ll_{\mathbb{F},\eps}
    \frac{X}{\log(X)^{2-\eps}}
    \prod_{
      \substack{
        \norm(\mathfrak{p})
        \leq X \\
        % \mathfrak{p} \sim \mathfrak{z}^{-1}
      }
    }
    \left( 1 + \frac{2 \lambda(\mathfrak{p})}{\norm(\mathfrak{p})} \right).
  \end{equation}
\end{theorem}

Preserve the hypotheses and notation of Theorem \ref{thm:shifted-sums-boxes}.
Throughout this section the nonzero shift $l \in \mathfrak{z}
\cap \mathbb{F}^*$ is fixed,
while $m$ and $n$ denote
elements of
$\mathfrak{z}$ having difference $m - n = l$.
To ease the notation, we
write $|\mathfrak{a}| = \norm(\mathfrak{a})$ for the norm of an
integral ideal $\mathfrak{a}$.
Theorem \ref{thm:shifted-sums-boxes} is trivial
for bounded values of $X$;
thus we may and shall assume for convenience
that $X$ is sufficiently large,
so that for
instance $\log \log(X) \gg 1$.
% Since we allow the implied
% constant in the bound asserted by Theorem
% \ref{thm:essential-sums} to depend upon $\mathfrak{z}$, we shall
% regard $\mathfrak{z}$ as fixed throughout this section and drop
% its dependence from any asymptotic notation; thus, any and all
% implied constants may depend upon $\mathfrak{z}$.

For a real parameter
\begin{equation}\label{eq:z-small-power-of-X}
  z =  X^{1/s}, \quad
  s \in \mathbb{R}_{>0},
\end{equation}
define the \emph{$z$-part}
of an element $n \in \mathfrak{z}$
to be the greatest divisor
of the integral ideal $\mathfrak{z}^{-1} n$ each of whose prime
factors has norm at most $z$, so that if
$\mathfrak{z}^{-1} n$
factors as a product of prime powers $\prod \mathfrak{p}_i^{k_i}$,
then the $z$-part of $n$
is $\prod_{|\mathfrak{p}_i| \leq z} \mathfrak{p}_i^{k_i}$.
Define the \emph{$z$-datum} of $n$
to be the unique triple
$(\mathfrak{a},\mathfrak{b},\mathfrak{c})$
of integral ideals
for which
\begin{itemize}
\item $\mathfrak{a}$ and $\mathfrak{b}$ are coprime,
\item $\mathfrak{a} \mathfrak{c}$ is the $z$-part of $m := n+l$,
  and
\item $\mathfrak{b} \mathfrak{c}$ is the $z$-part of $n$.
\end{itemize}
Thus the size of $\mathfrak{c}$ quantifies the overlap between
small primes occurring in $\mathfrak{z}^{-1} m$ and
$\mathfrak{z}^{-1} n$.  Let $\mathcal{Z}$ denote
the set of all $z$-data that arise in this way and
$\mathfrak{z}_{\mathfrak{a},\mathfrak{b},\mathfrak{c}}$ the set
of all elements $n \in \mathfrak{z}$ having $z$-datum
$(\mathfrak{a},\mathfrak{b},\mathfrak{c})$, so that we have a
partition
\begin{equation}\label{eq:partition-z-datum}
  \mathfrak{z} = \sqcup
  \{\mathfrak{z}_{\mathfrak{a},\mathfrak{b},\mathfrak{c}}:
  (\mathfrak{a},\mathfrak{b},\mathfrak{c}) \in \mathcal{Z} \}.
\end{equation}
Note that for all $(\mathfrak{a},\mathfrak{b},\mathfrak{c}) \in
\mathcal{Z}$
we have $\mathfrak{c} | \mathfrak{z}^{-1} l$,
so that $\mathfrak{c}^{-1} \mathfrak{z}^{-1} l$
is an integral ideal.

Now let
\begin{equation}\label{eq:y-small-power-of-X}
  y = X^{\alpha}, \quad \alpha \in \mathbb{R}_{>0}
\end{equation}
be a real parameter
and partition $\mathcal{Z}$
into subsets
\begin{align*}
  \mathcal{Z}_{\leq y}
  &= \{(\mathfrak{a},\mathfrak{b},\mathfrak{c}) \in \mathcal{Z}:
  \max(|\mathfrak{a} \mathfrak{c}|, |\mathfrak{b} \mathfrak{c}|)
  \leq y\},\\
  \mathcal{Z}_{> y}
  &= \{(\mathfrak{a},\mathfrak{b},\mathfrak{c}) \in \mathcal{Z}:
  \max(|\mathfrak{a} \mathfrak{c}|, |\mathfrak{b} \mathfrak{c}|)
  > y\}.
\end{align*}
Thus the $z$-datum of $n \in \mathfrak{z}$
belongs to
$\mathcal{Z}_{\leq y}$ if
both
$\mathfrak{z}^{-1} m$ and $\mathfrak{z}^{-1} n$
have few small prime factors
and to $\mathcal{Z}_{>y}$
if either
$\mathfrak{z}^{-1} m$ or $\mathfrak{z}^{-1} n$
has many small prime factors,
where $y$ determines
the threshold separating ``few'' from ``many.''
The latter case occurs infrequently,
as we now show in Lemma \ref{lem:many-small-primes};
the former case will be addressed
by Lemma \ref{lem:few-small-primes}.
\begin{lemma}\label{lem:many-small-primes}
  Suppose that $2 \leq z \leq y \leq X$
  with $s$ and $\alpha$ as in \eqref{eq:z-small-power-of-X},
  \eqref{eq:y-small-power-of-X} such that
  $s \asymp \log \log(X)$
  and $\alpha \asymp 1$.
  Then
  \begin{equation}\label{eq:many-small-primes-bound}
    \sum_{(\mathfrak{a},\mathfrak{b},\mathfrak{c})
      \in \mathcal{Z}_{>y}}
    \sum_{
      \substack{
        n \in
        \mathfrak{z}_{\mathfrak{a},\mathfrak{b},\mathfrak{c}} \\
        m,n \in \mathcal{R}_{X,\mathfrak{z}}
      }
    }
    \lambda ^0 (m) \lambda ^0 (n)
    \ll X \log(X)^{-A}.
  \end{equation}
\end{lemma}
\begin{proof}
  The LHS of \eqref{eq:many-small-primes-bound} is the sum of
  $\lambda^0(m) \lambda^0(n)$ taken over those $m,n \in
  \mathfrak{z} \cap \mathcal{R}_{X,\mathfrak{z}}$ with $m-n = l$
  for which the $z$-part of
  either $m$ or $n$ has norm greater than $y$.  Writing
  $\mathfrak{a}$ and $\mathfrak{b}$ for the $z$-parts of $m$ and
  $n$ and invoking Cauchy-Schwarz twice, we see that the LHS of
  \eqref{eq:many-small-primes-bound} is
  \begin{equation*}
    \begin{split}
      &\leq
      \left(
        \sum_{
          \substack{
            y < |\mathfrak{a}| \leq X \\
            \mathfrak{p} | \mathfrak{a} \implies
            |\mathfrak{p}| \leq z
          }
        }
        \# ( \mathfrak{a} \mathfrak{z} \cap \mathcal{R}_{X,\mathfrak{z}})
      \right)^{1/4}
      \left(
        \sum_{
          m \in \mathfrak{z} \cap \mathcal{R}_{X,\mathfrak{z}}
        }
        \lambda^0(m)^4
      \right)^{1/4}
      \left(
        \sum_{
          n \in \mathfrak{z} \cap \mathcal{R}_{X,\mathfrak{z}}
        }
        \lambda^0(n)^2
      \right)^{1/2}
      \\
      &\quad
      +
      \left(
        \sum_{
          \substack{
            y < |\mathfrak{b}| \leq X \\
            \mathfrak{p} | \mathfrak{b} \implies
            |\mathfrak{p}| \leq z
          }
        }
        \# ( \mathfrak{b} \mathfrak{z} \cap \mathcal{R}_{X,\mathfrak{z}})
      \right)^{1/4}
      \left(
        \sum_{
          m \in \mathfrak{z} \cap \mathcal{R}_{X,\mathfrak{z}}
        }
        \lambda^0(m)^2
      \right)^{1/2}
      \left(
        \sum_{
          n \in \mathfrak{z} \cap \mathcal{R}_{X,\mathfrak{z}}
        }
        \lambda^0(n)^4
      \right)^{1/4}.
    \end{split}
  \end{equation*}
  We have
  $\sum_{m \in \mathfrak{z} \cap \mathcal{R}_{X,\mathfrak{z}}}
  \lambda^0(m)^4 \ll
  X\log(X)^{15}$
  and $\sum_{m \in \mathfrak{z} \cap \mathcal{R}_{X,\mathfrak{z}}}
  \lambda^0(m)^2 \ll
  X\log(X)^{3}$
  by the same argument as when $\mathbb{F} = \mathbb{Q}$
  (see \cite[\S 1.6]{MR2061214})
  and $\# (\mathfrak{a} \mathfrak{z} \cap \mathcal{R}_{X,\mathfrak{z}})
  \ll 1 + |\mathfrak{a}|^{-1} X \ll |\mathfrak{a}|^{-1} X$,
  so that
  \begin{equation}\label{eq:sixsums2}
    \sum_{(\mathfrak{a},\mathfrak{b},\mathfrak{c})
      \in \mathcal{Z}_{>y}}
    \sum_{
      \substack{
        n \in
        \mathfrak{z}_{\mathfrak{a},\mathfrak{b},\mathfrak{c}} \\
        m,n \in \mathcal{R}_{X,\mathfrak{z}}
      }
    }
    \lambda ^0 (m) \lambda ^0 (n)
    \ll X \log(X)^{O(1)}
    \left(
      \sum_{
        \substack{
          y < |\mathfrak{a}| \leq X \\
          \mathfrak{p} | \mathfrak{a} \implies
          |\mathfrak{p}| \leq z
        }
      }
      \frac{1}{|\mathfrak{a}|}
    \right)^{1/4}.
  \end{equation}
  Let $\Psi(t,z)$ denote the number of integral ideals
  $\mathfrak{a} \subset \mathfrak{o}$
  of norm $|\mathfrak{a}| \leq t$
  each of whose prime divisors
  $\mathfrak{p} | \mathfrak{a}$
  satisfy  $|\mathfrak{p}| \leq z$,
  so that by partial summation
  \begin{equation}\label{eq:partial-summation-krause}
    \sum_{
      \substack{
        y < |\mathfrak{a}| \leq X \\
        \mathfrak{p} | \mathfrak{a} \implies
        |\mathfrak{p}| \leq z
      }
    }
    \frac{1}{|\mathfrak{a}|}
    = \frac{\Psi(X,z)}{X}
    -  \frac{\Psi(y,z)}{y}
    + \int _y ^X \frac{\Psi(t,z)}{t^2} \, d t.
  \end{equation}
  A theorem of Krause \cite{MR1078363} (see also the survey
  \cite{MR1265913}) asserts that
  \[
  \Psi(t,z) = t \rho(u) \left( 1
    + O \left( \frac{\log(u+1)}{\log z} \right) \right),
  \quad u := \frac{\log t}{\log z}
  \]
  uniformly for $t \geq 2$ and $1 \leq u \leq (\log
  z)^{3/5-\eps}$ for any $\eps > 0$, where the \emph{Dickman
    function} $\rho : \mathbb{R}_{> 0} \to \mathbb{R}_{>0}$
  satisfies the asymptotics $\log \rho(u) = -(1+o(1)) u \log u$
  as $u \to +\infty$.
  For $y \leq t \leq X$, our assumptions
  $\alpha \asymp 1$ and $s \asymp \log \log(X)$
  imply that $u \asymp \log \log t$.
  Thus $\log z \asymp \log t / \log \log t$, so the condition
  for uniformity is satisfied and we obtain
  \[
  \Psi(t,z) \ll
  t \exp(-2 C \log \log t \log \log \log t) = t (\log t)^{-2 C
    \log \log \log t} \ll_A t (\log t)^{-A}
  \]
  for some $C > 0$ and every $A > 0$.
  It follows from \eqref{eq:partial-summation-krause}
  that
  \begin{equation}\label{eq:sofew}
    \sum_{
      \substack{
        y < |\mathfrak{a}| \leq X \\
        \mathfrak{p} | \mathfrak{a} \implies
        |\mathfrak{p}| \leq z
      }
    }
    \frac{1}{|\mathfrak{a}|}
    \ll_A \log(X)^{-A}.
  \end{equation}
  We deduce the required bound
  by substituting
  \eqref{eq:sofew} into \eqref{eq:sixsums2}
  and taking $A$ sufficiently large.
\end{proof}

On the other hand, if $\mathfrak{z}^{-1} m$ and
$\mathfrak{z}^{-1} n$ have few small prime factors, then
we shall show
by an application of the large sieve
that they typically have few \emph{common} small prime factors;
anticipating the bound
given by Corollary \ref{cor:bound-for-B-y-z},
set
\begin{equation}\label{eq:defn-B-of-y-z}
  B(y,z)
  := \sup_{(\mathfrak{a},\mathfrak{b},\mathfrak{c}) \in
    \mathcal{Z}_{\leq y}}
  \frac{\# \{
    n \in
    \mathfrak{z}_{\mathfrak{a},\mathfrak{b},\mathfrak{c}}
    :
    m,n \in \mathcal{R}_{X}
    \}}
  {
    \displaystyle \frac{|\mathfrak{z}^{-1} l|}{|\mathfrak{c}|^2
      \phi(\mathfrak{a} \mathfrak{b} \mathfrak{c}^{-1}
      \mathfrak{z}^{-1} l)} },
\end{equation}
where $\phi$ denotes the Euler phi function (multiplicative,
$\mathfrak{p}^k \mapsto |\mathfrak{p}|^{k-1} (|\mathfrak{p}| -
1)$).
\begin{lemma}\label{lem:few-small-primes}
  For $y,z$ as in
  \eqref{eq:z-small-power-of-X},
  \eqref{eq:y-small-power-of-X},
  we have
  \begin{equation}\label{eq:few-small-primes}
    \sum_{(\mathfrak{a},\mathfrak{b},\mathfrak{c})
      \in \mathcal{Z}_{\leq y}}
    \sum_{
      \substack{
        n \in
        \mathfrak{z}_{\mathfrak{a},\mathfrak{b},\mathfrak{c}} \\
        m,n \in \mathfrak{z} \cap \mathcal{R}_{X,\mathfrak{z}}
      }
    }
    \lambda^0(m) \lambda^0(n)
    \ll 4^s B(y,z)
    \log(X)^{\eps}
    \prod _{
      \substack{
        \lvert \mathfrak{p}  \rvert \leq z \\
        % \mathfrak{p} \sim \mathfrak{z}^{-1}
      }
    }
    \left( 1 + \frac{2 \lambda (\mathfrak{p} )}{\lvert
        \mathfrak{p}  \rvert} \right
    ).
  \end{equation}
\end{lemma}
\begin{proof}
  First, write $\mathfrak{z}^{-1} m = \mathfrak{a} \mathfrak{c}
  \mathfrak{m}$
  and factor $\mathfrak{m}$ as a product of prime powers
  $\mathfrak{p}_i^{a_i}$
  with $|\mathfrak{p}_i| > z$;
  since $|\mathfrak{m}| \leq  X$, we have
  \[
  \sum a_i \log(z) 
  \leq \sum a_i \log|\mathfrak{p}_i|
  = \log |\mathfrak{m}|
  \leq \log(X)
  = s \log(z),
  \]
  so that our assumption $\lambda(\mathfrak{p}_i^{a_i}) \leq a_i + 1
  \leq 2^{a_i}$ implies
  $\lambda(\mathfrak{m})
  \leq 2^{\sum a_i}
  \leq 2^s$.
  Writing
  $\mathfrak{z}^{-1}n = \mathfrak{b} \mathfrak{c} \mathfrak{n}$,
  we find similarly that $\lambda(\mathfrak{n}) \leq 2^s$.
  Since $\gcd(\mathfrak{a} \mathfrak{c},\mathfrak{m})
  = \gcd(\mathfrak{b} \mathfrak{c}, \mathfrak{n}) =
  \mathfrak{o}$,
  we obtain
  $\lambda^0(m) \lambda^0(n)
  =\lambda(\mathfrak{a} \mathfrak{c}) \lambda(\mathfrak{b}
  \mathfrak{c})
  \lambda(\mathfrak{m}) \lambda(\mathfrak{n})
  \leq 4^s \lambda(\mathfrak{a} \mathfrak{c})
  \lambda(\mathfrak{b} \mathfrak{c})$.
  By the definition of $B(y,z)$
  and the inequality $\phi (\mathfrak{a} \mathfrak{b} )
  \geq \phi (\mathfrak{a} ) \phi (\mathfrak{b} )$,
  the LHS of \eqref{eq:few-small-primes} is thus
  \begin{equation}\label{eq:few-small-primes-1}
    \leq 4^s B(y,z)
    \sum_{
      \substack{
        \mathfrak{c} |\mathfrak{z}^{-1} l \\
        \mathfrak{p} | \mathfrak{c}
        \implies |\mathfrak{p}| \leq z
      }
    }
    \frac{
      |\mathfrak{z}^{-1}l|
    }{\phi(\mathfrak{c}^{-1} \mathfrak{z}^{-1}l ) |\mathfrak{c}|^2}
    \mathop{
      \sum_{|\mathfrak{a} \mathfrak{c}| \leq y}
      \sum_{|\mathfrak{b} \mathfrak{c}| \leq y}}_{
      \mathfrak{p} | \mathfrak{a} \mathfrak{b}
      \implies |\mathfrak{p}| \leq z}
    \frac{
      \lambda(\mathfrak{a \mathfrak{c}})
      \lambda(\mathfrak{b \mathfrak{c}})
    }
    {
      \phi(\mathfrak{a}) \phi(\mathfrak{b})
    }.
  \end{equation}
  For $\mathfrak{c}$ as in
  \eqref{eq:few-small-primes-1},
  the multiplicativity of $\lambda$ and $\phi$ implies that
  \begin{equation}\label{eq:few-small-primes-1a}
    \mathop{
      \sum_{|\mathfrak{a} \mathfrak{c}| \leq y}
      \sum_{|\mathfrak{b} \mathfrak{c}| \leq y}}_{
      \mathfrak{p} | \mathfrak{a} \mathfrak{b}
      \implies |\mathfrak{p}| \leq z}
    \frac{
      \lambda(\mathfrak{a \mathfrak{c}})
      \lambda(\mathfrak{b \mathfrak{c}})
    }
    {
      \phi(\mathfrak{a}) \phi(\mathfrak{b})
    }
    \leq
    \left(
      \prod_{|\mathfrak{p}| \leq z}
      \sum_{k \geq 0}
      \frac{
        \lambda(\mathfrak{p}^{k+v_\mathfrak{p}(\mathfrak{c})})
      }{
        \phi(\mathfrak{p}^k)
      }
    \right)^2,
  \end{equation}
  where $v_\mathfrak{p}(\mathfrak{c})$ denotes the order to
  which
  $\mathfrak{p}$ divides $\mathfrak{c}$.
  We rewrite
  \begin{equation}\label{eq:few-small-primes-2}
    \sum_{k \geq 0}
    \frac{\lambda(\mathfrak{p}^k)}{\phi(\mathfrak{p}^k)}
    = \left( 1 + \frac{\lambda(\mathfrak{p})}{|\mathfrak{p}|}
    \right)
    \left( 1 +
      \frac{
        \frac{\lambda (\mathfrak{p} )}{ \phi (\mathfrak{p} )}
        - \frac{\lambda (\mathfrak{p} ) }{ \lvert \mathfrak{p}
          \rvert}
        + \sum _{k \geq 2 }
        \frac{\lambda (\mathfrak{p} ^k )}{\phi (\mathfrak{p} ^k )}
      }{
        1 + \frac{\lambda (\mathfrak{p} )}{| \mathfrak{p} |}
      }
    \right).
  \end{equation}
  Using the inequalities $\lambda(\mathfrak{p}^k) \leq k+1$
  and $|\mathfrak{p}| \geq 2$
  and writing $q = |\mathfrak{p}|$ for clarity, we compute
  \begin{align*}
    \frac{\lambda (\mathfrak{p} )}{ \phi (\mathfrak{p} )}
    - \frac{\lambda (\mathfrak{p} ) }{ \lvert \mathfrak{p}
      \rvert}
    + \sum _{k \geq 2 }
    \frac{\lambda (\mathfrak{p} ^k )}{\phi (\mathfrak{p} ^k )}
    &\leq
    \frac{2}{q(q-1)}
    + \sum _{k \geq 2}
    \frac{k + 1}{q^{k-1}(q-1)}
    \\
    &=
    q^{-2} \left(
      2(1-q^{-1})^{-1}
      + 2 ( 1-q^{-1})^{-2}
      + (1 - q^{-1})^{-3}
    \right) \\
    &\leq
    20 q^{-2},
  \end{align*}
  so that \eqref{eq:few-small-primes-2} implies
  \begin{equation}
    \sum_{k \geq 0}
    \frac{\lambda(\mathfrak{p}^k)}{\phi(\mathfrak{p}^k)}
    \leq \left( 1 + \frac{\lambda(\mathfrak{p})}{|\mathfrak{p}|}
    \right)
    \left( 1 + \frac{20}{|\mathfrak{p}|^2} \right).
  \end{equation}
  If $\nu \geq 1$, then (writing $q = |\mathfrak{p}|$)
  \begin{align*}
    \sum_{k \geq 0}
    \frac{\lambda(\mathfrak{p}^{k+\nu})}{\phi(\mathfrak{p}^k)}
    &\leq
    \nu + 1 + \sum_{k  \geq 1}
    \frac{\nu + k + 1}{q^{k-1}(q-1)}
    \\
    &=
    1 + \nu \left( 1 + q^{-1} (1-q^{-1})^{-2} \right)
    + q^{-1}(1-q^{-1})^{-2} \\
    &\leq 3 \nu + 3.      
  \end{align*}
  Substituting these bounds
  into  \eqref{eq:few-small-primes-1}
  and \eqref{eq:few-small-primes-1a},
  the LHS of \eqref{eq:few-small-primes}
  is
  \begin{equation}\label{eq:few-small-primes-3}
    \begin{split}
      &\ll
      4^s B(y,z)
      \psi(\mathfrak{z}^{-1} l)
      \prod_{
        \substack{
          |\mathfrak{p}| \leq z \\
          % \mathfrak{p} \sim \mathfrak{z}^{-1}
        }
      }
      \left( 1 + \frac{2 \lambda (\mathfrak{p} )}{|\mathfrak{p}|}
      \right),
      \\
      &\quad
      \text{ with }
      \psi(\mathfrak{a})
      :=
      |\mathfrak{a}|
      \sum_{\mathfrak{c} | \mathfrak{a}}
      \frac{
        \prod_{\mathfrak{p}^\nu || \mathfrak{c}}
        (3 \nu + 3)^2
      }{
        \phi(\mathfrak{a}/\mathfrak{c})
        |\mathfrak{c}|^2
      }.
    \end{split}
  \end{equation}
  The function $\psi : I_{\mathbb{F}} \rightarrow
  \mathbb{R}_{\geq 0}$
  is multiplicative.
  On a prime power $\mathfrak{p}^a$
  with $a \geq 1$
  and $|\mathfrak{p}| = q \geq 2$
  it takes the value
  \[
  \psi(\mathfrak{p}^k)
  = \frac{1}{1-q^{-1}}
  + \frac{9}{q^a}
  \left( (a+1)^2
    + \frac{1}{1-q^{-1}}
    \sum_{i=1}^{a-1}
    \frac{(i+1)^2}{q^i}
  \right)
  \leq 1 + 10^6 q^{-1}.
  \]
  Since $\prod_{\mathfrak{p}|\mathfrak{a}}
  (1 + |\mathfrak{p}|^{-1}) \ll \log \log |\mathfrak{a}|$,
  it follows that $\psi(\mathfrak{a})
  \ll \log \log(\mathfrak{a})^{10^6}$.
  If $|\mathfrak{z}^{-1} l| > X$, then the LHS of
  \eqref{eq:few-small-primes}
  is zero; if otherwise $|\mathfrak{z}^{-1} l| \leq X$,
  then
  $\psi(\mathfrak{z}^{-1} \mathfrak{l})
  \ll \log(X)^{\eps}$.
  Thus \eqref{eq:few-small-primes} follows
  from \eqref{eq:few-small-primes-3}.
\end{proof}
% \begin{remark}
%   Our proof of Lemma \ref{lem:few-small-primes} makes essential
%   use of the ``Ramanujan bound'' for $\lambda$; it would not
%   suffice to know that the Ramanujan bound is satisfied for a
%   full density set of primes.
% \end{remark}

By Lemma \ref{lem:many-small-primes} and Lemma
\ref{lem:few-small-primes}, we see that Theorem
\ref{thm:essential-sums} follows from sufficiently
strong bounds for the quantity $B(y,z)$
given by \eqref{eq:defn-B-of-y-z};
the following lemma reduces such bounds to a classical
sieving
problem.
\begin{definition}\label{defn:sifted-sums}
  For a region $\mathcal{R} \subset \mathbb{F}_{\infty}
  \cong \mathbb{R}^d$,
  an ideal $\mathfrak{x} \subset \mathbb{F}$,
  a finite set $\mathcal{P}$ of
  primes
  in $\mathfrak{o}$ and a collection
  $(\Omega _{\mathfrak{p} } ) _{\mathfrak{p} \in  \mathcal{P} }$
  of sets of residue classes $\Omega_\mathfrak{p} \subset
  \mathfrak{x}/
  \mathfrak{p} \mathfrak{x}$,
  define the \emph{sifted set}
  \begin{equation}
    \mathcal{S}(\mathcal{R},\mathfrak{x},(\Omega_\mathfrak{p}))
    := \{n \in \mathfrak{x} \cap \mathcal{R} :
    n \notin \Omega_\mathfrak{p} \pod{\mathfrak{p} \mathfrak{x}}
    \text{ for all } \mathfrak{p} \in \mathcal{P}
    \}.
  \end{equation}
  Define also for any $Q \geq 1$ the quantity
  \begin{equation}\label{eq:H-defn}
    H((\Omega_\mathfrak{p}),Q)
    =
    \sum_{
      \substack{
        |\mathfrak{q}| \leq Q \\
        \mathfrak{p} | \mathfrak{q}
        \implies \mathfrak{p} \in \mathcal{P} 
      }
    }
    \prod_{\mathfrak{p}|\mathfrak{q}}
    \frac{\# \Omega_\mathfrak{p} }{|\mathfrak{p}| - \#
      \Omega_\mathfrak{p} }.
  \end{equation}
\end{definition}
\begin{lemma}\label{lem:reduce-classical-sieve-problem}
  Let $(\mathfrak{a},\mathfrak{b},\mathfrak{c}) \in
  \mathcal{Z}$.
  Choose an element $r \in \mathfrak{c} \mathfrak{z}$
  so that $r \equiv 0 \pod{\mathfrak{a} \mathfrak{c}
    \mathfrak{z}}$
  and $r = - l \pod{\mathfrak{b} \mathfrak{c} \mathfrak{z}}$,
  and define the region
  \begin{equation}\label{eq:defn-region-R-r}
    \mathcal{R}_r
    = \{x - r  |
    x \in \mathcal{R}_{X,\mathfrak{z}}  \}.
  \end{equation}
  Let
  $\mathfrak{x} = \mathfrak{a} \mathfrak{b} \mathfrak{c}
  \mathfrak{z}$
  and let
  $\mathcal{P}$ denote the set of odd
  primes $\mathfrak{p}$ in $\mathfrak{o}$
  of norm $|\mathfrak{p}| \leq z$.
  Then there exists a collection of sets of residue classes
  $(\Omega_\mathfrak{p})_{\mathfrak{p} \in \mathcal{P}}$
  with $\Omega_\mathfrak{p} \subset \mathfrak{x} / \mathfrak{p}
  \mathfrak{x}$
  such that
  \begin{equation}\label{eq:defn-little-omega}
    \# \Omega_\mathfrak{p}
    :=
    \begin{cases}
      1 & \mathfrak{p} | \mathfrak{a} \mathfrak{b} \mathfrak{c}^{-1} \mathfrak{z}^{-1} l \\
      2 & \text{otherwise}
    \end{cases}
  \end{equation}
  and
  \begin{equation}\label{eq:antidivisibility-implies-congruences}
    \# (\mathfrak{z}_{\mathfrak{a},\mathfrak{b},\mathfrak{c}}
    \cap \mathcal{R}_{X,\mathfrak{z}})
    \leq \# \mathcal{S}(\mathcal{R}_r, \mathfrak{x},
    (\Omega_\mathfrak{p})).
  \end{equation}
\end{lemma}
\begin{proof}
  Indeed, let $(\mathfrak{a},\mathfrak{b},\mathfrak{c})
  \in \mathcal{Z}$,
  so that $\mathfrak{c} | \mathfrak{z}^{-1} l$
  and $\gcd(\mathfrak{a},\mathfrak{b}) = \mathfrak{o}$.
  Let $n \in \mathfrak{\mathfrak{z}}$.
  Then $n$ belongs to
  $\mathfrak{z}_{\mathfrak{a},\mathfrak{b},\mathfrak{c}}$
  if and only if
  \begin{enumerate}
  \item $n \in \mathfrak{a} \mathfrak{c} \mathfrak{z}$,
  \item $n + l \in \mathfrak{b} \mathfrak{c} \mathfrak{z}$,
  \item $\mathfrak{p} \nmid \mathfrak{z}^{-1} n / \mathfrak{a}
    \mathfrak{c}$ for each prime $\mathfrak{p}$ with norm
    $|\mathfrak{p}| \leq z$,
    and
  \item
    $\mathfrak{p} \nmid \mathfrak{z}^{-1} (n + l)
    / \mathfrak{b} \mathfrak{c}$
    for each prime $\mathfrak{p}$ with norm
    $|\mathfrak{p}| \leq z$.
  \end{enumerate}
  If $n \in
  \mathfrak{z}_{\mathfrak{a},\mathfrak{b},\mathfrak{c}}$, then
  conditions (1)--(2) assert that $n - r \in \mathfrak{a}
  \mathfrak{b} \mathfrak{c} \mathfrak{z}$, while conditions
  (3)--(4) assert (slightly more than) that for each prime
  $\mathfrak{p}$ with $|\mathfrak{p}| \leq z$, the number $n -
  r \in \mathfrak{a} \mathfrak{b} \mathfrak{c} \mathfrak{z} $
  does not belong to a certain collection $\Omega_\mathfrak{p}
  \subset \mathfrak{a} \mathfrak{b} \mathfrak{c} \mathfrak{z} /
  \mathfrak{p} \mathfrak{a} \mathfrak{b} \mathfrak{c}
  \mathfrak{z}$ of residue classes.  Precisely, let $\zeta \in
  \mathfrak{a} \mathfrak{b} \mathfrak{c} \mathfrak{z}$ and $n
  = \zeta + r$.
  \begin{itemize}
  \item Suppose $\mathfrak{p} | \mathfrak{a}$,
    $\mathfrak{p} \nmid \mathfrak{b}$.
    Let $\zeta_1 :=
    (\mathfrak{a} \mathfrak{b} \mathfrak{c} \mathfrak{z}
    / \mathfrak{p} \mathfrak{a}
    \mathfrak{b} \mathfrak{c} \mathfrak{z}
    \xrightarrow{\cong} \mathfrak{a} \mathfrak{c} \mathfrak{z} / \mathfrak{p}
    \mathfrak{a} \mathfrak{c} \mathfrak{z})^{-1}(-r)$.
    Then (3) holds iff $\zeta + r \notin \mathfrak{p} \mathfrak{a}
    \mathfrak{c} \mathfrak{z}$
    iff $\zeta - \zeta_1 \notin \mathfrak{p} \mathfrak{a}
    \mathfrak{b} \mathfrak{c} \mathfrak{z}$,
    while (4) holds iff
    $\zeta + r + l \notin \mathfrak{p} \mathfrak{b}
    \mathfrak{c} \mathfrak{z}$
    iff 
    (since $\zeta \in \mathfrak{a} \mathfrak{b} \mathfrak{c} \mathfrak{z}
    \subset \mathfrak{p} \mathfrak{b} \mathfrak{c} \mathfrak{z}$)
    $r + l \notin \mathfrak{p} \mathfrak{b}
    \mathfrak{c} \mathfrak{z}$
    iff $\mathfrak{p} \mathfrak{b} \mathfrak{z} \nmid \frac{r 
      + l }{\mathfrak{c} }$ iff (since
    $(\mathfrak{p},\mathfrak{b})=1$
    and $r + l \in \mathfrak{b} \mathfrak{c}$)
    $r + l \notin \mathfrak{p} \mathfrak{c} \mathfrak{z}$;
    we may take $\Omega_\mathfrak{p} = \{\zeta_1\}$,
    $\# \Omega_\mathfrak{p} = 1$.
  \item If $\mathfrak{p} \nmid \mathfrak{a}$,
    $\mathfrak{p} | \mathfrak{b}$, then we may similarly
    take $\#\Omega_\mathfrak{p} = 1$.
  \item The case $\mathfrak{p} | \mathfrak{a},
    \mathfrak{p} | \mathfrak{b}$ does not occur
    because $(\mathfrak{a},\mathfrak{b}) = 1$.
  \item Suppose $\mathfrak{p} \nmid \mathfrak{a} \mathfrak{b}$.
    Let $\zeta_1 :=
    (\mathfrak{a} \mathfrak{b} \mathfrak{c} \mathfrak{z}
    / \mathfrak{p} \mathfrak{a}
    \mathfrak{b} \mathfrak{c} \mathfrak{z}
    \xrightarrow{\cong} \mathfrak{a} \mathfrak{c} \mathfrak{z} / \mathfrak{p}
    \mathfrak{a} \mathfrak{c} \mathfrak{z})^{-1}(-r)$,
    $\zeta_2 :=
    (\mathfrak{a} \mathfrak{b} \mathfrak{c} \mathfrak{z}
    / \mathfrak{p} \mathfrak{a}
    \mathfrak{b} \mathfrak{c}  \mathfrak{z}
    \xrightarrow{\cong} \mathfrak{b} \mathfrak{c} \mathfrak{z}
    / \mathfrak{p}
    \mathfrak{b} \mathfrak{c} \mathfrak{z})^{-1}(-r-l)$.
    Then (3) holds iff
    $\zeta + r \notin \mathfrak{p} \mathfrak{a} \mathfrak{c}
    \mathfrak{z}$
    iff $\zeta - \zeta_1 \notin
    \mathfrak{p} \mathfrak{a} \mathfrak{b} \mathfrak{c}
    \mathfrak{z}$,
    while (4) holds iff
    $\zeta + r + l \notin \mathfrak{p} \mathfrak{b}
    \mathfrak{c} \mathfrak{z}$ iff
    $\zeta - \zeta_2 \notin
    \mathfrak{p} \mathfrak{a} \mathfrak{b} \mathfrak{c} \mathfrak{z}$.
    We may therefore take $\Omega_\mathfrak{p} =
    \{\zeta_1,\zeta_2\}$.
    We have $\zeta_1 \equiv \zeta_2 \pod{\mathfrak{p} \mathfrak{a}
      \mathfrak{b} \mathfrak{c} \mathfrak{z}}$ iff $l \in \mathfrak{p}
    \mathfrak{c} \mathfrak{z}$, in which case $\# \Omega_{\mathfrak{p}} = 1$;
    if $l \notin \mathfrak{p} \mathfrak{c} \mathfrak{z}$,
    then $\# \Omega_{\mathfrak{p}} = 2$.
  \end{itemize}
  Thus $n \mapsto n - r$ gives an inclusion
  $\mathfrak{z}_{\mathfrak{a},\mathfrak{b},\mathfrak{c}}
  \cap \mathcal{R} 
  \hookrightarrow 
  \mathcal{S}(\mathcal{R}_r,\mathfrak{a} \mathfrak{b}
  \mathfrak{c} \mathfrak{z}, (\Omega_\mathfrak{p}))$, and the $\#
  \Omega_\mathfrak{p}$ are as claimed.
\end{proof}

The large sieve machinery alluded to above
allows us to show the following, the proof of which
we postpone to a later subsection;
the proof is independent of what follows
in this subsection,
so there is no circularity in our arguments.
\begin{proposition}\label{prop:large-sieve-as-a-sieve}
  Let $\mathfrak{x}$, $\mathcal{P}$,
  and
  $(\Omega_\mathfrak{p})_{\mathfrak{p} \in \mathcal{P}}$
  be as in Definition \ref{defn:sifted-sums}.
  Let $\mathcal{R}$ be the region $\mathcal{R}_{X,\mathfrak{x}}$
  as in \eqref{eq:defn-R-X-z}
  or a translate thereof.
  There exists a positive constant
  $c_2(\mathbb{F}) > 0$
  such that for $X > c_2(\mathbb{F})$
  and $Q \geq 1$,
  we have
  \begin{equation}\label{eq:large-sieve-as-sieve}
    \mathcal{S}(\mathcal{R},\mathfrak{x},
    (\Omega_\mathfrak{p}))
    \ll
    \frac{
      X + Q^2
    }
    {
      H((\Omega_\mathfrak{p}),Q)
    }.
  \end{equation}
\end{proposition}
\begin{proof}
  See \S\ref{sec:reduct-from-class}.
\end{proof}
% \begin{proposition}\label{prop:large-sieve-as-a-sieve}
%   Let $\mathcal{R}$ be $\mathcal{R}_{T,U}$
%   or a translate thereof.
%   Let $\mathfrak{x}$, $\mathcal{P}$,
%   and
%   $(\Omega_\mathfrak{p})_{\mathfrak{p} \in \mathcal{P}}$
%   be as in Definition \ref{defn:sifted-sums}.
%   There exists a positive constant
%   $c_2(\mathbb{F}) > 0$
%   such that for $X > c_2(\mathbb{F}) |\mathfrak{x}|$
%   and any $Q \geq 1$,
%   we have
%   \begin{equation}
%     \mathcal{S}(\mathcal{R},\mathfrak{x},(\Omega_\mathfrak{p}))
%     \ll
%     \frac{\log(U)^{d-1}
%       ( |\mathfrak{x}|^{-1} X + Q^2)
%     }
%     {
%       H((\Omega_\mathfrak{p}),Q)
%     }.
%   \end{equation}
% \end{proposition}
As a consequence, we deduce the following
bound for $B(y,z)$.
\begin{corollary}\label{cor:bound-for-B-y-z}
  Let $c_2(\mathbb{F}) > 0$ be as in Proposition
  \ref{prop:large-sieve-as-a-sieve}.
  Then for $X > c_2(\mathbb{F})y^2$,
  the quantity
  $B(y,z)$ given by \eqref{eq:defn-B-of-y-z}
  satisfies
  \[
  B(y,z)
  \ll \frac{ X + y^2 z^2 }{\log(z)^2}.
  \]
\end{corollary}
\begin{proof}
  Let $(\mathfrak{a},\mathfrak{b},\mathfrak{c})
  \in \mathcal{Z}_{\leq y}$
  and let the region $\mathcal{R}_r$,
  the ideal $\mathfrak{x} = \mathfrak{a}
  \mathfrak{b} \mathfrak{c} \mathfrak{z}$,
  the set of primes $\mathcal{P}$ and
  the collection of sets of residue classes
  $(\Omega_\mathfrak{p})$
  be as in Lemma \ref{lem:reduce-classical-sieve-problem},
  so that \eqref{eq:antidivisibility-implies-congruences} holds.
  Then $|\mathfrak{x}| \leq y^2 |\mathfrak{z}|$,
  so that $X > c_2(\mathbb{F}) y^2$
  implies $X' > c_2(\mathbb{F})$
  with $X' := |\mathfrak{x}^{-1} \mathfrak{z}| X$;
  the hypothesis of Proposition
  \ref{prop:large-sieve-as-a-sieve}
  are then satisfied (taking $X'$ in place of $X$),
  and setting $Q = z$ we obtain
  \[
  \# (\mathfrak{z}_{\mathfrak{a},\mathfrak{b},\mathfrak{c}}
  \cap\mathcal{R}_{X,\mathfrak{z}})
  \ll
  \frac{|\mathfrak{x}^{-1} \mathfrak{z} | X +
    z^2}{H((\Omega_\mathfrak{p}),z)}.
  \]
  Set $\mathfrak{m} = \mathfrak{a} \mathfrak{b}
  \mathfrak{c}^{-1} \mathfrak{z}^{-1} l$
  (see \eqref{eq:defn-little-omega}).
  The lower bound
  \[
  H((\Omega_\mathfrak{p}),z)
  \gg_{\mathbb{F}}
  \frac{\phi(\mathfrak{m})}{|\mathfrak{m}|}\log(z)^2
  \]
  is standard when $\mathbb{F}=\mathbb{Q} $
  and follows in general
  from the arguments of \cite[pp55-59, Thm 2]{MR1836967}
  upon redefining ``$P(z)$'' to be the product
  of all prime ideals of norm up to $z$, replacing
  every sum over integers (resp. primes)
  satisfying some inequalities
  by the analogous sum over ideals (resp. prime ideals)
  with norms satisfying the analogous inequalities,
  and replacing the Riemann zeta function $\zeta$ by
  the Dedekind zeta function $\zeta_{\mathbb{F}}$.
  Thus recalling the definition \eqref{eq:defn-B-of-y-z}
  of $B(y,z)$,
  we obtain
  \[
  B(y,z)
  \ll \frac{|\mathfrak{x}|^{-1} X + z^2}{
    \frac{\phi(\mathfrak{m})}{|\mathfrak{m}|}
    \log(z)^2}
  \frac{|\mathfrak{c}|^2
    \phi(\mathfrak{m})}{|\mathfrak{z}^{-1} l|}
  = \frac{ X +
    |\mathfrak{a} \mathfrak{b} \mathfrak{c}|
    z^2}{\log(z)^2}.
  \]
  Since $|\mathfrak{a} \mathfrak{b} \mathfrak{c}|
  \leq y^2$, we deduce
  the claimed bound.
\end{proof}

\begin{proof}[Proof of Theorem \ref{thm:shifted-sums-boxes}]
  Let $y,z$ be given by \eqref{eq:z-small-power-of-X},
  \eqref{eq:y-small-power-of-X}
  with $\alpha \in (0,\tfrac{1}{2})$ and $s = \alpha \log \log(X)$.
  We eventually
  (i.e., as $X \rightarrow \infty$)
  have
  $X > c_2(\mathbb{F}) y^2$
  and $2 \leq z \leq y \leq X$.
  Thus the hypotheses of Lemma
  \ref{lem:many-small-primes},
  Lemma \ref{lem:few-small-primes}
  and Corollary \ref{cor:bound-for-B-y-z}
  are eventually satisfied,
  so we obtain
  \[
  \sum
  _{
    \substack{
      n \in \mathfrak{z} \cap \mathcal{R}_{X,\mathfrak{z}}\\
      m := n + l \in \mathfrak{z} \cap \mathcal{R}_{X,\mathfrak{z}}
    }
  }
  \lambda^0(m) \lambda^0(n)
  \ll 4^s
  \frac{X + y^2 z^2}{\log(z)^2} \log(X)^\eps
  \prod_{
    \substack{
      \norm(\mathfrak{p})
      \leq z \\
      % \mathfrak{p} \sim \mathfrak{z}^{-1}
    }
  }
  \left( 1 + \frac{2 \lambda(\mathfrak{p})}{\norm(\mathfrak{p})} \right).
  \]
  We have
  $4^s = \log(X)^{\alpha \log(4)}$,
  $\log(z) \gg_\alpha \log(X)^{2-\eps}$
  and $y^2 z^2 \ll_\alpha X$,
  so letting $\alpha \rightarrow 0$
  we deduce
  the assertion of Theorem \ref{thm:shifted-sums-boxes}.
\end{proof}

% \begin{proof}
%   \begin{eqnarray*}
%     \sum_{|\mathfrak{q}| \leq Q}
%     h(\mathfrak{q})
%     &=& \sum_{
%       \substack{
%         |\mathfrak{q}_1| \leq Q \\
%         \mathfrak{q}_1 | \mathfrak{m}
%       }
%     }
%     h(\mathfrak{q}_1)
%     \sum_{
%       \substack{
%         |\mathfrak{q}_2| \leq Q/|\mathfrak{q}_1| \\
%         (\mathfrak{q}_2,\mathfrak{m}) = 1
%       }
%     }
%     h(\mathfrak{q}_2) \\
%     &\geq
%     \left(
%       \prod_{
%         \substack{
%           \mathfrak{p}|\mathfrak{m} \\
%           |\mathfrak{p}| \leq z
%         }
%       }
%     \right)
%   \end{eqnarray*}
% \end{proof}

\appendix
\section{Sieve bounds}\label{sec:reduct-from-class}
Inequalities of the shape \eqref{eq:large-sieve-as-sieve} (with
explicit constants) have appeared in papers of Schaal \cite[Thm
5]{MR0272745} and Hinz \cite[Satz 2]{MR871630}, but only under
additional assumptions such as $Q \gg_{\mathbb{F}}
1$, $X \gg Q^2$, and
$\Omega_\mathfrak{p} = \emptyset$ for all $\mathfrak{p} |
\mathfrak{z}$.  Although it would possible to get around such
assumptions in our intended applications (at the cost of
sacrificing the uniformity in $\mathfrak{z}$,
which is ultimately not needed), we prefer to establish a result
in which such assumptions are not present.
We neglect here the issue of the leading coefficient
of such bounds, which is important in some of the applications
of the authors just cited but not in ours; for this reason
our analysis is substantially simplified.

Our arguments in this short section are standard; we have been
influenced by the books of Davenport \cite{Dav80} and Kowalski
\cite{MR2426239}, to which we refer the reader for a discussion
of the history of these ideas. Fix a fractional ideal
$\mathfrak{x}$ of $\mathbb{F}$.  Let $\mathfrak{q}$ be an
integral ideal in $\mathbb{F}$ and $\alpha : \mathfrak{x} /
\mathfrak{q} \mathfrak{x} \to \mathbb{C}$ a function on the
group $\mathfrak{x} / \mathfrak{q} \mathfrak{x}$.  Define
$L^2(\mathfrak{x}/\mathfrak{q}\mathfrak{x})$, $\|.\|_2$ with
respect to the counting measure, and for $\psi$ in the
Pontryagin dual $(\mathfrak{x} /
\mathfrak{q}\mathfrak{x})^\wedge$, define $\alpha^\wedge(\psi) =
\sum_{\mathfrak{x}/\mathfrak{q}\mathfrak{x}} \alpha(\zeta)
\overline{\psi}(\zeta)$; then the Fourier inversion and
Plancherel formulas read
\[
\alpha = |\mathfrak{q}|^{-1} \sum_{(\mathfrak{x} / \mathfrak{q} \mathfrak{x})^\wedge}
\alpha^\wedge(\psi) \psi,
\quad
\sum_{\mathfrak{x}/\mathfrak{q}\mathfrak{x}} \lvert \alpha(\zeta) \rvert^2
=
\|\alpha \|_2^2 = \|\alpha^\wedge\|_2^2 = |\mathfrak{q}|^{-1}
\sum_{(\mathfrak{x}/\mathfrak{q}\mathfrak{x})^\wedge} |\alpha^\wedge
(\psi)|^2.\]

For a proper divisor $\mathfrak{q}'$ of $\mathfrak{q}$, the
projection $\mathfrak{x} / \mathfrak{q} \mathfrak{x}\to \mathfrak{x} /
\mathfrak{q}'\mathfrak{x}$ induces an inclusion
$L^2(\mathfrak{x}/\mathfrak{q}'\mathfrak{x}) \hookrightarrow
L^2(\mathfrak{x}/\mathfrak{q}\mathfrak{x})$.  Let
$L^2_\#(\mathfrak{x}/\mathfrak{q}\mathfrak{x})$ denote the orthogonal
complement of the span of the images of these inclusions, write
$L^2(\mathfrak{x}/\mathfrak{q}\mathfrak{x}) \ni \alpha \mapsto
\alpha_\# \in L^2_\#(\mathfrak{x}/\mathfrak{q}\mathfrak{x})$ for the
associated orthogonal projection, and let
$(\mathfrak{x}/\mathfrak{q}\mathfrak{x})^\wedge_\#$ denote the set of
characters $\psi \in (\mathfrak{x}/\mathfrak{q}\mathfrak{x})^\wedge$ that
do not factor through any proper projection $\mathfrak{x} /
\mathfrak{q}\mathfrak{x} \to \mathfrak{x} /
\mathfrak{q}'\mathfrak{x}$,
so that
\[
\|\alpha_\#\|_2^2 = \lvert \mathfrak{q} \rvert^{-1}
\sum_{(\mathfrak{x}/\mathfrak{q}\mathfrak{x})^\wedge_\#} \lvert
\alpha^\wedge (\psi) \rvert^2.\]
For $\psi \in (\mathfrak{x}/\mathfrak{q}\mathfrak{x})^\wedge_\#$
call $\mathfrak{q}$ the \emph{conductor} of $\psi$.

Let $\mathcal{R}$ be a region in $\mathbb{F}_\infty$,
$\mathcal{P}$ a finite set of primes, $Q \geq 1$ a parameter,
and $\mathcal{Q}$ the set of squarefree ideals $\mathfrak{q}$
composed of primes $\mathfrak{p} \in \mathcal{P}$ with $\lvert
\mathfrak{q}\rvert \leq Q$.  Let $V(\mathcal{R},\mathfrak{x})$
be the Hilbert space of complex-valued functions $(a_n)_n :
\mathfrak{x} \to \mathbb{C}$ supported on $\mathcal{R} \cap
\mathfrak{x}$, where for $(a_n) \in V(\mathcal{R},\mathfrak{x})$
we set $\|a\|_2^2 := \sum_n |a_n|^2$.  For $\mathfrak{q} \in
\mathcal{Q}$ define $a[\mathfrak{q}] \in L^2(\mathfrak{x} /
\mathfrak{q}\mathfrak{x})$ by the formula
$a[\mathfrak{q}](\zeta) = \sum_{n = \zeta(\mathfrak{q}
  \mathfrak{x})} a_n$.  Let $E(\cdot;\mathfrak{x},Q)$ be the
quadratic form on $V(\mathcal{R},\mathfrak{x})$ defined by
\begin{equation}\label{eq2:8}
  E((a_n);\mathfrak{x},Q)
  = \sum_{\mathfrak{q} \in \mathcal{Q}}
  \lvert \mathfrak{q}  \rvert 
  \|a[\mathfrak{q}]_\#\|_2^2
  = \sum_{\mathfrak{q} \in \mathcal{Q}}
  \sum_{(\mathfrak{x}/\mathfrak{q}\mathfrak{x})^\wedge_\#}
  \lvert a[\mathfrak{q}]^\wedge(\psi) \rvert^2,
\end{equation}
and $D(\mathcal{R},\mathfrak{x},Q)$ the squared norm of
$E(\cdot;\mathfrak{x},Q)$, i.e., the smallest non-negative real with the
property that $|E((a_n);\mathfrak{x},Q)| \leq D(\mathcal{R},\mathfrak{x},Q)
\|a\|_2^2$ for all $(a_n) \in V(\mathcal{R},\mathfrak{x})$.

Suppose that $\alpha[\mathfrak{p}](\zeta) = 0$ for (at least)
$\omega(\mathfrak{p})$ values of $\zeta$ mod $\mathfrak{p}$ for
each $\mathfrak{p} \in \mathcal{P}$, and set $h(\mathfrak{q}) =
\prod_{\mathfrak{p}|\mathfrak{q}}
\frac{\omega(\mathfrak{p})}{|\mathfrak{p} |-
  \omega(\mathfrak{p})}$ for each $\mathfrak{q} \in
\mathcal{Q}$.  An inequality due to Montgomery \cite{MR0224585}
in the $(\mathbb{F},\mathfrak{x}) = (\mathbb{Q},\mathbb{Z})$
case (refining earlier work of Linnik, R\'{e}nyi, and
Bombieri-Davenport), whose proof generalizes painlessly to the
present situation and has been formulated axiomatically by
Kowalski \cite[Lem 2.7]{MR2426239}, shows that
$h(\mathfrak{q}) \|a[\mathfrak{o}]\|_2^2 \leq \lvert
\mathfrak{q} \rvert \|a[\mathfrak{q}]_\#\|_2^2$, so
recalling from \eqref{eq:H-defn}
that $H((\Omega_\mathfrak{p}),Q) = \sum_{\mathfrak{q} \in \mathcal{Q}} h(\mathfrak{q})$ we obtain
\[
\|a[\mathfrak{o}]\|_2^2 H((\Omega_\mathfrak{p}),Q)
\leq
D(\mathcal{R},\mathfrak{x},Q) \|a\|_2^2.
\]

In the special case
that $(a_n)_n$ is the indicator function of
$\mathcal{S}(\mathcal{R},\mathfrak{x},(\Omega_\mathfrak{p}))$
for some subsets $\Omega_\mathfrak{p} \subset \mathfrak{x} /
\mathfrak{p} \mathfrak{x}$, let
$Z := \# \mathcal{S}(
\mathcal{R}, \mathfrak{x},
(\Omega_\mathfrak{p}))$,
so that
\[
\|a\|_2^2 = \sum_n \lvert a_n|^2 = Z,
\quad
\|a[\mathfrak{o}]\|_2^2
= \lvert \sum_n a_n \rvert^2
= Z^2,
\]
and $a_n = 0$ whenever
$n \in \Omega_\mathfrak{p} \pod{\mathfrak{p}}$ for any
$\mathfrak{p} \in \mathcal{P}$.
Thus
\begin{equation}\label{eq2:sieve-to-large-sieve}
  \# \mathcal{S}(\mathcal{R}, \mathfrak{x},
  (\Omega_\mathfrak{p}))
  \leq \frac{D(  \mathcal{R},\mathfrak{x}
    ,Q)}{
    H((\Omega_\mathfrak{p}),Q)
  }.
\end{equation}

In this context, an additive large sieve inequality is by
definition a bound for $D(\mathcal{R},\mathfrak{x},Q)$.
The homomorphism $\mathbb{F}_\infty /
\mathfrak{x}^{-1} \mathfrak{d}^{-1} \ni \xi \mapsto
[\mathfrak{x} \ni n \mapsto e(\Tr \xi n)] \in
\mathfrak{x}^\wedge$ ($e(x) = e^{2 \pi i x}$) induces for
integral ideals $\mathfrak{q} ' | \mathfrak{q}$ the
compatible
isomorphisms
\begin{equation*}
  \begin{CD}         
    \mathfrak{q} '^{-1} \mathfrak{x}^{-1} \mathfrak{d}^{-1}
    / \mathfrak{x}^{-1} \mathfrak{d}^{-1} @>\cong>>
    (\mathfrak{x} / \mathfrak{q}' \mathfrak{x})^\wedge\\
    @VVV  @VVV \\
    \mathfrak{q} ^{-1} \mathfrak{x}^{-1} \mathfrak{d}^{-1}
    / \mathfrak{x}^{-1} \mathfrak{d}^{-1} @>\cong>>(\mathfrak{x} / \mathfrak{q}
    \mathfrak{x})^\wedge\\
  \end{CD}
\end{equation*}
by which we regard the family
$\sqcup \{ (\mathfrak{x}/\mathfrak{q} \mathfrak{x})^\wedge_\#
: \mathfrak{q} \in \mathcal{Q} \}$ of primitive
additive characters having (squarefree) conductor up to $Q$
(and supported on the primes of $\mathcal{P}$) as a subset
$\mathcal{F} :=
\mathcal{F}(\mathfrak{x},Q) \subset \mathbb{F} / \mathfrak{x}^{-1}
\mathfrak{d}^{-1} \subset \mathbb{F}_\infty /
\mathfrak{x}^{-1} \mathfrak{d}^{-1}$
of the family of all (finite order) additive characters
on $\mathfrak{x}$,
thus
\begin{equation}\label{eq2:E-as-exp-sum}
  E((a_n);\mathfrak{x},Q) = \sum_{\xi \in \mathcal{F}(\mathfrak{x},Q)}
  \left\lvert \sum_n a_n e(\Tr \xi n) \right\rvert^2.
\end{equation}
Write $D(\mathcal{R},\mathfrak{x},\mathcal{F})$
synonymously for $D(\mathcal{R},\mathfrak{x},Q)$.
The group $\mathfrak{o}_+^*$ acts on
$\mathbb{F}_\infty$ and
$\mathbb{F}_\infty / \mathfrak{x}^{-1} \mathfrak{d}^{-1}$
by multiplication,
stabilizing $\mathfrak{x}$ and $\mathcal{F}$.
The $\ell^\infty$ metric on $\mathbb{F}_\infty$
given by $d_{\mathbb{F}_\infty}(\xi,\eta)
= \max_i |\xi_i - \eta_i|$ induces
on $\mathbb{F}_\infty / \mathfrak{x}^{-1} \mathfrak{d}^{-1}$
by
the formula
$d(\xi,\eta) := \min_{n \in \mathfrak{x}^{-1} \mathfrak{d}^{-1}}
d_{\mathbb{F}_\infty}(\xi,\eta+n)$
a metric $d$ with respect to which we call
\[\delta := \delta(\mathcal{F}(\mathfrak{x},Q)) := \min_{\xi \neq \eta \in
  \mathcal{F}(\mathfrak{x},Q)} d(\xi, \eta)\] the smallest
spacing for the family $\mathcal{F}(\mathfrak{x},Q)$
and say that $\mathcal{F}(\mathfrak{x},Q)$
is \emph{$\delta(\mathcal{F}(\mathfrak{x},Q))$-spaced}.
\begin{lemma}\label{lem:well-spaced}
  $\delta(\mathcal{F}(\mathfrak{x},Q)) \geq (|\mathfrak{x}|
  \Delta_\mathbb{F} Q^2)^{-1/[\mathbb{F}:\mathbb{Q}]}$ (here $\Delta_\mathbb{F} = |\mathfrak{d}|$ is the
  discriminant of $\mathbb{F}$).
\end{lemma}
\begin{proof}
  Suppose that $\mathfrak{q}_1,\mathfrak{q}_2 \in \mathcal{Q}$,
  $\xi \in \mathfrak{q}_1^{-1} \mathfrak{x}^{-1}
  \mathfrak{d}^{-1}$,
  and
  $\eta \in \mathfrak{q}_2^{-1} \mathfrak{x}^{-1}
  \mathfrak{d}^{-1}$
  with $\xi - \eta \notin \mathfrak{x}^{-1} \mathfrak{d}^{-1}$.
  We must show,
  for any $n \in \mathfrak{x}^{-1} \mathfrak{d}^{-1}$,
  that $\zeta := \xi - \eta - n$
  satisfies $\max_i |\zeta_i| \geq (|\mathfrak{x}| \Delta_\mathbb{F}
  Q^2)^{-1/[\mathbb{F}:\mathbb{Q}]}$.
  Indeed, we have $0 \neq \zeta \in \mathfrak{q}_1^{-1} \mathfrak{q}_2^{-1}
  \mathfrak{x}^{-1} \mathfrak{d}^{-1}$,
  so that \[\prod |\xi_i - \eta_i|
  = |\xi - \eta|^{\mathbf{1}} \geq | \mathfrak{q}_1^{-1} \mathfrak{q}_2^{-1}
  \mathfrak{x}^{-1} \mathfrak{d}^{-1} | \geq \Delta_{\mathbb{F}}^{-1} |\mathfrak{x}|^{-1}
  Q^{-2}.\]
  Thus for some index $i$
  we have $|\zeta_i| \geq (|\mathfrak{x}| \Delta_\mathbb{F}
  Q^2)^{-1/[\mathbb{F}:\mathbb{Q}]}$, hence the claim.
\end{proof}

The duality principle for bilinear forms, which
asserts that a form and its transpose have the same norm,
implies that
$D(\mathcal{R},\mathfrak{x},\mathcal{F})$ is the smallest
non-negative real such that
\begin{equation}\label{eq2:6}
  \sum_{n \in \mathfrak{x} \cap \mathcal{R}}
  \left\lvert
    \sum_{\xi \in \mathcal{F}} b_\xi e(\Tr \xi n)
  \right\rvert^2
  \leq D(\mathcal{R},\mathfrak{x},\mathcal{F}) \|b\|_2^2
\end{equation}
for all $(b_\xi)_\xi : \mathcal{F} \to \mathbb{C}$, where
$\|b\|_2^2 = \sum |b_\xi|^2$.  Call a nonnegative-valued Schwarz
function $f \in \mathcal{S}( \mathbb{F}_\infty \to
\mathbb{R}_{\geq 0})$ \emph{$\mathcal{R}$-admissible} if it
satisfies $f|_{\mathcal{R}} \geq 1$,
and let $f$ be $\mathcal{R}$-admissible.
Opening the square in \eqref{eq2:6} and invoking the elementary
inequality
$|b_\xi \overline{b_\eta }| \leq \tfrac{1}{2}
(|b_\xi|^2 + |b_\eta|^2)$,
we find that
\begin{align*}
  \sum_{n \in \mathfrak{x} \cap \mathcal{R}}
  \left\lvert
    \sum_{\xi \in \mathcal{F}} b_\xi e(\Tr \xi n)
  \right\rvert^2
  &\leq
  \sum_{n \in \mathfrak{x}}
  f(n)
  \left\lvert
    \sum_{\xi \in \mathcal{F}} b_\xi e(\Tr \xi n)
  \right\rvert^2 \\
  &\leq
  \sup_{\xi \in \mathcal{F}}
  \sum_{\eta \in \mathcal{F}}
  \left\lvert \sum_{n \in \mathfrak{x}}
    f(n)
    e(\Tr n (\xi - \eta))
  \right\rvert \|b\|_2^2.
\end{align*}
Applying the Poisson summation formula,
which asserts in this context that
\begin{equation*}
  \begin{split}
    \sum_{n \in \mathfrak{x}}
    f(n)
    e(\Tr n (\xi - \eta))
    &=
    \vol(\mathbb{F}_\infty / \mathfrak{x})^{-1}
    \sum_{\mu \in \mathfrak{x}^{-1} \mathfrak{d}^{-1}}
    \hat{f}(\mu - \xi + \eta),
    \\
    &\quad
    \text{ with }
    \hat{f}(y) := \int_{\mathbb{F}_\infty} f(x)
    e(-x \cdot y) \, d y,
  \end{split}
\end{equation*}
we obtain
\begin{equation}\label{eq2:12}
  \begin{split}
    &D(\mathcal{R},\mathfrak{x},\mathcal{F})
    \leq \vol(\mathbb{F}_\infty/\mathfrak{o})^{-1} |\mathfrak{x}|^{-1}
    F(f;\mathfrak{x},\mathcal{F}),
    \\
    &\quad
    \text{ with }
    F(f;\mathfrak{x},\mathcal{F})
    :=
    \sup_{\xi \in \mathcal{F}}
    \sum_{\eta \in \mathcal{F}}
    \left\lvert \sum_{\mu \in \mathfrak{x}^{-1} \mathfrak{d}^{-1}}
      \hat{f}(\mu - \xi + \eta)
    \right\rvert.
  \end{split}
\end{equation}

\begin{lemma}\label{lem:basic-large-sieve}
  There exists a positive constant $c_2(\mathbb{F}) > 0$
  with the following
  property.
  For any rectangle $\mathcal{R} = \prod [a_i, b_i] =
  [a_1,b_1] \times \dotsb \times [a_d, b_d]$ whose volume
  $\vol(\mathcal{R}) = \prod |a_i - b_i|$ satisfies
  $\vol(\mathcal{R}) > c_2(\mathbb{F}) |\mathfrak{x}|$,
  there exists an $\mathcal{R}$-admissible function $f$
  such that
  \begin{equation}\label{eq2:15}
    F(f,\mathfrak{x},\mathcal{F}) \ll_{\mathbb{F}} \vol(\mathcal{R})
    + \delta^{-d}.
  \end{equation}
\end{lemma}
\begin{proof}
  For a unit $\eta \in \mathfrak{o}_+^*$ and an
  $\mathcal{R}$-admissible function $f$, define the $\eta
  \mathcal{R}$-admissible function $\eta f$ by the formula $\eta
  f(\eta x) = f(x)$.  Since $\mathfrak{x}$ and $\mathcal{F}$ are
  $\mathfrak{o}_+^*$-stable, we have $F(\eta
  f;\mathfrak{x},\mathcal{F}) = F(f;\mathfrak{x},\mathcal{F})$.
  Therefore we may assume that $\mathcal{R}$ is chosen so that
  $|a_i - b_i| \asymp |a_j - b_j|$ for all $i,j \in
  \{1,\dotsc,d\}$, where the implied constant depends only upon
  $\mathbb{F}$.
  Now the formula
  \[
  f(x)
  =\left( \frac{\pi^2}{8} \right)^d
  \prod_{i=1}^d
  \sinc^2 \left(
    \frac{x_i - \frac{a_i+b_i}{2}}{ 2 |a_i - b_i|}
  \right),
  \quad \sinc(x) = \frac{\sin(\pi x)}{\pi x}
  \]
  defines an $\mathcal{R}$-admissible function $f$
  whose Fourier transform is supported
  in the \emph{dual rectangle}
  \[
  \widehat{\mathcal{R}}
  = \prod [c_i ,d_i],
  \quad
  |c_i - d_i| = |a_i - b_i|^{-1}, \quad c_i = - d_i < 0 < d_i
  \]
  and satisfies
  $\|\hat{f}\|_\infty \leq (\pi^2/4)^d \prod |a_i - b_i|$.
  Since $|a_i - b_i| \asymp |a_j - b_j|$ for all $i,j$, there
  exists a constant $c_2(\mathbb{F}) > 0$, depending only upon $\mathbb{F}$,
  such that $\vol(\mathcal{R}) > c_2(\mathbb{F}) | \mathfrak{x}|$
  implies that $|a_i - b_i| > \tfrac{1}{2}
  \Delta_{\mathbb{F}}^{1/d} |\mathfrak{x}|^{1/d}$ for each $i$.
  If we assume now (as we may) that the latter assertion holds,
  then any translate of the dual rectangle
  $\widehat{\mathcal{R}}$ contains at most one element of the
  dual lattice $\mathfrak{x}^{-1} \mathfrak{d}^{-1}$, so that
  each sum over $\mu$ in \eqref{eq2:12} contains at most one
  nonzero term, thus
  \[
  \sum_{\eta \in \mathcal{F}}
  \left\lvert \sum_{\mu \in \mathfrak{x}^{-1} \mathfrak{d}^{-1}}
    \hat{f}(\mu - \xi + \eta)
  \right\rvert
  \leq \|\hat{f}\|_\infty \cdot \#  \Bigl\{ \eta \in \mathcal{F}
  \, : \,
  \mu - \xi + \eta \in \widehat{\mathcal{R}} + \mathfrak{x}^{-1} \mathfrak{d}^{-1} \Bigr\}.
  \]
  The above set is a $\delta$-spaced subset of
  $\widehat{\mathcal{R}} \pmod {\mathfrak{x}^{-1} \mathfrak{d}^{-1}}$; a cube-packing argument shows that any such set has
  cardinality at most $\prod (1 + \lfloor \delta^{-1} |c_i -
  d_i|\rfloor )$, so that
  \begin{equation}\label{eq2:13}
    F(f,\mathfrak{x},\mathcal{F})
    \leq
    \left( \frac{\pi^2}{4} \right)^d
    \prod_{i=1}^d |a_i - b_i| (1 + \lfloor \delta^{-1} |c_i -
    d_i|\rfloor )
    \ll  \prod_{i=1}^d ( |a_i - b_i| + \delta^{-1} ).
  \end{equation}
  Since $|a_i - b_i| \asymp |a_j - b_j|$,
  we obtain $F(f,\mathfrak{x},\mathcal{F}) \ll
  \vol(\mathcal{R}) + \delta^{-d}$, as desired.
\end{proof}

\begin{proof}[Proof of Proposition
  \ref{prop:large-sieve-as-a-sieve}]
  Take $c_2(\mathbb{F})$ as in Lemma
  \ref{lem:basic-large-sieve},
  and suppose that $X > c_2(\mathbb{F})$
  and $Q \geq 1$.
  Then $\vol(\mathcal{R}_{X,\mathfrak{z}}) > c_2(\mathbb{F})
  |\mathfrak{z}|$, so the hypotheses
  of Lemma \ref{lem:basic-large-sieve}
  are satisfied.
  The claimed bound \eqref{eq:large-sieve-as-sieve}
  follows immediately from \eqref{eq2:sieve-to-large-sieve},
  Lemma \ref{lem:well-spaced},
  equation \eqref{eq2:12}
  and Lemma \ref{lem:basic-large-sieve}.
\end{proof}

\section{Bounds for special
  functions}\label{sec:bounds-an-integral}
In this self-contained section we establish the technical lemmas
that were needed in the proof
of Lemma \ref{lem:j-bound-legit}.
First, recall \cite{MR0178117} that the Gauss hypergeometric function
$F = {}_2 F_1$
is defined for $\Re(c) > \Re(b) > 0$
and $|\arg(1-z)| < \pi$
by the integral
\begin{equation*}
  F
  \left( \efrac{a,b}{c} ;z
  \right)
  = \frac{\Gamma(c)}{\Gamma(b) \Gamma(c-b)}
  \int_0^1 \frac{t^{b-1} (1-t)^{c-b-1}}{(1-z t)^a}
\end{equation*}
where $\arg(1-z t) = 0$ for $z \in \mathbb{R}_{<0}$,
and for $|z| < 1$
and arbitrary $a,b,c$ by the series
\[
F
\left( \efrac{a,b}{c} ;z
\right)
=
\sum_{n=0}^\infty \frac{(a)_n (b)_n}{(c)_n}
\frac{z^n}{n!},
\quad (a)_n := a(a+1)(a+2)  \dotsb (a+n-1),
\]
which implies $F \left( \efrac{a,b}{c}; 0 \right) = 1$.
It satisfies
the differential equation
\[
x(1-x) y'' + (c-(a+b+1) x) y' - a b y = 0,
\quad
y(x) := {}_2 F_1 \left( \efrac{a,b}{c};x \right)
\]
for $x \notin \{1,\infty\}$.

\begin{lemma}\label{lem:technical-1}
  Let $x \in \mathbb{R}_{\geq 0}$, $\nu \in i \mathbb{R} \cup (-1/2,1/2)$
  and $s \in \mathbb{C}$ with $\Re(s) \geq 1/2$.
  Then
  \[
  \left\lvert
    {}_2 F_1
    \left( \efrac{\tfrac{1}{2} + \nu, \tfrac{1}{2} - \nu }{
        s}; - x \right)
  \right\rvert
  \leq 1.
  \]
\end{lemma}
\begin{proof}
  Fix $\nu$ and $s$ as above, and let
  \[
  F_s(x) = {}_2 F _1
  \left( \efrac{\tfrac{1}{2} - \nu, \tfrac{1}{2} + \nu}{s}; - x \right)
  \]
  for $x \in \mathbb{R}_{\geq 0}$.
  Then $F_s$ satisfies the differential equation
  \begin{equation}\label{eq:diffeq}
    x (1 + x ) F_s'' (x ) + (s + 2 x ) F_s ' (x )
    + \lambda F_s (x ) = 0
    \quad \text{ with } \lambda = \tfrac{1}{4} + r^2 > 0.
  \end{equation}
  Note that since $\{\overline{\tfrac{1}{2} + i r},
  \overline{\tfrac{1}{2} - i r}\}=\{\tfrac{1}{2} + i r,
  \tfrac{1}{2} - i r\}$,
  we have
  $\overline{F_s} = F_{\bar{s}}$
  and $\overline{F_s}' = F_{\bar{s}}'$.
  Let $f$ be a smooth function on $\mathbb{R}$
  and $H = |F_s|^2 + f |F_s'|^2$, so that
  \begin{equation}
    H' = F_s' F_{\bar{s} }
    + F _s F _{\bar{s} }'
    + f ' |F_s|^2
    + f (F_s'' F _{\bar{s} }'
    + F _s ' F _{\bar{s} } '').
  \end{equation}
  By the differential equation (\ref{eq:diffeq}), we have
  \[
  H' =
  \left( F _s ' F _{\bar{s}}
    + F _{s} F _{\bar{s} }'
  \right)
  \left( 1 - f \frac{\lambda }{x(1+x)}
  \right)
  + |F_s'|^2
  \left( f ' - f \frac{s + \bar{s} + 4 x }{ x (1 + x )} \right).
  \]
  Taking $f(x) = x(1+x)/\lambda$ gives
  \[
  H'(x) = \frac{1 - 2 \Re(s) - 2 x}{\lambda}
  |F_s'|^2 (x),
  \]
  so that $H'(x) \leq 0$ for $\Re(s) \geq 1/2$ and $x \geq 0$.
  Since $f(0) = 0$
  and $f(x) \geq 0$ for $x \geq 0$,
  we obtain
  \[
  |F_s|^2(x) \leq H(x) \leq H(0) = |F_s|^2(0) = 1,
  \]
  as desired.
\end{proof}
\begin{lemma}\label{lem:technical-2}
  Let $\nu \in i \mathbb{R} \cup (-\onehalf,\onehalf)$
  and $s \in \mathbb{C}$ with $\Re(s) \geq 1$.
  Then
  \[
  \left\lvert
    \frac{
      \Gamma(s + \nu )
      \Gamma(s- \nu)
    }{
      \Gamma(s+\onehalf) \Gamma(s-\onehalf)
    }
  \right\rvert \leq 1.
  \]
\end{lemma}
\begin{proof}
  Recall that Kummer's first formula asserts
  \begin{equation}\label{eq:kummer}
    \frac{
      \Gamma(s + \nu )
      \Gamma(s- \nu)
    }{
      \Gamma(s+\onehalf) \Gamma(s-\onehalf)
    }
    =
    \lim_{x \rightarrow 1^-}
    F_{\nu,s}(x),
    \quad
    F_{\nu,s}(x) :=
    F \left(
      \efrac{
        \nu + \tfrac{1}{2},\nu - \tfrac{1}{2}
      }{
        s + \nu 
      };
      x \right).
  \end{equation}
  Write $\sigma = \Re(s)$ and $u = \Re(\nu)$.
  Take
  $H = |F_{\nu,s}|^2 + f |F'_{\nu,s}|^2$
  for a smooth function $f$.
  The differential equation
  \begin{equation*}
    x(1-x) F_{\nu,s}''(x)
    + ( s + \nu - (2 \nu + 1) x) F_{\nu,s}'(x)
    + \lambda F_{\nu,s}(x)
    =0,
  \end{equation*}
  with $\lambda = \tfrac{1}{4} -  \nu^2 >0$,
  implies that
  \begin{equation*}
    \begin{split}
      H'
      &= \left(
        F_{\nu,s}' F_{\bar{\nu }, \bar{s} }
        + F_{\nu,s} F_{\bar{\nu }, \bar{s} }'
      \right)
      \left( 1 - f \frac{\lambda }{x(1-x)} \right) \\
      &\quad 
      + |F_{\nu,s}'|^2
      \left( f ' - f
        \frac{
          2 \sigma + 2 u - 2(2 u + 1) x
        }{
          x (1-x)
        }
      \right).
    \end{split}
  \end{equation*}
  Taking $f(x) = x(1-x)/\lambda$ gives
  \[
  H'(x) = \frac{
    1 - 2 \sigma
    - 2u(1-x) + 2u x
  }{
    \lambda
  }
  |F_{s,\nu}'|^2(x),
  \]
  so that
  our hypotheses $u \in (-\tfrac{1}{2},\tfrac{1}{2})$,
  $\Re(s) \geq 1$
  imply $H'(x) \leq 0$ for $0 \leq x < 1$.
  Since $f(0) = 0$ and
  $f(x) \geq 0$ for $0 \leq x \leq 1$,
  we obtain  $|F_{s,\nu}|^2(x) \leq H(x) \leq H(0) =
  |F_{s,\nu}|^2(0) = 1$ for $x \in (0,1)$,
  and the lemma follows from \eqref{eq:kummer}.
\end{proof}
\begin{remark}
  The proof of Lemma \ref{lem:technical-2} shows
  that the hypothesis $\Re(s) \geq 1$
  can be relaxed to $\Re(s) \geq \onehalf + \Re(\nu)$;
  we believe that Lemma \ref{lem:technical-2}
  holds in the larger range $\Re(s) \geq \onehalf$,
  $\nu \in i \mathbb{R} \cup (-\onehalf,\onehalf)$,
  but have not proven this.
  Such refinements are not necessary for our applications
  in the proof of Lemma \ref{lem:j-bound-legit}.
\end{remark}
\begin{remark}
  The bounds asserted by Lemmas \ref{lem:technical-1}
  and \ref{lem:technical-2} are sharp
  for several extremal
  cases of the parameters.
\end{remark}

% -----------------
\bibliography{refs}{}

\def\cprime{$'$} \def\cprime{$'$} \def\cprime{$'$}
\begin{thebibliography}{10}

\bibitem{MR2327298}
Don Blasius.
\newblock Hilbert modular forms and the {R}amanujan conjecture.
\newblock In {\em Noncommutative Geometry and Number Theory}, Aspects Math.,
  E37, pages 35--56. Vieweg, Wiesbaden, 2006.

\bibitem{2009arXiv0904.2429B}
Valentin Blomer and Gergely Harcos.
\newblock Twisted {$L$}-functions over number fields and {H}ilbert's eleventh
  problem.
\newblock {\em Geom. Funct. Anal.}, 20(1):1--52, 2010.

\bibitem{MR819779}
Y.~Colin~de Verdi{\`e}re.
\newblock Ergodicit\'e et fonctions propres du laplacien.
\newblock In {\em Bony-{S}j\"ostrand-{M}eyer seminar, 1984--1985}, pages Exp.\
  No. 13, 8. \'Ecole Polytech., Palaiseau, 1985.

\bibitem{Dav80}
Harold Davenport.
\newblock {\em Multiplicative Number Theory}, volume~74 of {\em Graduate Texts
  in Mathematics}.
\newblock Springer-Verlag, New York, second edition, 1980.
\newblock Revised by Hugh L. Montgomery.

\bibitem{MR533066}
Stephen Gelbart and Herv{\'e} Jacquet.
\newblock A relation between automorphic representations of {${\rm GL}(2)$} and
  {${\rm GL}(3)$}.
\newblock {\em Ann. Sci. \'Ecole Norm. Sup. (4)}, 11(4):471--542, 1978.

\bibitem{MR546600}
Stephen Gelbart and Herv{\'e} Jacquet.
\newblock Forms of {${\rm GL}(2)$} from the analytic point of view.
\newblock In {\em Automorphic Forms, Representations and {$L$}-functions
  ({P}roc. {S}ympos. {P}ure {M}ath., {O}regon {S}tate {U}niv., {C}orvallis,
  {O}re., 1977), {P}art 1}, Proc. Sympos. Pure Math., XXXIII, pages 213--251.
  Amer. Math. Soc., Providence, R.I., 1979.

\bibitem{GR}
I.~S. Gradshteyn and I.~M. Ryzhik.
\newblock {\em Table of integrals, series, and products}.
\newblock Elsevier/Academic Press, Amsterdam, seventh edition, 2007.
\newblock Translated from the Russian, Translation edited and with a preface by
  Alan Jeffrey and Daniel Zwillinger, With one CD-ROM (Windows, Macintosh and
  UNIX).

\bibitem{MR1836967}
George Greaves.
\newblock {\em Sieves in number theory}, volume~43 of {\em Ergebnisse der
  Mathematik und ihrer Grenzgebiete (3) [Results in Mathematics and Related
  Areas (3)]}.
\newblock Springer-Verlag, Berlin, 2001.

\bibitem{harris-kudla-1991}
Michael Harris and Stephen~S. Kudla.
\newblock The central critical value of a triple product {$L$}-function.
\newblock {\em Ann. of Math. (2)}, 133(3):605--672, 1991.

\bibitem{MR1265913}
Adolf Hildebrand and G{\'e}rald Tenenbaum.
\newblock Integers without large prime factors.
\newblock {\em J. Th\'eor. Nombres Bordeaux}, 5(2):411--484, 1993.

\bibitem{MR871630}
J{\"u}rgen~G. Hinz.
\newblock Methoden des grossen {S}iebes in algebraischen {Z}ahlk\"orpern.
\newblock {\em Manuscripta Math.}, 57(2):181--194, 1987.

\bibitem{HL94}
Jeffrey Hoffstein and Paul Lockhart.
\newblock Coefficients of {M}aass forms and the {S}iegel zero.
\newblock {\em Ann. of Math. (2)}, 140(1):161--181, 1994.
\newblock With an appendix by Dorian Goldfeld, Hoffstein and Daniel Lieman.

\bibitem{MR2484279}
Roman Holowinsky.
\newblock A sieve method for shifted convolution sums.
\newblock {\em Duke Math. J.}, 146(3):401--448, 2009.

\bibitem{holowinsky-2008}
Roman Holowinsky.
\newblock Sieving for mass equidistribution.
\newblock {\em Ann. of Math. (2)}, 172(2):1499--1516, 2010.

\bibitem{holowinsky-soundararajan-2008}
Roman Holowinsky and Kannan Soundararajan.
\newblock Mass equidistribution for {H}ecke eigenforms.
\newblock {\em Ann. of Math. (2)}, 172(2):1517--1528, 2010.

\bibitem{MR2449948}
Atsushi Ichino.
\newblock Trilinear forms and the central values of triple product
  {$L$}-functions.
\newblock {\em Duke Math. J.}, 145(2):281--307, 2008.

\bibitem{MR1942691}
Henryk Iwaniec.
\newblock {\em Spectral methods of automorphic forms}, volume~53 of {\em
  Graduate Studies in Mathematics}.
\newblock American Mathematical Society, Providence, RI, second edition, 2002.

\bibitem{IwaniecQUENotes}
Henryk Iwaniec.
\newblock Notes on the quantum unique ergodicity for holomorphic cusp forms,
  2010.

\bibitem{MR2061214}
Henryk Iwaniec and Emmanuel Kowalski.
\newblock {\em Analytic number theory}, volume~53 of {\em American Mathematical
  Society Colloquium Publications}.
\newblock American Mathematical Society, Providence, RI, 2004.

\bibitem{MR1826269}
Henryk Iwaniec and Peter Sarnak.
\newblock Perspectives on the analytic theory of {$L$}-functions.
\newblock {\em Geom. Funct. Anal.}, (Special Volume, Part II):705--741, 2000.
\newblock GAFA 2000 (Tel Aviv, 1999).

\bibitem{Ja72}
Herv{\'e} Jacquet.
\newblock {\em Automorphic forms on {${\rm GL}(2)$}. {P}art {II}}.
\newblock Lecture Notes in Mathematics, Vol. 278. Springer-Verlag, Berlin,
  1972.

\bibitem{MR0401654}
Herv{\'e} Jacquet and R.~P. Langlands.
\newblock {\em Automorphic forms on {${\rm GL}(2)$}}.
\newblock Lecture Notes in Mathematics, Vol. 114. Springer-Verlag, Berlin,
  1970.

\bibitem{MR2426239}
E.~Kowalski.
\newblock {\em The large sieve and its applications}, volume 175 of {\em
  Cambridge Tracts in Mathematics}.
\newblock Cambridge University Press, Cambridge, 2008.
\newblock Arithmetic geometry, random walks and discrete groups.

\bibitem{MR1078363}
Uwe Krause.
\newblock Absch\"atzungen f\"ur die {F}unktion {$\Psi_K(x,y)$} in algebraischen
  {Z}ahlk\"orpern.
\newblock {\em Manuscripta Math.}, 69(3):319--331, 1990.

\bibitem{MR540902}
J.-P. Labesse and R.~P. Langlands.
\newblock {$L$}-indistinguishability for {${\rm SL}(2)$}.
\newblock {\em Canad. J. Math.}, 31(4):726--785, 1979.

\bibitem{MR2195133}
Elon Lindenstrauss.
\newblock Invariant measures and arithmetic quantum unique ergodicity.
\newblock {\em Ann. of Math. (2)}, 163(1):165--219, 2006.

\bibitem{MR1361757}
Wenzhi Luo and Peter Sarnak.
\newblock Quantum ergodicity of eigenfunctions on {${\rm PSL}_2(\bold
  Z)\backslash \bold H^2$}.
\newblock {\em Inst. Hautes \'Etudes Sci. Publ. Math.}, (81):207--237, 1995.

\bibitem{luo-sarnak-mass}
Wenzhi Luo and Peter Sarnak.
\newblock Mass equidistribution for {H}ecke eigenforms.
\newblock {\em Comm. Pure Appl. Math.}, 56(7):874--891, 2003.
\newblock Dedicated to the memory of J{\"u}rgen K. Moser.

\bibitem{2010arXiv1006.3305M}
S.~{Marshall}.
\newblock {Mass Equidistribution for Automorphic Forms of Cohomological Type on
  GL\_2}.
\newblock {\em ArXiv e-prints}, June 2010.

\bibitem{MR0224585}
H.~L. Montgomery.
\newblock A note on the large sieve.
\newblock {\em J. London Math. Soc.}, 43:93--98, 1968.

\bibitem{MR1197420}
Mohan Nair.
\newblock Multiplicative functions of polynomial values in short intervals.
\newblock {\em Acta Arith.}, 62(3):257--269, 1992.

\bibitem{MR1618321}
Mohan Nair and G{\'e}rald Tenenbaum.
\newblock Short sums of certain arithmetic functions.
\newblock {\em Acta Math.}, 180(1):119--144, 1998.

\bibitem{PDN-HQUE-LEVEL}
Paul Nelson.
\newblock Equidistribution of cusp forms in the level aspect.
\newblock {\em Duke Math. J.}, to appear.

\bibitem{MR1266075}
Ze{\'e}v Rudnick and Peter Sarnak.
\newblock The behaviour of eigenstates of arithmetic hyperbolic manifolds.
\newblock {\em Comm. Math. Phys.}, 161(1):195--213, 1994.

\bibitem{MR1321639}
Peter Sarnak.
\newblock Arithmetic quantum chaos.
\newblock In {\em The {S}chur lectures (1992) ({T}el {A}viv)}, volume~8 of {\em
  Israel Math. Conf. Proc.}, pages 183--236. Bar-Ilan Univ., Ramat Gan, 1995.

\bibitem{sarnak-progress-que}
Peter Sarnak.
\newblock Recent {P}rogress on {Q}{U}{E}.
\newblock \url{http://www.math.princeton.edu/sarnak/SarnakQUE.pdf}, 2009.

\bibitem{MR0272745}
Werner Schaal.
\newblock On the large sieve method in algebraic number fields.
\newblock {\em J. Number Theory}, 2:249--270, 1970.

\bibitem{MR0402834}
A.~I. Schnirelman.
\newblock Ergodic properties of eigenfunctions.
\newblock {\em Uspehi Mat. Nauk}, 29(6(180)):181--182, 1974.

\bibitem{Sh75}
Goro Shimura.
\newblock On the holomorphy of certain {D}irichlet series.
\newblock {\em Proc. London Math. Soc. (3)}, 31(1):79--98, 1975.

\bibitem{MR507462}
Goro Shimura.
\newblock The special values of the zeta functions associated with {H}ilbert
  modular forms.
\newblock {\em Duke Math. J.}, 45(3):637--679, 1978.

\bibitem{MR2346281}
Lior Silberman and Akshay Venkatesh.
\newblock On quantum unique ergodicity for locally symmetric spaces.
\newblock {\em Geom. Funct. Anal.}, 17(3):960--998, 2007.

\bibitem{aws2010sound}
K.~{Soundararajan}.
\newblock Arizona winter school lecture notes on quantum unique ergodicity and
  number theory.
\newblock \url{http://math.arizona.edu/~swc/aws/10/2010SoundararajanNotes.pdf},
  2010.

\bibitem{2009arXiv0901.4060S}
Kannan Soundararajan.
\newblock Quantum unique ergodicity for {${\rm SL}_2(\mathbb{Z}) \backslash
  \mathbb{H} $}.
\newblock {\em Ann. of Math. (2)}, 172(2):1529--1538, 2010.

\bibitem{MR2680497}
Kannan Soundararajan.
\newblock Weak subconvexity for central values of {$L$}-functions.
\newblock {\em Ann. of Math. (2)}, 172(2):1469--1498, 2010.

\bibitem{MR882550}
E.~C. Titchmarsh.
\newblock {\em The theory of the {R}iemann zeta-function}.
\newblock The Clarendon Press Oxford University Press, New York, second
  edition, 1986.
\newblock Edited and with a preface by D. R. Heath-Brown.

\bibitem{venkatesh-2005}
Akshay Venkatesh.
\newblock Sparse equidistribution problems, period bounds and subconvexity.
\newblock {\em Ann. of Math. (2)}, 172(2):989--1094, 2010.

\bibitem{watson-2008}
Thomas~C. Watson.
\newblock Rankin triple products and quantum chaos.
\newblock \url{arXiv.org:0810.0425}, 2008.

\bibitem{MR1610485}
Andr{\'e} Weil.
\newblock S\'eries de {D}irichlet et fonctions automorphes.
\newblock In {\em S\'eminaire {B}ourbaki, {V}ol.\ 10}, pages Exp.\ No.\ 346,
  547--552. Soc. Math. France, Paris, 1995.

\bibitem{MR0178117}
E.~T. Whittaker and G.~N. Watson.
\newblock {\em A course of modern analysis. {A}n introduction to the general
  theory of infinite processes and of analytic functions: with an account of
  the principal transcendental functions}.
\newblock Fourth edition. Reprinted. Cambridge University Press, New York,
  1962.

\bibitem{MR916129}
Steven Zelditch.
\newblock Uniform distribution of eigenfunctions on compact hyperbolic
  surfaces.
\newblock {\em Duke Math. J.}, 55(4):919--941, 1987.

\end{thebibliography}
\bibliographystyle{plain}
% -----------------
\end{document}